\documentclass[a4paper,11pt]{article}

\usepackage[utf8]{inputenc}
\usepackage{amsmath,amsfonts}
\usepackage{amsthm} 
\usepackage{amssymb}
\usepackage{multirow}
\usepackage{latexsym}
\usepackage{color}
\usepackage{graphicx} 
\usepackage{caption}
\usepackage{subcaption}
\usepackage[numbers]{natbib}
\usepackage[colorlinks=true,linkcolor=blue,urlcolor=black,citecolor=blue]{hyperref}
\usepackage{bbm}

\newtheorem{theorem}{Theorem}[section] 
\newtheorem{lemma}[theorem]{Lemma}
\newtheorem{definition}[theorem]{Definition} 
\newtheorem{proposition}[theorem]{Proposition}

\newtheorem{example}[theorem]{Example}
\newtheorem{corollary}[theorem]{Corollary}
\newtheorem{remark}[theorem]{Remark}

\newcommand{\argmin}[1]{\underset{#1}{\operatorname{argmin}}}

\newcommand{\IR}{\mathbb{R}}
\newcommand{\IE}{\mathbb{E}}
\newcommand{\IP}{\mathbb{P}}
\newcommand{\Ind}{\mathbbm{1}}
\newcommand{\IN}{\mathbb{N}}
\newcommand{\IZ}{\mathbb{Z}}
\newcommand{\IQ}{\mathbb{Q}}
\newcommand{\Kappa}{\mathcal{K}}
\newcommand{\IT}{\mathbbm{T}}

\renewcommand{\tilde}{\widetilde}
\renewcommand{\epsilon}{\varepsilon}
\renewcommand{\phi}{\varphi}
\renewcommand{\bar}{\overline}

\allowdisplaybreaks

 \numberwithin{equation}{section}
\begin{document}

\begin{titlepage}
\title{Correlation bounds, mixing and $m$-dependence under random time-varying network distances with an application to Cox-Processes}

\author{Alexander Krei{\ss} \\
KU Leuven \\
ORSTAT KU Leuven \\
Naamsestraat 69 \\
3000 Leuven\\
Belgium\\
alexander.kreiss@kuleuven.be}

\date{\today}

\maketitle

\begin{abstract}
We will consider multivariate stochastic processes indexed either by vertices or pairs of vertices of a dynamic network. Under a dynamic network we understand a network with a fixed vertex set and an edge set which changes randomly over time. We will assume that the spatial dependence-structure of the processes conditional on the network behaves in the following way: Close vertices (or pairs of vertices) are dependent, while we assume that the dependence decreases conditionally on that the distance in the network increases. We make this intuition mathematically precise by considering three concepts based on correlation, $\beta$-mixing with time-varying $\beta$-coefficients and conditional independence. These concepts allow proving weak-dependence results, e.g. an exponential inequality, which might be of independent interest. In order to demonstrate the use of these concepts in an application we study the asymptotics (for growing networks) of a goodness of fit test in a dynamic interaction network model based on a Cox-type model for counting processes. This model is then applied to bike-sharing data.
\end{abstract}
\end{titlepage}

\setcounter{page}{1} 
\setcounter{section}{0}

\section{Introduction}
Data indexed by vertices or pairs of vertices of networks has become popular in recent times (see e.g. \citet{BNS18}, \citet{DDLY16}, \citet{B08} for recent applications) when also the availability of such data sets increases, see e.g. websites of SNAP (Stanford University) or KONECT (University of Koblenz-Landau). In order to illustrate the contribution of this paper, we consider the following example of network data: Suppose we observe on the interval $[0,T]$ a network with vertex set $V_n:=\{1,...,n\}$ and random, dynamic adjacency matrix $C_n$, i.e., for all $i,j\in V_n, i\neq j$ we have a stochastic process $C_{n,ij}:[0,T]\to\{0,1\}$, where $C_{n,ij}(t)=1$ means that $i$ and $j$ are connected by an edge at time $t$. We consider the vertices to be actors who can interact with each other whenever they are connected by an edge. As an example, the actors could be users of a social media platform and an interaction is sending a private message. Then, we observe for all pairs $(i,j)$ a counting process $N_{n,ij}$ which counts the interactions between $i$ and $j$ and a multivariate process $X_{n,ij}$ which carries information about $i$ and $j$, e.g. the number of interactions in the past or information about mutually shared interests. In this situation it is intuitive that the tuples $(N_{n,ij},X_{n,ij},C_{n,ij})$ cannot be modelled as independent. Instead, we adopt the following heuristic: For any two pairs $(i,j),(k,l)$ and time points $t\in[0,T]$, we suppose that on a small neighbourhood $U_t$ around $t$ the dependence is influenced by the closeness of $(i,j)$ and $(k,l)$ where \emph{closeness} is to be understood relative to the random adjacency matrix $C_n(t)$ (we will be more precise later):
\begin{enumerate}
\item The processes $(N_{n,ij},X_{n,ij},C_{n,ij})$ and $(N_{n,kl},X_{n,kl},C_{n,kl})$ restricted to $U_t$ are dependent conditional on $(i,j)$ and $(k,l)$ being close at time $t$ in $C_n(t)$.
\item The processes $(N_{n,ij},X_{n,ij},C_{n,ij})$ and $(N_{n,kl},X_{n,kl},C_{n,kl})$ restricted to $U_t$ are almost independent conditional on $(i,j)$ and $(k,l)$ being far apart in $C_n(t)$.
\end{enumerate}
Note that we implicitly allow that the dependence structure may randomly change over time by allowing that the adjacency matrix $C_n$ is a random function of time. In order to use this intuition we will have to assume that in large networks a given pair $(i,j)$ is at a given time $t$ most likely not close to too many other pairs.

The main contribution of this work is to make the above heuristic precise. We do this by formalizing time-varying spatial dependence concepts for multivariate processes indexed by pairs of vertices in a network. More precisely, we extend the concept of \emph{asymptotic uncorrelation} which was used in the previous work \citet{KMP19} to \emph{momentary-$m$-dependence} and \emph{$\beta$-mixing on networks}. In contrast to asymptotic uncorrelation the two new concepts take the random network structure into account. Thus, time-varying \emph{$\beta$-Mixing} coefficients will allow us to prove exponential inequalities. Moreover, by using \emph{Momentary-$m$-Dependence} we can adapt a technique from \citet{MN07} to networks in order to handle predictability problems related to counting processes which were also noted e.g. in \citet{NLB98}. In order to illustrate the necessity of these concepts in a specific situation we study a goodness of fit test in a counting process based network model. In the derivation of the asymptotic distribution of the test statistic under the null we require uniform control of the whole estimated parameter function. This cannot be handled by simple second order conditions (e.g. asymptotic uncorrelation as in \citet{KMP19}). However, more generally, these concepts can be used to provide interpretable conditions to transfer inference results for multivariate counting processes from the iid case (cf. \citet{AG83}) to the case of random network data. Thus, Section \ref{sec:main} is of independent interest for the literature on multivariate (counting) processes on networks (see e.g. \citet{B08,PW13,FSSCB16,VLMP17} for such models).

For an overview of statistical methods in network analysis we refer to the books \citet{K09}, \citet{J08} and \citet{N10}. The general situation that the relational structure of network data is different from other dependent-data scenarios like time-series and spatial data analysis is for example mentioned in the beginning of Chapter 2 in \citet{K17}. Classical results about dependent processes can e.g. be found in the books \citet{D94} and \citet{R17}. Some models which are used in the context of mixing, particularly in econometrics, are mentioned in \citet{ND04}. Further asymptotic normality results based on local dependence can be found in \citet{CS04} for random fields and in \citet{SH15,KMS19} for random, non-dynamic networks. Other approaches for modelling dependence in random networks are for example the extension of the concept of stationarity to random (but not time-changing) networks as in \citet{V18} or Bayesian networks (cf. \citet{P88} and \citet{FMR98} for an extension to time series and \citet{GHEGM08} for an application). In the application we will study a Cox-type proportional hazard model (cf. \citet{ABGK93,MS06,C72,AG82}). Generalisations and variations of such models have been studied outside a network context e.g. in \citet{NLB98,NL95,LMNvK11,LNvdG03}. For network interactions, parameter estimation in models of this type has been considered e.g. in \citet{B08} and \citet{PW13}. The goodness of fit test which we will consider is based on an $L^2$-type test statistic as in \citet{HM93}. Particular references for smooth testing in survival analysis are \citet{MvK18} (use the same type of test statistic) and \citet{KB03} (use a local likelihood approach), however, not within a network context.

After collecting some notation and briefly reviewing asymptotic uncorrelation in Sections \ref{sec:notation} and \ref{subsec:asymptotic_uncorrelation}, we introduce in the main part of Section \ref{sec:main} the concepts of momentary-$m$-dependence (Section \ref{subsec:m_dependence}) and $\beta$-mixing on networks (Section \ref{subsec:mixing_networks}). In the end of Section \ref{sec:main}, in Section \ref{subsec:example}, we provide examples of data generating processes and motivate why they exhibit these properties. In Section \ref{sec:model} we apply the methods established in Section \ref{sec:main} to a goodness of fit test problem. The whole procedure is then illustrated on bike sharing data in Section \ref{sec:bikes}. In the Appendix (Section \ref{sec:appendix}) we collect missing proofs from the main part of the paper as well as some additional technical results. R-code which is used for the bike-data illustration is available on \href{https://github.com/akreiss/Estimate-Event-Network}{https://github.com/akreiss/Estimate-Event-Network}.

\section{Describing Dependence on Dynamic Networks}
\label{sec:main}
In this section we introduce the dependence concepts. For ease of exposition we stick to a model for relational event data which was also used e.g. in \citet{B08,PW13,KMP19}. In Section \ref{sec:notation} we will briefly review the basics of this framework and in Sections \ref{subsec:asymptotic_uncorrelation}-\ref{subsec:mixing_networks} we introduce the dependence concepts. Section \ref{subsec:example} provides examples. We finish in Section \ref{subsec:vertex_processes} with a short note on processes indexed by vertices.

\subsection{Preliminaries and Notation}
\label{sec:notation}

We use the following notation from graph theory. We consider directed (undirected), dynamic, random networks $G_{n,t}=(V_n,E_{n,t})$ for $n\in\IN$ and $t\in[0,T]$ which are comprised of a fixed vertex set $V_n:=\{1,...,n\}$ and a random dynamic edge set $E_{n,t}\subseteq L_n$, where $L_n:=\{(i,j):i,j\in V_n, i\neq j\}$ is the set of all directed (undirected) pairs (we exclude loops). The adjacency matrix of $G_{n,t}$ at time $t$ is denoted by $(C_{n,ij}(t))_{i,j\in V_n}$. Furthermore we denote by $r_n:=|L_n|$ the number of directed (undirected) pairs of vertices.

We study stochastic processes $(N_{n,ij},X_{n,ij},C_{n,ij})_{(i,j)\in L_n}$ with the following properties.

\textbf{Measurability}\\
\emph{For all $n\in\IN$ there is a filtration $(\mathcal{F}_t^n)_{t\in[0,T]}$ such that for all $i,j\in V_n$ the processes $N_{n,ij}:[0,T]\to\IN$ are counting processes which are adapted to $\mathcal{F}_t^n$ and such that for all $i,j\in V_n$ the processes $X_{n,ij}:[0,T]\to\IR^q$ and $C_{n,ij}:[0,T]\to\{0,1\}$ are predictable with respect to $\mathcal{F}_t^n$. Moreover, the intensity function of $N_{n,ij}$ is given by $\lambda_{n,ij}(t)=C_{n,ij}(t)\lambda(t,X_{n,ij}(t))$ for some link function $\lambda:[0,T]\times\IR^q\to[0,\infty)$.}

\begin{remark}
\begin{enumerate}
\item We choose here to index the processes with pairs of vertices. Similarly one could also index the processes with the vertices directly (cf. Section \ref{subsec:vertex_processes}). We choose pairs here because we imagine observations to be driven by the interplay of two actors.
\item The processes $C_{n,ij}$ are indicators which indicate whether the pair $(i,j)$ is currently \emph{active} at time $t$ ($C_{n,ij}(t)=1$) or not ($C_{n,ij}(t)=0$). Our understanding is that, for a given $(i,j)\in L_n$, the process $N_{n,ij}$ is only interesting (i.e. useful for inference) on the set $\{t\in[0,T]:\,C_{n,ij}(t)=1\}$.
\item We are not too much concerned about the existence of a filtration as required in the above definition. One possibility would be to assume that $C_{n,ij}$ and $X_{n,ij}$ are continuous from the left and let $\mathcal{F}_t^n:=\sigma(N_{n,ij}(s),X_{n,ij}(s),C_{n,ij}(s):\,(i,j)\in L_n,s\leq t)$ be the filtration generated by the processes $(N_{n,ij},X_{n,ij},C_{n,ij})$ for all $(i,j)\in L_n$.
\end{enumerate}
\end{remark}

It is intuitively reasonable to assume that relabelling the vertices is not going to change the distribution of the processes. Hence, we will assume that
\begin{equation}
\label{eq:exchangeability}
(N_{n,ij},X_{n,ij},C_{n,ij})_{(i,j)\in L_n}\sim (N_{n,\sigma(i)\sigma(j)},X_{n,\sigma(i)\sigma(j)},C_{n,\sigma(i)\sigma(j)})_{(i,j)\in L_n}
\end{equation}
holds for all permutations $\sigma:V_n\to V_n$ and all $n\in\IN$. This property is also called joint exchangeability of arrays (cf. \citet{OR15}). Note that for any two different pairs $(i,j),(k,l)\in L_n$ we can construct a permutation $\sigma$ with $\sigma(i)=k$ and $\sigma(j)=l$ (recall that we consider networks without loops). Hence, $(N_{n,ij},X_{n,ij},C_{n,ij})$ and $(N_{n,kl},X_{n,kl},C_{n,kl})$ are identically distributed. This notion allows for the concept of hubs but every vertex has a priory the same potential of becoming a hub. Moreover, we assume that all possible interactions between vertices are observed. Therefore we do not have to worry about edge sampling issues as mentioned in \citet{CD18}. Note lastly that the permutations $\sigma$ from above are deterministic and thus in particular they may not be chosen dependent on the actual observed network structure. This will be similar in Section \ref{subsec:asymptotic_uncorrelation} when discussing asymptotic uncorrelation. Thus, these two properties do not take the actually observed network into account. However, in Sections \ref{subsec:m_dependence} and \ref{subsec:mixing_networks}, when introducing Momentary-$m$-Dependence and $\beta$-Mixing we condition on the observed network. Thus, in these concepts we consider the observations after making choices which are strongly dependent on the observed network.

One way of taking the network structure into account is through random distance functions on networks. A \emph{random distance function on networks} is a collection of stochastic processes $d_t^n:L_n\times L_n\to[0,\infty]$ such that for any $t\in[0,T]$ and $n\in\IN$, $d_t^n$ is almost surely a metric. For later reference we collect all the above in a single definition.
\begin{definition}
\label{defin:struc_interaction_net}
The processes $(N_{n,ij},X_{n,ij},C_{n,ij})_{(i,j)\in L_n}$ on $[0,T]$ together with the random distance function on networks $d_t^n$ is called \emph{structured interaction network process} if for all $n\in\IN$
\begin{enumerate}
\item the above mentioned measurability properties hold,
\item the network process is exchangeable, i.e., \eqref{eq:exchangeability} holds for all permutations $\sigma:V_n\to V_n$,
\item $t\mapsto d_t^n((i,j),(k,l))$ is predictable with respect to $\mathcal{F}_t^n$ for all $(i,j),(k,l)\in L_n$.
\end{enumerate}
In this case $p_n(t):=\IP(C_{n,ij}(t)=1)$ is well defined.
\end{definition}

\begin{remark}
\label{rem:disunknown}
\begin{itemize}
\item Later the interpretation of $d_t^n$ will be as follows: The distance $d_t^n((i,j),(k,l))$ reflects how strongly the pairs $(i,j)$ and $(k,l)$ are related conditionally on the observed network (short distance means strong relation, large distance means weak relation).
\item From a modelling perspective, we emphasize that the distance function $d^n_t$ does not need to be known to the researcher. It is only necessary that it exists.
\item In order to allow sparsity we explicitly allow that $p_n(t)\to0$ for $n\to\infty$.
\end{itemize}
\end{remark}

\subsection{Asymptotic Uncorrelation}
\label{subsec:asymptotic_uncorrelation}
We briefly review a stationarity type result which was similarly used for static networks in \citet{V18} and for dynamic networks in \citet{KMP19}. In this subsection we restrict to undirected networks (see also the paragraph below Corollary \ref{cor:au}). Consider square integrable random variables $(Z_{n,ij})_{(i,j)\in L_n}$ with the following property: \emph{The $Z_{n,ij}$ are identically distributed. For $(i,j),(k,l)\in L_n$ let $\kappa((i,j),(k,l)):=|\{i,j\}\cap \{k,l\}|\in\{0,1,2\}$ denote the number of common vertices of $(i,j)$ and $(k,l)$. For $(i,j),(k,l),(i',j'),(k',l')\in L_n$ the pairs $(Z_{n,ij},Z_{n,kl})$ and $(Z_{n,i'j'},Z_{n,k'l'})$ are identically distributed if $\kappa((i,j),(k,l))=\kappa((i',j'),(k',l'))$.}

We will later consider $Z_{n,ij}:=\phi(N_{n,ij},X_{n,ij},C_{n,ij})$ where $\phi$ takes real values and $\IE(Z_{n,ij}^2)<\infty$. The exchangeability assumption in Definition \ref{defin:struc_interaction_net} guarantees the above property which in turn yields the following corollary:
\begin{corollary}
\label{cor:au}
For all $n\in \IN$, let $(Z_{n,ij})_{(i,j)\in L_n}$ be as above. Recall that $r_n=\frac{n(n-1)}{2}$ is the number of undirected pairs. Then, for pairwise different vertices $v_1,v_2,v_3,v_4\in V_n$,
\begin{align*}
&\mathcal{V}_n:=\textrm{Var}\left(\frac{1}{r_n}\sum_{(i,j)\in L_n}Z_{n,ij}\right) \\
=&\frac{1}{r_n^2}\sum_{(i,j)\in L_n}\textrm{Var}(Z_{n,ij})+\frac{1}{r_n^2}\underset{\kappa((i,j),(k,l))=1}{\sum_{(i,j),(k,l)\in L_n}}\textrm{Cov}(Z_{n,ij},Z_{n,kl})+\frac{1}{r_n^2}\underset{\kappa((i,j),(k,l))=0}{\sum_{(i,j),(k,l)\in L_n}}\textrm{Cov}(Z_{n,ij},Z_{n,kl}) \\
=&r_n^{-1}\cdot \textrm{Var}(Z_{n,v_1v_2})+O\left(r_n^{-\frac{1}{2}}\right)\textrm{Cov}(Z_{n,v_1v_2},Z_{n,v_2v_3})+O\left(1\right)\textrm{Cov}(Z_{n,v_1v_2},Z_{n,v_3v_4}).
\end{align*}
\end{corollary}
For the proof of this corollary we just need to think about the number of terms in each sum. It is a combinatorial exercise to find that their sizes are of the order $r_n, r_n^{\frac{3}{2}}$ and $r_n^2$ respectively. In order to have $\mathcal{V}_n\to0$ we hence require that $\textrm{Cov}(Z_{n,v_1v_2},Z_{n,v_2v_3})=o\left(r_n^{\frac{1}{2}}\right)$ and $\textrm{Cov}(Z_{n,v_1v_2},Z_{n,v_3v_4})=o(1)$. We will call assumptions of this type \emph{asymptotic uncorrelation assumptions}. This result naturally extends to directed networks by splitting the sum in all possible patterns which two directed pairs can have.

\subsection{Momentarily m-Dependent Networks}
\label{subsec:m_dependence}
We introduce momentary-$m$-dependence for processes $\mathcal{N}_{n,ij}:=(N_{n,ij},X_{n,ij},C_{n,ij})$. The aim is to mathematically formulate and use the following intuition: The processes $\mathcal{N}_{n,ij}$ and $\mathcal{N}_{n,kl}$ are dependent for any fixed choice of $(i,j)$ and $(k,l)$. However, if we choose $(i,j)$ and $(k,l)$ such that they are far apart in the observed network (in terms of $d_{t_0}^n$ for some time $t_0\in[0,T]$), then for real world actors it is likely that it takes some time for the pair $(i,j)$ to receive knowledge of interactions between $(k,l)$ and to process them before reacting by casting interactions themselves. Therefore, we assume: Provided that we know the network structure at time $t_0$ and that we know the past of all processes up to time $t_0$ and provided that we know that for two pairs $(i,j),(k,l)\in L_n$ the distance $d_{t_0}^n((i,j),(k,l))$ is large, then the processes $\mathcal{N}_{n,ij}(t)_{t\in[t_0,t_0+6\Delta]}$ and $\mathcal{N}_{n,kl}(t)_{t\in[t_0,t_0+6\Delta]}$ are conditionally independent given all information up to time $t_0$ for some $\Delta>0$ (the factor six is chosen for later convenience). We illustrate this in Figure \ref{fig:info_flow}: The horizontal axis is time and the vertical axis is distance. The two lines correspond to two pairs $(i,j)$ and $(k,l)$ and the vertical distance between these two lines represents the distance between the pairs $(i,j)$ and $(k,l)$. Dots on the lines indicate events between the respective vertices. The two gray rectangles in the future (next to the line at $t_0$) stand for the information of the processes of $(i,j)$ on the interval $[t_0,t_0+6\Delta]$ and the processes of $(k,l)$ on the interval $[t_0,t_0+6\Delta]$. We suppose that these two are conditionally independent given the information up to time $t_0$. So there is no direct \emph{information flow} between these two areas. However, they are not unconditionally independent because we can infer from the gray rectangle in the future of $(k,l)$ on its past when $(i,j)$ and $(k,l)$ were possibly close, such that we can infer on the past of $(i,j)$ which is informative about its future. But if we already know the past, then additional knowledge of the future of $(k,l)$ is independent of the future of $(i,j)$.

\begin{figure}
\centering
\includegraphics[width=0.6\textwidth]{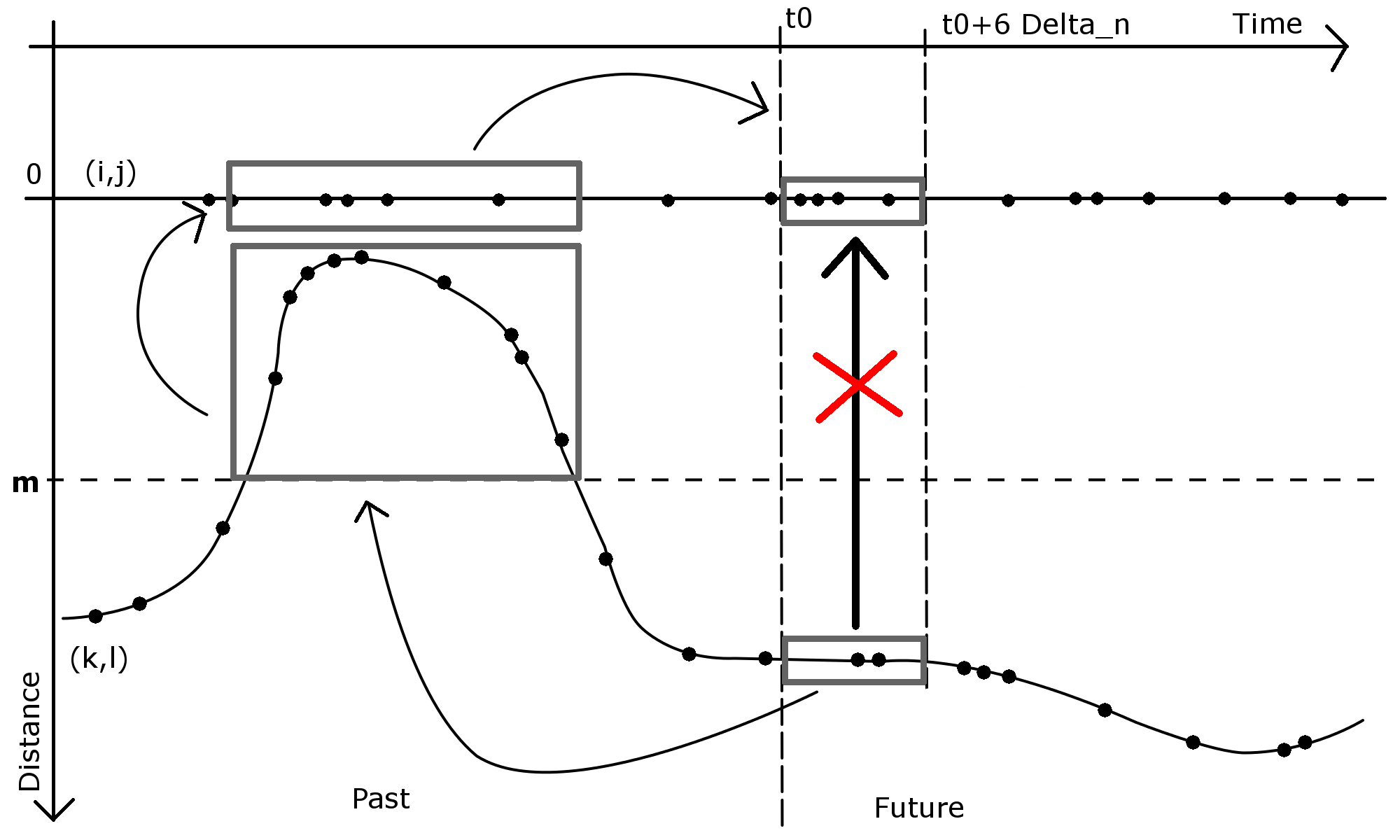}
\caption{Graphical illustration of the information flow in a momentarily-$m$-dependent network.}
\label{fig:info_flow}
\end{figure}

In mathematical terms this can be described as follows. For a set $J\subseteq L_n$ of pairs, let $d_s^n((i,j),J):=\min\{d_s^n((i,j),(k,l)):\,(k,l)\in J\}$ be the distance of $(i,j)$ to $J$ at time $s$.
\begin{definition}
\label{defin:m-dependence}
A structured interaction network process $(\mathcal{N}_{n,ij})_{(i,j)\in L_n}$ with filtration $(\mathcal{F}_t^n)_{t\in[0,T]}$ and distance $d^n_t$ is said to be \emph{momentarily-$m$-dependent} for $m\in[0,\infty)$, if
\begin{align*}
&\forall n\in\IN,\,\exists\delta_n>0,\,\forall t_0\in[0,T],\,\forall J\subseteq L_n: \textrm{Given }\mathcal{F}_{t_0}^n\\
&\quad\quad (\mathcal{N}_{n,ij}(t))_{(i,j)\in J,\,t\in[t_0,t_0+6\delta_n]}\textrm{ is cond. independent of } \\
&\quad\quad\sigma\Big(\mathcal{N}_{n,ij}(r)\cdot\Ind(d_s^n((i,j),J)\geq m):s\leq t_0,r\leq s+6\delta_n,\, (i,j)\in L_n\Big).
\end{align*}
\end{definition}

In order to work with momentary-$m$-dependent networks, we introduce two augmentations of the filtration $\mathcal{F}_t^n$. Generally, when extending filtrations, we have more predictable processes and fewer martingales. In the following definition we introduce two extensions of $\mathcal{F}_t^n$, one of which is the exact right trade-off: Certain processes become predictable with respect to the extension while certain other processes remain martingales (see Lemma \ref{lem:martingale}).

For two $\sigma$-fields $\mathcal{A}$ and $\mathcal{B}$ we denote by $\mathcal{A}\lor\mathcal{B}$ the $\sigma$-field which is generated by the union of $\mathcal{A}$ and $\mathcal{B}$.
\begin{definition}
\label{defin:leave-one-out}
Let $(\mathcal{N}_{n,ij})_{(i,j)\in L_n}$ be a structured interaction network with filtration $(\mathcal{F}_t^n)_{t\in[0,T]}$ and distance $d^n_t$. For a subset $J\subseteq L_n$ define
\begin{align*}
\mathcal{F}^{n,J,m}_t&:=\mathcal{F}^n_t\lor\sigma\Big(\mathcal{N}_{n,ij}(r)\Ind(d_s^n((i,j),J)\geq m):s\leq t, r\leq s+6\delta_n, (i,j)\in L_n\Big).
\end{align*}
We call $\mathcal{F}^{n,J,m}_t$ the \emph{long-sighted leave-$J$-out filtration}. In contrast, the \emph{short-sighted leave-$J$-out filtration} $\tilde{\mathcal{F}}^{n,J,m}_{I,t}$ for $I\subseteq J$ is defined by
\begin{align*}
\tilde{\mathcal{F}}^{n,J,m}_{I,t}:=&\sigma(X_{n,i}(\tau):\,(i,j)\in I,\tau\leq t) \\
&\lor\sigma(\mathcal{N}_{n,ij}(r)\Ind(d_s^n((i,j),J)\geq m):s\leq\max(0,t-4\delta_n), r\leq s+6\delta_n, (i,j)\in L_n).
\end{align*}
Denote further for any pair $(i,j)\in L_n$, $F_{(i,j)}(t):=\{(k,l)\in L_n:\,d_t^n((i,j),(k,l))\geq m\}$. Functions which are predictable with respect to $\tilde{\mathcal{F}}_{I,t}^{n,J,m}$ will be called of \emph{leave-$m$-out} type.
\end{definition}
It holds that $\mathcal{F}_t^{n,J,m}\supseteq\tilde{\mathcal{F}}_t^{n,J,m}$. We can now make the earlier mentioned property of the long-sighted leave-$J$-out filtration precise: The counting processes stay counting processes and in particular their martingales are still martingales. The proof of the result is a direct consequence of the definition and can be found in Appendix \ref{subsec:fproofs}.
\begin{lemma}
\label{lem:martingale}
We consider a structured momentarily-$m$-dependent interaction network. For $J\subseteq L_n$, the processes $\left(N_{n,ij}(t)\right)_{(i,j)\in J}$ form a multivariate counting process with respective intensity functions $(\lambda_{n,ij}(t))_{(i,j)\in J}$ with respect to $\mathcal{F}_t^{n,J,m}$. This means in particular that the counting process martingales $M_{n,ij}(t):=N_{n,ij}(t)-\int_0^t\lambda_{n,ij}(s)ds$ remain martingales with respect to $\mathcal{F}_t^{n,J,m}$.
\end{lemma}

\begin{remark}
Throughout we will use the notion of Stieltjes and Itô Integration interchangeably when possible. In particular, when $\phi$ is predictable, we will understand $\int_0^T\phi(t)dM_{n,ij}(t)$ as an Itô Integral and use its martingale properties (since $M_{n,ij}$ is a martingale). If $\phi$ is not predictable we can understand the same integral as Stieltjes Integral which is defined path-wise (no predictability required) but is itself no martingale (in contrast to the Itô Integral). 
\end{remark}

We can use momentary-$m$-dependence in order to extend a technique which \citet{MN07} applied to iid observations in a non-network context: Approximate non-predictable integrands by processes which are predictable with respect to a larger filtration. The proof of the following result is along the lines of \citet{MN07} and is given in Appendix \ref{subsec:fproofs}.
\begin{proposition}
\label{prop:single_non-pred}
Let $\mathcal{N}_{n,ij}$ be momentarily-$m$-dependent and let $\phi_{n,ij}:[0,T]\to\IR$ for $n\in\IN, (i,j)\in L_n$ be random functions (not necessarily predictable). Let furthermore $\tilde{\phi}_{n,ij}^J:[0,T]\to\IR$ for $(i,j)\in J\subseteq L_n$ and $|J|\leq 2$ be of leave-$m$-out type, i.e., predictable with respect to $\tilde{\mathcal{F}}_{(i,j),t}^{n,J,m}$. Then, we have ($\lambda_{n,ij}$ and $M_{n,ij}$ mean the same as in Lemma \ref{lem:martingale})
\begin{align*}
&\IE\left(\left(\sum_{(i,j)\in L_n}\int_0^T\phi_{n,ij}(t)dM_{n,ij}(t)\right)^2\right)\leq\sum_{(i,j)\in L_n}\int_0^T\IE\left(\tilde{\phi}_{n,ij}^{ij}(t)^2C_{n,ij}(t)\lambda_{n,ij}(t)\right)dt \\
&+2\sum_{(i,j),(k,l)\in L_n}\IE\left(\int_0^T\tilde{\phi}^{ij,kl}_{n,ij}(t)dM_{n,ij}(t)\int_0^T\left(\phi_{n,kl}(r)-\tilde{\phi}_{n,kl}^{ij,kl}(r)\right)dM_{n,kl}(r)\right) \\
&+\sum_{(i,j),(k,l)\in L_n}\IE\left(\int_0^T\left(\phi_{n,ij}(t)-\tilde{\phi}_{n,ij}^{ij,kl}(t)\right)dM_{n,ij}(t)\int_0^T\left(\phi_{n,kl}(r)-\tilde{\phi}_{n,kl}^{ij,kl}(r)\right)dM_{n,kl}(r)\right).
\end{align*}
\end{proposition}

For our purposes we have to extend this technique even further: When studying kernel estimators we encounter integrals of the type
\begin{equation}
\label{eq:ip}
\frac{1}{r_n}\underset{(i,j)\neq (k,l)}{\sum_{(i,j),(k,l)\in L_n}}\int_0^T\int_{t-2h}^{t-}\phi_{n,ij,kl}(t,r)dM_{n,kl}(r)dM_{n,ij}(t),
\end{equation}
where $h\to0$ is a bandwidth and
$$\phi_{n,ij,kl}:=\underset{(v_1,v_2)\neq(i,j),(k,l)}{\sum_{(v_1,v_2)\in L_n}}\int_t^{r+2h}f(X_{n,v_1v_2}(\tau))d\tau$$
for some real-valued function $f$. Hence, for a given $t$ the integrand in \eqref{eq:ip} is non-predictable. However, under momentary-$m$-dependence, by removing the correct terms from the sum in the definition of $\phi_{n,ij,kl}(t,r)$, we obtain processes which are partially predictable with respect to $\tilde{\mathcal{F}}^{n,I,m}_{\{(i,j),(k,l)\},t}$:
\begin{definition}
\label{def:preliminary_partially_predictable}
Let $\phi$ be a real-valued stochastic process defined on $[0,T]^2$. $\phi$ is called \emph{partially-predictable} with respect to a filtration $\mathcal{G}_t$ if for any filtration $\mathcal{H}_t\supseteq\mathcal{G}_t$ and any process $X$ which is adapted to $\mathcal{H}_t$ the process
$$t\mapsto\int_0^{t-}\phi(t,r)dX(r)$$
is predictable with respect to $\mathcal{H}_t$. Note that $\phi(r,t)=g(r)h(t)f(r,t)$ with $g$ being adapted, $h$ being predictable (both with respect to $\mathcal{G}_t$) and $f$ deterministic has this property.
\end{definition}
Since the martingales $M_{n,ij}$ and $M_{n,kl}$ remain martingales under the correct long-sighted filtration, we can now use stochastic integral properties. For ease of notation, we use the convention $u_r:=(i_r,j_r)\in L_n$ for $r=1,...,4$ and we write sets without curly brackets, e.g. instead of $\{u_1,u_2,u_3\}\subseteq L_n$, we simply write $u_1u_2u_3\subseteq L_n$. The proof of the following result is given in Appendix \ref{sec:proofs}.

\begin{theorem}
\label{thm:easy_non-pred}
Let $(N_{n,ij},X_{n,ij},C_{n,ij})_{(i,j)\in L_n}$ be a structured interaction network with filtration $(\mathcal{F}_t^n)_{t\in[0,T]}$ and distance $d^n_t$. Let $\phi_{n,u_1u_2}:[0,T]\times[0,T]\to\IR$ for $u_1,u_2\in L_n$ be random functions (possibly not predictable with respect to $\mathcal{F}_t^n$). It holds that
\begin{equation}
\label{eq:aim}
\frac{1}{r_n}\underset{u_1\neq u_2}{\sum_{u_1,u_2\in L_n}}\int_0^T\int_{t-2\delta_n}^{t-}\phi_{n,u_1u_2}(t,r)dM_{n,u_2}(r)dM_{n,u_1}(t)\overset{\IP}{\to}0,
\end{equation}
for $n\to\infty$, if
\begin{enumerate}
\item the processes $(N_{n,ij},X_{n,ij},C_{n,ij})_{(i,j)\in L_n}$ are momentarily-$m$-dependent and
\item there exist random functions $\tilde{\phi}_{n,u_1u_2}^{I}(t,r)$ for all $u_1u_2\subseteq J\subseteq L_n$ with $|J|\leq 4$ which are partially predictable with respect to $\tilde{\mathcal{F}}_{u_1u_2,t}^{n,J,m}$, respectively, and such that (the symbol $\neg$ means negation)
\end{enumerate}
\begin{align}
&\frac{1}{r_n}\underset{u_1\neq u_2}{\sum_{u_1,u_2\in L_n}}\int_0^T\int_{t-2\delta_n}^{t-}\left(\phi_{n,u_1u_2}(t,r)-\tilde{\phi}_{n,u_1u_2}^{u_1u_2}(t,r)\right)dM_{n,u_2}(r)dM_{n,u_1}(t)=o_P(1), \label{eq:cond1} \\
&\IE\Bigg(\frac{1}{r_n^2}\underset{u_1\neq u_2,u_3\neq u_4}{\sum_{u_1,u_2,u_3,u_4\in L_n}}\int_0^T\int_{t-2\delta_n}^{t-}\left(\tilde{\phi}_{n,u_1u_2}^{u_1u_2}(t,r)-\tilde{\phi}_{n,u_1u_2}^{u_1u_2u_3u_4}\right)(t,r)dM_{n,u_2}(r)dM_{n,u_1}(t) \nonumber \\
&\quad\quad\times\int_0^T\int_{t-2\delta_n}^{t-}\left(\tilde{\phi}_{n,u_3u_4}^{u_3u_4}(t,r)-\tilde{\phi}_{n,u_3u_4}^{u_1u_2u_3u_4}(t,r)\right)dM_{n,u_4}(r)dM_{n,u_3}(t)\Bigg)=o(1), \label{eq:cond2a} \\
&\frac{2}{r_n^2}\underset{u_1\neq u_2,u_3\neq u_4}{\sum_{u_1,u_2,u_3,u_4\in L_n}}\IE\Bigg[\int_0^T\int_{t-2\delta_n}^{t-}\left(\tilde{\phi}_{n,u_1u_2}^{u_1u_2}(t,r)-\tilde{\phi}_{n,u_1u_2}^{u_1u_2u_3u_4}(t,r)\right)dM_{n,u_2}(r) \nonumber \\
&\times\int_t^{t+2\delta_n}\int_{\xi-2\delta_n}^{\xi-}\tilde{\phi}_{n,u_3u_4}^{u_1u_2u_3u_4}(\xi,\rho)dM_{n,u_4}(\rho)dM_{n,u_3}(\xi)\Ind(\neg u_3,u_4\in F_{u_1}(t-2\delta_n))dM_{n,u_1}(t)\Bigg]=o(1), \label{eq:cond3} \\
&\frac{1}{r_n^2}\underset{u_1\neq u_2}{\sum_{u_1,u_2\in L_n}}\int_0^T\int_{t-2\delta_n}^{t-}\IE\Bigg[\tilde{\phi}_{n,u_1u_2}^{u_1u_2}(t,r)^2C_{n,u_1}(t)\lambda_{n,u_1}(t)C_{n,u_2}(r)\lambda_{n,u_2}(r) \nonumber \\
&\quad\quad\times\Ind(u_2\in F_{u_1}(t-2\delta_n))\Bigg]drdt =o(1), \label{eq:cond4} \\
&\frac{1}{r_n^2}\underset{u_1\neq u_2}{\sum_{u_1,u_2\in L_n}}\underset{u_4\neq u_2}{\sum_{u_4\in L_n}}\int_0^T\IE\Bigg[\int_{t-2\delta_n}^{t-}\tilde{\phi}_{n,u_1u_2}^{u_1u_2u_4}(t,r)dM_{n,u_2}(r)\int_{t-2\delta_n}^{t-}\tilde{\phi}_{n,u_1u_4}^{u_1u_2u_4}(t,r')dM_{n,u_4}(r') \nonumber \\
&\quad\quad\quad\times C_{n,u_1}(t)\lambda_{n,u_1}(t)\Ind(\neg u_2,u_4\in F_{u_1}(t-2\delta_n))\Bigg]dt =o(1). \label{eq:cond5}
\end{align}
\end{theorem}

\subsection{Mixing Networks}
\label{subsec:mixing_networks}
In this section our interest lies in proving a Bernstein type exponential inequality e.g. for the following average
$$\frac{1}{r_np_n(t)}\sum_{(i,j)\in L_n}\left(Z_{n,ij}-\IE\left(Z_{n,ij}\right)\right),$$
where we will later have $Z_{n,ij}=\phi(N_{n,ij},X_{n,ij},C_{n,ij})$ for a real-valued function $\phi$. However, the following results do not depend on this specific functional form as long as the $Z_{n,ij}$ have the exchangeability property 2 in Definition \ref{defin:struc_interaction_net}. The difficulties here are two-fold: We usually have that $Z_{n,ij}=0$ when $C_{n,ij}(t)=0$ and hence the number of terms in the sum is random and, secondly, the terms are dependent. We argued in the discussion of Figure \ref{fig:info_flow} that unconditional independence is not a good assumption. However, it is reasonable to assume that, conditionally on the network, far apart actors influence each other very weakly. We include this aspect in the model by imposing mixing assumptions with time-varying mixing coefficients. These mixing assumptions will be used in the proofs by applying the grouping technique for mixing random variables (cf. \citet{R17,D94,V97}): The idea is to group the random variables $Z_{n,ij}$ in blocks which have large distances between each other in the observed network. To this end, we define a partitioning of a network as follows (the existence of such partitions will be discussed in Section \ref{subsec:example}).

\begin{definition}
\label{def:graph_partition}
Let $\Delta>0$, $t\in[0,T]$, $\Kappa,n,m\in\IN$ and $k\in\{1,...,\Kappa\}$. We call the random sets $G^t(k,m,\Delta)\subseteq L_n$ a $\Delta$-partition of the network at time $t$ (note that we omit $n$ in the notation) if
\begin{enumerate}
\item $(k,m)\neq(k',m')\,\Rightarrow\, G^t(k,m,\Delta)\cap G^t(k',m',\Delta)=\emptyset$,
\item For $k\in\{1,...,\Kappa\}$ and $m\neq m'$: $(i,j)\in G^t(k,m,\Delta),\,(k,l)\in G^t(k,m',\Delta)\,\Rightarrow\,d_t^n((i,j),(k,l))\geq\Delta$.
\end{enumerate}
\end{definition}

Intuitively speaking, the sets $G^t(k,m,\Delta)$ form random groups where two different groups of the same type $k$ are far apart in the random network. For the following definition we use the notion of $\beta$-mixing coefficients. For any two $\sigma$-fields $\mathcal{A}$ and $\mathcal{B}$ denote the $\beta$-mixing coefficient by (cf. e.g. \citet{R17})
$$\beta(\mathcal{A},\mathcal{B}):=\sup_{C\in\mathcal{A}\otimes\mathcal{B}}\left|\IP_{\mathcal{A}\otimes\mathcal{B}}(C)-\left(\IP_{\mathcal{A}}\otimes\IP_{\mathcal{B}}\right)(C)\right|,$$
where $\IP_{\mathcal{A}\otimes\mathcal{B}}$ and $\IP_{\mathcal{A}}\otimes\IP_{\mathcal{B}}$ denote measures on $\mathcal{A}\otimes\mathcal{B}$ for which for all sets $A\times B\in\mathcal{A}\otimes\mathcal{B}$
\begin{align*}
\IP_{\mathcal{A}\otimes\mathcal{B}}(A\times B)&=\IP(A\cap B) \textrm{ and } \\
\left(\IP_{\mathcal{A}}\otimes\IP_{\mathcal{B}}\right)(A\times B)&=\IP(A)\IP(B).
\end{align*}

For two random variables $X,Y$ we denote $\beta(X,Y):=\beta(\sigma(X),\sigma(Y))$ where $\sigma(X)$ and $\sigma(Y)$ denote the $\sigma$-fields generated by $X$ and $Y$ respectively.

\begin{definition}
\label{defin:network-beta}
Let $(Z_{n,ij})_{(i,j)\in L_n}$ be a sequence of random variables, let $\Delta>0$ and let $G^t(k,m,\Delta)$ be a $\Delta$-partition of the network as in Definition \ref{def:graph_partition}. For every time point $t$ and every pair $(i,j)\in L_n$, we define
\begin{align*}
I_{n,ij}^{k,m,t}(\Delta)&:=\Ind((i,j)\in G^t(k,m,\Delta)),
\end{align*}
the indicator function which checks if $(i,j)$ belongs to the $m$-th block of type $k$ at time $t$. Group the $Z_{n,ij}$ based on the partition $G^t(k,m,\Delta)$, i.e.,
$$U_{k,m}^{n,t}(\Delta):=\sum_{(i,j)\in L_n}\left[Z_{n,ij}\cdot I_{n,ij}^{k,m,t}(\Delta)-\IE\left(Z_{n,ij}\cdot I_{n,ij}^{k,m,t}(\Delta)\right)\right].$$
Then we define the $\beta$-Mixing coefficient which depends on the graph partitioning $G^t(k,m,\Delta)$ (which we do not indicate in the notation) via:
$$\beta_t(\Delta):=\underset{k\in\{1,...,\Kappa\}}{\max_{M\in\IN}}\beta\left(\left[U_{k,m}^{n,t}(\Delta)\right]_{m\leq M-1},U_{k,M}^{n,t}(\Delta)\right).$$
\end{definition}

\begin{remark}
In most (but not all) situations we have additionally to the properties of Definition \ref{def:graph_partition} that
\begin{equation}
\label{eq:SpecificCoverCondition}
\bigcup_{k=1}^{\mathcal{K}}\bigcup_mG^t(k,m,\Delta)=E_{n,t},
\end{equation}
where $E_{n,t}$ is the random edge set of the network. In case where $Z_{n,ij}=0$ for $C_{n,ij}(t)=0$, i.e., if $Z_{n,ij}C_{n,ij}(t)=Z_{n,ij}$ all relevant pairs $(i,j)\in L_n$ are covered by the partition and it holds that
\begin{equation}
\label{eq:GeneralCoverCondition}
\sum_{(i,j)\in L_n}Z_{n,ij}=\sum_{k=1}^{\mathcal{K}}\sum_{m=1}^{\infty}\sum_{(i,j)\in L_n}Z_{n,ij}I_{n,ij}^{k,m,t}.
\end{equation}
In general, for our results to hold, we do not have to require \eqref{eq:SpecificCoverCondition}. It will be sufficient to assume that \eqref{eq:GeneralCoverCondition} holds.
\end{remark}

In applications, the random variables $Z_{n,ij}$ will depend on a time point $t_0\in[0,T]$. So it will be the case that for $t$ close to $t_0$ the $\beta$-Mixing coefficients at time $t$ will be small while they might be large for $t$ far away from $t_0$. The following result is the main result of this section (inspired by \citet{D94}). The proof is deferred to Section \ref{sec:proofs}.

\begin{lemma}
\label{lem:network_exp}
Let $(Z_{n,ij})_{(i,j)\in L_n}$ be an array of random variables which fulfils \eqref{eq:GeneralCoverCondition} for a given $t\in[0,T]$ and let $\Delta_n>0, \mathcal{K}_n\in\IN$. Suppose that for all $n\in\IN$ there exist $\Delta_n$-partitions with $\mathcal{K}_n$ block types and numbers $E_{k,m}^{n,t}, E_k^{n,t}>0$ for $k=1,...,\mathcal{K}_n$ and $m=1,...,r_n$ as well as $\sigma^2,c_1,c_2,c_3>0$ such that (cf. notation from Definition \ref{defin:network-beta})
\begin{enumerate}
\item $\forall n\in\IN,\rho\in\IN\setminus \{0,1\},k\in\{1,...,\Kappa\},m\in\{1,...,r_n\}:$
$$\IE(|U_{k,m}^{n,t}(\Delta_n)|^{\rho})\leq\frac{\rho!}{2} E_{k,m}^{n,t}\sigma^2\cdot\left(E_k^{n,t}c_1\right)^{\rho-2},$$
\item  for $|E|_{n,t}:=\sum_{k=1}^{\mathcal{K}}\sum_{m=1}^{r_n}E_{k,m}^{n,t}$, it holds for all $n\in\IN$ and all $k=1,...,\mathcal{K}_n$
\begin{align*}
\frac{1}{|E|_{n,t}}\sum_{m=1}^{r_n}E_{k,m}^{n,t}\geq c_2\textrm{ and }E_k^{n,t}\leq c_3\sqrt{\frac{|E|_{n,t}}{\log |E|_{n,t}}}.&&
\end{align*}
\end{enumerate}
Then, for any $x>0$ and all $n\in\IN$,
\begin{align}
&\IP\left(\frac{1}{|E|_{n,t}}\sum_{(i,j)\in L_n}(Z_{n,ij}-\IE(Z_{n,ij}))\geq x\cdot\sqrt{\frac{\log |E|_{n,t}}{|E|_{n,t}}}\right) \nonumber \\
\leq&\Kappa_n|E|_{n,t}^{-\frac{c_2\cdot x^2}{2(\sigma^2+c_1c_3 x)}}+\beta_{t}(\Delta_n)\cdot \Kappa_n r_n. \label{eq:state1}
\end{align}
\end{lemma}

Note that $E_{k,m}^{n,t}$ can be understood as the expected size of group $m$ of type $k$ and that $E_k^{n,t}$ can be understood as the largest expected group size of groups of type $k$. Then, $|E|_{n,t}$ can be understood as the expected number of edges in the network at time $t$, i.e., $|E|_{n,t}\approx r_np_n(t)$. Now the first part of condition 2 in Lemma \ref{lem:network_exp} translates to assuming that the expected fraction of edges contained in all groups of type $k$ is non-negligible. The second part means that the largest single group cannot be too large. The first condition, the moment condition, will be discussed in the next lemma. We will also show that it suffices to assume the existence of a suitable partition as above with high probability. To this end we introduce an indicator function $\Gamma_n^t$ which ensures that we can partition the network suitably. Conditionally on that, we can use the previous mixing results. In order to obtain an unconditional result we need to assume that $\Gamma_n^t=1$ sufficiently often. This is reflected in the unusual condition on $x$. The proof of the following result can be found in Appendix \ref{sec:proofs}. In addition we will show in the Appendix (Lemma \ref{lem:exponential_inequality}) a different result which provides an exponential inequality for (unbounded) martingales and also avoids the moment condition.

\begin{lemma}
\label{lem:sufficient_mixing}
Let $(Z_{n,ij})_{(i,j)\in L_n}$ be random variables bounded by $M>0$ and let $(C_{n,ij}(t))_{i,j\in V_n}$ be the adjacency matrix of a random, undirected network at time $t\in[0,T]$. Let furthermore $\Delta_n>0$ and $I_{n,ij}^{k,m,t}$ be the indicators of a $\Delta_n$-partition with $\mathcal{K}_n$ group types which fulfils \eqref{eq:GeneralCoverCondition} (cf. also Definition \ref{defin:network-beta}). Suppose there are numbers $E_{k,m}^{n,t}>0$, $c_3>0$ such that for $|E|_{n,t}:=\sum_{k,m}E_{k,m}^{n,t}$ and
\begin{align*}
S_k(t)&:=\max_{m=1,...,r_n}\sum_{(i,j)\in L_n}I_{n,ij}^{k,m,t}, \quad \Gamma_n^t:=\Ind\left(\forall k\in\{1,...,\mathcal{K}\}:\frac{S_k^2\cdot\log(|E|_{n,t})}{|E|_{n,t}}\leq c_3^2\right), \\
E_k^{n,t}&:=c_3\sqrt{\frac{|E|_{n,t}}{\log(|E|_{n,t})}}, \\
Y_{n,ij}&:=\left(Z_{n,ij}(t)-\IE\left(Z_{n,ij}(t)\Big|\Gamma_n^t=1\right)\right)\cdot \Gamma_n^t
\end{align*}
there are constants $c_2>0$ and $C>0$ such that $\forall k=1,...,\mathcal{K}:\frac{1}{|E|_{n,t}}\sum_{m=1}^{r_n}E_{k,m}^{n,t}\geq c_2$ and that for pairwise different vertices $i,j,k,l\in V_n$
\begin{align*}
&\frac{r_n}{E_{k,m}^{n,t}}\textrm{Var}\left(Y_{n,ij}I_{n,ij}^{k,m,t}\right)\leq CM^2, \quad\frac{r_n^{\frac{3}{2}}}{E_{k,m}^{n,t}}\textrm{Cov}\left(Y_{n,ij}I_{n,ij}^{k,m,t},Y_{n,jk}I_{n,jk}^{k,m,t}\right)\leq CM^2, \\
&\frac{r_n^2}{E_{k,m}^{n,t}}\textrm{Cov}\left(Y_{n,ij}I_{n,ij}^{k,m,t},Y_{n,kl}I_{n,kl}^{k,m,t}\right)\leq CM^2.
\end{align*}
Let $\beta_t(\Delta_n)$ denote the $\beta$-mixing coefficients with respect to $(Y_{n,ij})_{(i,j)\in L_n}$ as in Definition \ref{defin:network-beta}. Then, for $x>M\IP(\Gamma_n^t=0)r_n\left(\log(|E|_{n,t})\cdot|E|_{n,t}\right)^{-\frac{1}{2}}$
\begin{align*}
&\IP\left(\frac{1}{|E|_{n,t}}\sum_{(i,j)\in L_n}\left[Z_{n,ij}(t)-\IE(Z_{n,ij}(t))\right]\geq 3x\cdot\sqrt{\frac{\log |E|_{n,t}}{|E|_{n,t}}}\right) \\
\leq&2\mathcal{K}_n\cdot(|E|_{n,t})^{-\frac{c_2\cdot x^2}{2(3CM^2+4Mc_3x)}}+\beta_t(\Delta_n)\cdot\mathcal{K}r_n+\IP(\Gamma_n^t=0)
\end{align*}
\end{lemma}

\subsection{Examples}
\label{subsec:example}
In the following we discuss the previous concepts on examples.

\subsubsection{On $\Delta$-Partitions}
\label{subsubsec:partitions}
For the exponential inequality to hold, we do not need to know the specific partition in practice: Knowledge of existence is sufficient. Nevertheless, we discuss under which circumstances a $\Delta$-partition with the properties of Lemma \ref{lem:sufficient_mixing} can be expected to exist. Let $G_n$ be a given random, dynamic network with adjacency matrix $C_n$. As a distance function we take the graph distance, i.e., $d(ij,kl)$ denotes the length of the shortest path between the pairs $(i,j)$ and $(k,l)$ if $C_{n,ij}=C_{n,kl}=1$ (e.g. $d(ij,kl)=1$ if $(i,j)$ and $(k,l)$ are adjacent, $d(ij,kl)=2$ if there is one edge between $(ij)$ and $(k,l)$ and so forth). Otherwise or if there is no path, we set $d(ij,kl)=\infty$. We begin by supposing that for a given point in time $t$ the network $G_n(t)$ is a two-dimensional grid. In that situation we consider a chess-board like partitioning $(G^t(k,m,\Delta))_{k,m}$ of the edges as illustrated in Figure \ref{fig:grid_partition} where the sides and corners of the blocks lie exactly on the vertices. For edges which lie on the sides of the blocks we take the convention that the bottom and left side belong to the respective block. Each block (square) is of side length $\Delta$ and each block is assigned one of four types. In Figure \ref{fig:grid_partition} all blocks of the same type $k\in\{1,...,4\}$ have been assigned the same number. It is clear that the distance between two points taken from two different blocks of the same type $k$ is at least $\Delta$. We assign numbers $\{1,2,3,...\}$ to all blocks of the same type such that we can speak of the $m$-th block of type $k$. Later we will choose $\Delta_n\approx a\log n$ and $E_{k,m}^{n,t}$ in Lemma \ref{lem:sufficient_mixing} will be the expected size of the $m$-th block of type $k$. Above the $\Delta$-partition is made such that all edges are contained in exactly one set $G^t(k,m,\Delta)$ and we obtain as a consequence that by definition $|E|_{n,t}$ will equal the expected number of edges. Moreover, the blocks $G^t(k,m,\Delta)$ all have identical size $2\Delta_n^2\approx2(a\log n)^2$. Thus $\frac{1}{|E|_{n,t}}\sum_{m=1}^{r_n}E_{k,m}^{n,t}=1/4$. Also $S_k(t)=2(a\log n)^2$ and hence $\Gamma_n^t=0$ for $c_3$ chosen large enough. These considerations can be directly transferred to higher dimensional grids. Hence, for networks which form a grid of any dimension, the assumption of the existence of a sequence of $\Delta_n$-partitions as required in Lemma \ref{lem:sufficient_mixing} with $\Delta_n=O(\log n)$ is proven. In consideration of this, we conclude that for a network which roughly looks like a grid, the above construction still yields a valid partition.

\begin{figure}
\center
\includegraphics[width=0.5\textwidth]{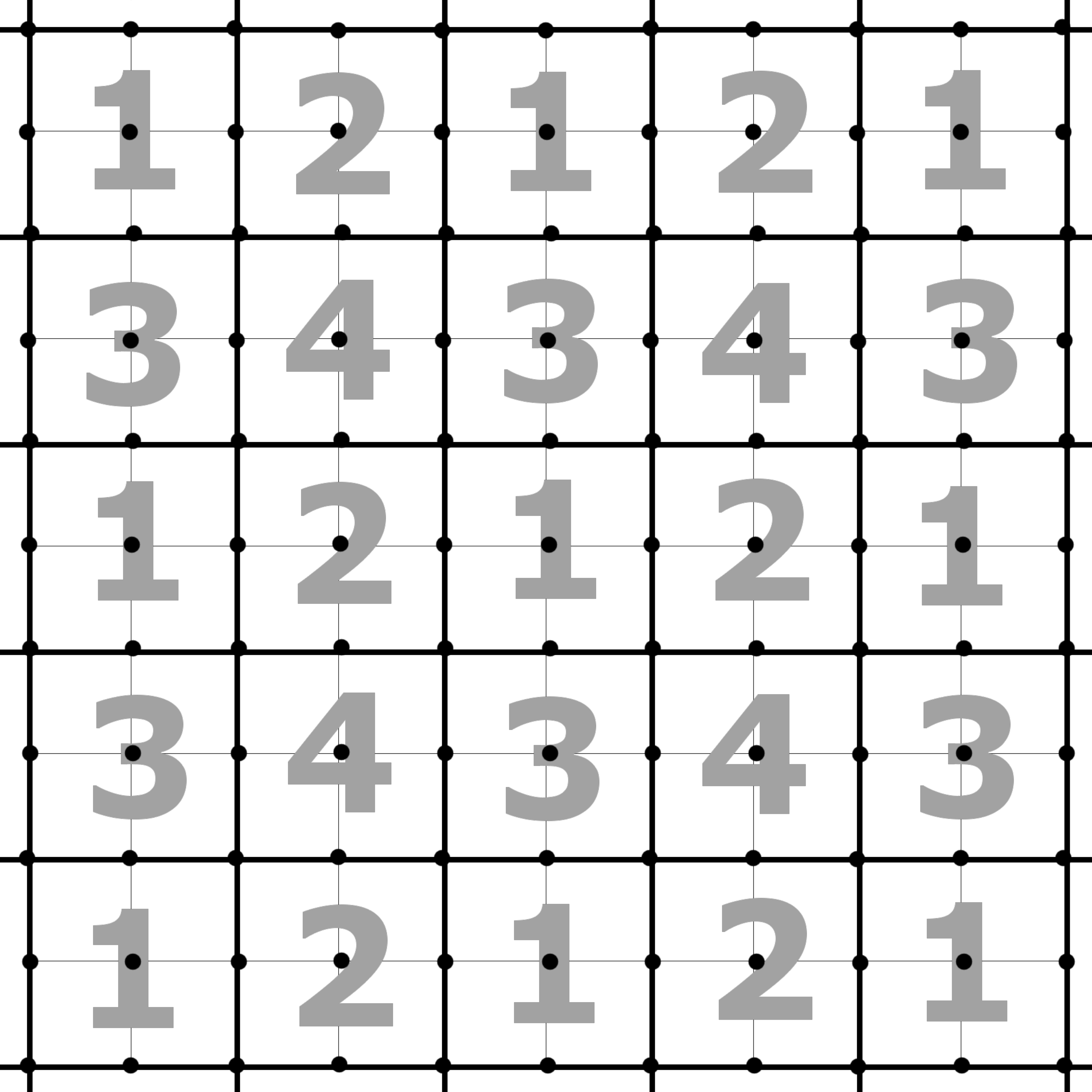}
\caption{Partition of the two-dimensional plane: Blocks of the same type have the same number $k\in\{1,...,4\}$. Here $\Delta=2$.}
\label{fig:grid_partition}
\end{figure}

In order to check the assumption for a general network, we assign to each pair of vertices random, $d$-dimensional coordinates. Then, we plot these coordinates in the $d$-dimensional plane and partition the edges by using a chess-board like partitioning as before. We suggest two example strategies for doing this.

\begin{example}
\label{exp:net_coords}
Let $e_1,,...,e_d\in L_n$ be arbitrary pairs of vertices. For any $n\in\IN, t\in[0,T]$ and $(i,j)\in L_n$ we call $(d_t^n((i,j),e_1),...,d_t^n((i,j),e_d))$ the coordinates of $(i,j)$ at time $t$. Let $G^t(k,m,\Delta)$ for $k=1,...,2^d$ and $m\in\IN$ comprise all pairs $(i,j)$ with coordinates lying in the $m$-th block of type $k$ in a chess-board like grouping (similar to Figure \ref{fig:grid_partition} in the case $d=2$).
\end{example}
Note that above we construct the partition for each time point $t$ individually. Hence, the choice of the reference pairs $e_1,...,e_d$ may depend on time as well. Moreover, the pairs may be chosen randomly since $\Delta$-partitions are allowed to be random. That we produce indeed a $\Delta$-partition in the above example is ensured by the following Lemma.
\begin{lemma}
Let $\Delta>0$ be given. The sets $G^t(k,m,\Delta)$ defined in Example \ref{exp:net_coords} form a $\Delta$-partition of the network in the sense of Definition \ref{def:graph_partition}.
\end{lemma}
\begin{proof}
We consider the case $d=2$ (The proof for $d>2$ follows analogous arguments). $G^t(k,m,\Delta)$ and $G^t(k',m',\Delta)$ are disjoint for $(k,m)\neq (k',m')$ by construction. Let $(i,j),(k,l)\in L_n$ and denote by $(q,r):=(d_t^n((i,j),e_1),d_t^n((i,j),e_2))$ and $(q',r'):=(d_t^n((k,l),e_1),d_t^n((k,l),e_2))$ their respective coordinates. Then we obtain by the triangle inequality
\begin{align*}
q'=d_t^n(e_1,(k,l))&\leq d_t^n(e_1,(i,j))+d_t^n((i,j),(k,l))=q+d_t^n((i,j),(k,l)) \\
q=d_t^n(e_1,(i,j))&\leq d_t^n(e_1,(k,l))+d_t^n((k,l),(i,j))=q'+d_t^n((i,j),(k,l)),
\end{align*}
which yields $d_t^n((i,j),(k,l))\geq|q-q'|$. Analogously, we obtain $d_t^n((i,j),(k,l))\geq|r-r'|$. The second condition in Definition \ref{def:graph_partition} follows if we notice that by definition for $m\neq m'$, $(i,j)\in G^t(k,m,\Delta)$ and $(k,l)\in G^t(k,m',\Delta)$ implies either $|q-q'|\geq\Delta$ or $|r-r'|\geq\Delta$.
\end{proof}

Additionally to Example \ref{exp:net_coords}, we provide another method of how to equip edges with $d$-dimensional coordinates via multidimensional scaling.

\begin{example}
\label{exp:MDS}
Use Multidimensional Scaling (MDS) (cf. \citet{CC94}) to find for each $(i,j)\in L_n$ coordinates $p(i,j)$ in $\IR^d$ such that $\|p(i,j)-p(k,l)\|_2\approx d_t^n((i,j),(k,l))$, where $\|.\|_2$ denotes the Euclidean distance in $\IR^d$.
\end{example}

In general it is not possible to have equality above. So the method yields only an approximation. However, the resulting partition might still be valid for a different $\Delta$. In general we expect that for networks in which the vertices are already related to some position in $\IR^d$ (e.g. geographical positions) the assumption of the existence of such a $\Delta$-partition is not restrictive.

\subsubsection{Example: Momentary-$m$-Dependence}
\label{subsubsec:m-dep}
This section provides an example of a data generating process which is momentarily-$m$-dependent and exchangeable. Consider the following over-simplistic model for the use of on-line communication: A population $V_n:=\{1,...,n\}$ of people (e.g. employees of a company) is connected through a social network with adjacency matrix $C_n$, i.e., people $i,j\in V_n$ are in regular personal contact if $C_{n,ij}=1$. Now a new on-line communication tool is introduced. Consider a pair $(i,j)\in L_n$ with $C_{n,ij}=1$. At a given point in time $t\in[0,T]$, the pair either has started to communicate via the on-line tool ($N_{n,ij}(t)=1$) or not ($N_{n,ij}(t)=0$). We suppose that pairs $(i,j)$ with $C_{n,ij}=0$ will also not connect via the communication tool and hence $N_{n,ij}\equiv 0$ in these cases. So the processes $N_{n,ij}$ have at most one jump in the period $[0,T]$. Suppose we are interested in studying a statistic which depends on the array $\left(N_{n,ij}\right)_{n,ij}$. Clearly it would not be justifiable to assume that all $N_{n,ij}$ are independent because people who are connected will influence each other. However, assuming momentarily-$m$-dependence and exchangeability is less restrictive as we will motivate next.

In order to focus on the main ideas, we restrict to a time-constant network model. However, we can also apply dynamic network models and consider the distribution of a snapshot of the network at a given time of interest $t_0$. As a network generating process we consider a stochastic block model (cf. \citet{HLL83}) with random group assignments. That is, we suppose that every vertex $i\in V_n$ is randomly assigned to a group $g(i)\in\{1,...,\mathcal{G}\}$. While the number $\mathcal{G}\in\IN$ is fixed, the random variables $g(i)$ for $i=1,...,n$ are assumed to be independent and identically distributed. Now we suppose that the random variables $C_{n,ij}\in\{0,1\}$ are independent conditionally on all $g(i)$ and that for $i>j$ $\IP(C_{n,ij}=1|g(i),g(j))=Q(g(i),g(j))$ where $Q\in[0,1]^{\mathcal{G}\times\mathcal{G}}$ contains the connection probabilities. Set $C_{n,ij}=C_{n,ji}$ for $i>j$ and $C_{n,ii}=0$. We suppose that all these random variables are measurable with respect to $\mathcal{F}^n_0$.

The model for the processes $N_{n,ij}$ is as follows. We assume that the decision of a pair $(i,j)$ with $C_{n,ij}=1$ to use the communication tool is influenced by how many neighbouring communication connections are established in the sense that the pair is more likely to use the tool if many others use it as well. In addition, we assume that it takes some time to process information such that if a pair $(i,j)$ uses the tool at time $t_0$ pair $(j,k)$ will be influenced by it not before time $t+\delta_{ij,jk}$ (let $\delta_{ij,jk}=\infty$ if $C_{n,ij}=0$ or $C_{n,jk}=0$). We allow that some pairs process information faster than others but we do not allow chains of arbitrary fast communication, i.e., we suppose there is $\delta_{0,n}>0$ and $m_0\in\IN$ such that
$$\inf_{i_1,...,i_{m_0+1}\in V_n}\sum_{k=1}^{m_0-1}\delta_{i_ki_{k+1},i_{k+1}i_{k+2}}\geq\delta_{0,n}\textrm{ for all }n\in\IN.$$
Let $U_{n,ij}\geq 0$ denote pair $(i,j)$'s perception of the new tool. We suppose that the $U_{n,ij}$ are independent and identically distributed among all pairs. For any pair $(i,j)\in L_n$ define moreover by $L_n(i,j):=\{(k,l):k\in\{i,j\},l\in V_n\}$ the set of potential neighbours of $(i,j)$. Using these preparations, we consider the following model for the process $N_{n,ij}$ for given $\alpha_0,\theta_0>0$
\begin{equation}
\label{eq:example}
N_{n,ij}(t)=C_{n,ij}\Ind\left(U_{n,ij}+\sum_{(k,l)\in L_n(i,j)}^nC_{n,kl}\alpha_0\int_0^{t-\delta_{ij,kl}}N_{n,kl}(r)dr>\theta_0\right).
\end{equation}
For simplicity of exposition, we choose here a model without covariates and consider only the process $(N_{n,ij},C_{n,ij})_{(i,j)\in L_n}$. Since the group assignments and the initial perceptions $U_{n,ij}$ are iid, the process $(N_{n,ij},C_{n,ij})_{(i,j)\in L_n}$ fulfils the exchangeability property \eqref{eq:exchangeability}.

Let $\mathcal{F}^n_t$ denote the canonical filtration with respect to which all $N_{n,ij}$ are adapted and $C_{n,ij}$ and $g(1),...,g(n)$ are measurable with respect to $\mathcal{F}_0^n$. Definition \ref{defin:m-dependence} reads in this situation as follows
\begin{align*}
&\exists m>0,\,\forall n\in\IN,\,\exists\delta_n>0,\,\forall t_0\in[0,T],\,\forall J\subseteq L_n \textrm{Given }\mathcal{F}_{t_0}^n: \\
&\quad\quad(C_{n,ij},N_{n,ij}(t))_{(i,j)\in J,\,t\in[t_0,t_0+6\delta_n]}\textrm{ is cond. independent of } \\
&\quad\quad\sigma\Big((C_{n,ij},N_{n,ij}(r))\cdot\Ind(d((i,j),J)\geq m): r\leq t_0+6\delta_n,\, (i,j)\in L_n\Big).
\end{align*}
Note firstly that $C_{n,ij}$ is measurable with respect to $\mathcal{F}_{t_0}^n$ for all $t_0$ and thus may be treated as a constant. In order to see that the above holds for $m=m_0$ and $6\delta_n<\delta_{0,n}$ we use the following notation:  A sequence of pairs $P=(p_a)_{a=1}^M\subseteq L_n$ is called a path from $(i,j)$ to $(k,l)$ if $p_1=(i,j)$, $p_M=(k,l)$ and $p_a$ and $p_{a+1}$ share at least one vertex. For such a path we denote by $\delta(P):=\sum_{a=1}^{M-1}\delta_{p_a,p_{a+1}}$. Let $t_0\in[0,T], J\subseteq L_n$ be arbitrary and let $(k,l)\in J$. Let, moreover, $(i,j)\in L_n$ be given with $d((i,j),(k,l))\geq m_0$ and let $t\in[t_0,t_0+6\delta_n]$. By construction it is clear that $N_{n,kl}(t)$ depends only on those events of $N_{n,ij}$ which happened before time (the $\inf$ is taken over all paths from $(k,l)$ to $(i,j)$)
$$t-\inf_P \delta(P)\leq t-\delta_{0,n}\leq t_0+6\delta_n-\delta_{0,n}\leq t_0$$
since $6\delta_n\leq\delta_{0,n}$. Information about these is available in $\mathcal{F}^n_{t_0}$. Hence, the events of the processes $N_{n,ij}\Ind(d((i,j),J)\geq m_0)$ on $[t_0,t_0+6\delta_n]$ are non-influential to $N_{n,kl}$ for $(k,l)\in J$ on $[t_0,t_0+6\delta_n]$ provided that $\mathcal{F}_{t_0}^n$ is known. Therefore momentary-$m$-dependence holds.

\subsubsection{Example: Mixing}
\label{subsubsec:mixing}
In this section we will show that a simplified version of the process described in Section \ref{subsubsec:m-dep} is exchangeable and $\beta$-mixing (see also Remark \ref{rem:mge} at the end of this section). Let $G_0$ be a 2-dimensional discrete torus with a suitable number of $n$ vertices, i.e., the network has grid structure as in Figure \ref{fig:grid_partition} but the vertices on the left and on the right are identified, as well as the vertices on the bottom and the top. The random network $G_n$ is obtained by randomly assigning labels to the vertices of $G_0$. As before $C_n$ denotes the adjacency matrix of $G_n$. We consider processes $N_{n,ij}(t)=C_{n,ij}\Ind(A_{n,ij}(t)\geq \theta_0)$ where $A_{n,ij}(t)$ is a stochastic process which we specify now. Let $\phi:\{1,...,r_n\}\to L_n$ be an arbitrary enumeration of the pairs of vertices $L_n$ and let $A(t):=(A_{n,\phi(x)}(t))_{x=1,...,r_n}$. Denote by $\tilde{C}\in\{0,1\}^{r_n\times r_n}$ the random matrix with $\tilde{C}_{x,y}=1$ if and only if $C_{n,\phi(x)}=C_{n,\phi(y)}=1$ and the pairs $\phi(x)$ and $\phi(y)$ share exactly one vertex. Set $\tilde{C}_{x,x}=0$. We suppose that $A(t)$ follows the AR-model
$$A(t)=\alpha_0\tilde{C}A(t)+\epsilon,$$
where $\alpha_0<1/6$ and $\epsilon=(\epsilon_1,...,\epsilon_{r_n})^T$ is (for simplicity) a vector of independent Brownian motions scaled by $t^{-1/2}$ for $t>0$. Then, $\epsilon_x(t)\sim N(0,1)$ for all $t$ and all $x$. Since we assigned the vertex labels randomly, the processes $A_{n,ij}$ and thus $N_{n,ij}$ are exchangeable.

We prove now that the mixing coefficients at a given time $t$ decay exponentially fast. The $\Delta$-partition we consider is as follows: Fix a chess-board like partitioning as in Figure \ref{fig:grid_partition} with side-length $\Delta-1\in\IN$ on the deterministic network $G_0$. The random blocks $G^t(k,m,\Delta)$ are formed based on the edges which lie in the corresponding square in $G_0$. Fix $(k,m_1)\neq(k,m_2)$ and let for ease of notation $I_1:=G^t(k,m_1,\Delta)$ and $I_2:=G^t(k,m_2,\Delta)$. The distance $d$ is defined as before. Then, $d(ij,kl)\geq \Delta$ if $(i,j)\in I_1$ and $(k,l)\in I_2$. Denote $U_1:=\sum_{(i,j)\in L_n}\left(N_{n,ij}(t)-\IE(N_{n,ij}(t)\Ind((i,j)\in I_1)\right)$ and $U_2$ is defined analogously for $I_2$. Note that by the symmetry of the network and the choice of the $\Delta$-partition the conditional distribution of $U_2$ given $C_n$ is actually the same for all realisations of $C_n$. As a consequence $\IP(U_2\in S_2|C_n=C_0)=\IP(U_2\in S_2)$ for all adjacency matrices $C_0$ and all sets $S_2\subseteq\IR$. In consideration of this, we can find the mixing coefficient $\beta(U_1,U_2)$ as the supremum over all partitions $(S_{1,a}),(S_{2,b})$ of $\IR$ of (cf. \citet{DDLLLP07})
\begin{align}
&\frac{1}{2}\sum_{a,b}\left|\IP(U_1\in S_{1,a},U_2\in S_{2,b})-\IP(U_1\in S_{1,a})\IP(U_2\in S_{2,b})\right| \label{eq:eaim} \\
\leq&\sum_{C_0}\IP(C_n=C_0)\frac{1}{2}\sum_{a,b}\big|\IP(U_1\in S_{1,a},U_2\in S_{2,b}|C_n=C_0) \nonumber \\
&\quad-\IP(U_1\in S_{1,a}|C_n=C_0)\IP(U_2\in S_{2,b}|C_n=C_0)\big|, \nonumber
\end{align}
where $\sum_{C_0}$ is the sum over all adjacency matrices. On $C_n=C_0$, the random variables $U_1$ and $U_2$ are deterministic functions of $D_1:=(A_{n,ij})_{(i,j)\in I_1}$ and $D_2:=(A_{n,ij})_{(i,j)\in I_2}$, respectively. Thus, by Pinsker's Inequality (e.g., Lemma 2.5 in \citet{T09})
\begin{align*}
&\frac{1}{2}\sum_{a,b}\big|\IP(U_1\in S_{1,a},U_2\in S_{2,b}|C_n=C_0)-\IP(U_1\in S_{1,a}|C_n=C_0)\IP(U_2\in S_{2,b}|C_n=C_0)\big| \\
\leq& \sqrt{\frac{1}{2}KL((D_1,D_2),(\tilde{D}_1,\tilde{D}_2)|C_n=C_0)},
\end{align*}
where $KL(\cdot,\cdot|C_n=C_0)$ denotes the Kullback-Leibler divergence (conditionally on $C_n=C_0$) and $(\tilde{D_1},\tilde{D}_2)$ are independent with the same marginal distributions as $(D_1,D_2)$. It follows from the properties of the normal-distribution that an exponential bound on $\textrm{Cov}(A_x(t),A_y(t)|C_n=C_0)$ implies a similar exponential bound on the Kullback-Leibler divergence and thus on \eqref{eq:eaim}. Details are given in Appendix \ref{subsec:addexp}. We prove now an exponential bound on the covariances for $x\neq y$.

Note that all eigenvalues of $\alpha_0 \tilde{C}$ can be bounded in absolute value by $6\alpha_0<1$ (since every edge has exactly six neighbours). Hence, $A(t)=(I-\alpha_0\tilde{C})^{-1}\epsilon(t)$ and by the Neumann series representation
$$A_x(t)=\sum_{z=1}^{r_n}\delta_z(x)\epsilon_z(t)\textrm{ for } \delta_z(x)=\sum_{k=0}^{\infty}\alpha_0^k\left(\tilde{C}^k\right)_{x,z}.$$
Thus, conditionally on $C_n$, all $A_x(t)$ are normally distributed. Recall that $\left(\tilde{C}^k\right)_{z_1,z_2}$ gives the number of paths from $\phi(z_1)$ to $\phi(z_2)$ of exactly length $k$. Hence, for all pairs $\phi(z)\in L_n$ we must have $\left(\tilde{C}^k\right)_{x,z}\left(\tilde{C}^r\right)_{y,z}=0$ whenever $r+k< d(\phi(x),\phi(y))$ because otherwise there would be a path of length shorter than $d(\phi(x),\phi(y))$ which connects $\phi(x)$ and $\phi(y)$ via $\phi(z)$. Moreover, $\left(\tilde{C}^k_{z_1,z_2}\right)\leq 6^k$ for all $z_1,z_2\in\{1,...,r_n\}$. Therefore we obtain by symmetry of $\tilde{C}$ that there is a constant $c^*>0$ (which depends only on $\alpha_0$) such that
\begin{align}
\gamma_{xy}:=&\textrm{Cov}(A_x(t),A_y(t))=\sum_{z=1}^{r_n}\delta_z(x)\delta_z(y) \nonumber \\
=&\sum_{k,r=0}^{\infty}\sum_{z=1}^{r_n}\Ind(k+r\geq d(\phi(x),\phi(y))\alpha_0^{k+r}\left(\tilde{C}^k\right)_{x,z}\left(\tilde{C}^r\right)_{y,z} \nonumber \\
\leq&\sum_{k,r=0}^{\infty}\Ind(k+r\geq d(\phi(x),\phi(y))\left(6\alpha_0\right)^{k+r}\leq c^* \sqrt{6\alpha_0}^{d(\phi(x),\phi(y))}. \label{eq:expdecay}
\end{align}
\begin{remark}
\label{rem:mge}
If $\IP(B\in S_2|C_n=C_0)$ depends on $C_0$, we can write a more general version of \eqref{eq:eaim} which requires two estimates: Firstly, the distribution of the sum over a single block may not depend too strongly on the specific network. In that sense, the main task of the $\Delta$-partition is to group pairs together such that similarly behaved blocks emerge. This is possibly also the case in the example from Section \ref{subsubsec:m-dep} if the $\Delta$-partition takes the original group structure into account. Once this holds, in a second step, it suffices to bound the \emph{conditional} mixing coefficients for all fixed network realisations.
\end{remark}

\subsection{Processes Indexed by Vertices}
\label{subsec:vertex_processes}
The dependence concepts in Sections \ref{subsec:m_dependence} and \ref{subsec:mixing_networks} have been introduced for processes $Z_{n,ij}$ which are indexed by pairs $(i,j)\in L_n$. The results also transfer to processes $(\tilde{Z}_{n,i})_{i\in V_n}$ indexed by vertices. The results and definitions from Sections \ref{subsec:m_dependence} and \ref{subsec:mixing_networks} can be obtained for this case by replacing all $Z_{n,ij}$ by $\tilde{Z}_{n,i}$, all indices $(i,j)$ of pairs of vertices by vertex indices $i$ and by replacing the set $L_n$ by $V_n$. Moreover, $r_n$ has to be adopted.

\section{Application}
\label{sec:model}
We apply the previously introduced dependence concepts to find the asymptotic null-distribution of an $L^2$-type test statistic in the following situation. We consider a structured interaction network process $(N_{n,ij},X_{n,ij},C_{n,ij})_{(i,j)\in L_n}$ (cf. Definition \ref{defin:struc_interaction_net}). In the measurability assumption in Section \ref{sec:notation} we consider a Cox-type link function $\lambda$ which depends on an unknown parameter function $\theta_0:[0,T]\to\Theta\subseteq\IR^q$ (recall that $q$ is the dimension of the covariate functions $X_{n,ij}$), i.e., the intensity functions of the counting processes $N_{n,ij}$ are given by
$$\lambda_{n,ij}(t)=C_{n,ij}(t)\exp\left(\theta_0(t)^TX_{n,ij}(t)\right).$$
Examples for choices of the covariate vector $X_{n,ij}$ can be found in \citet{B08}, \citet{PW13} and \citet{KMP19}. Our interest lies in testing the hypothesis
$$\textrm{H}_0:\,\theta_0\equiv\textrm{ const.}\quad\textrm{ vs. }\quad\textrm{H}_1:\,\theta_0\textrm{ is time varying}.$$
On $\textrm{H}_0$, we denote the value of the constant parameter function also by $\theta_0$. For setting up a test statistic, we compare a non-parametric estimator of $\theta_0(t)$ with a parametric estimator which assumes that $\theta_0(t)$ is constant. As non-parametric estimator we use the local maximum likelihood estimator $\hat{\theta}_n(t_0):=\argmin{\theta\in\Theta}\,\ell_n(\theta;t_0)$ as in \citet{KMP19} where $\ell_n(\theta;t_0)$ is the localized-likelihood which is given by
\begin{equation}
\label{eq:local_likelihood}
\begin{array}{ll}
\ell_n(\theta;t_0):=&\underset{(i,j)\in L_n}{\sum}\Big(\int_0^TK_{h,t_0}(t)\theta^TX_{n,ij}(t)dN_{n,ij}(t) \\
&\quad-\int_0^TK_{h,t_0}(t)C_{n,ij}(t)\exp(\theta^TX_{n,ij}(t))dt\Big)
\end{array},
\end{equation}
where $K_{h,t_0}(t):=\frac{1}{h}K\left(\frac{t-t_0}{h}\right)$ is a kernel with kernel function $K$ and bandwidth $h>0$. Note that when removing the kernel $K_{h,t_0}$ in \eqref{eq:local_likelihood} we end up with the \emph{regular} likelihood $\ell_n(\theta)$ for the case when $\theta_0$ is a constant (cf. \citet{ABGK93}). Denote finally by $\bar{\theta}_n$ a parametric estimator for $\theta_0$ which assumes that the parameter function is constant (e.g. the maximum-likelihood estimator $\bar{\theta}_n:=\argmin{\theta\in\Theta}\,\ell_n(\theta)$). Similar as in \citet{HM93} we compare the non-parametric and parametric estimator above by means of the following test statistic
$$T_n:=\int_0^T\left\|\hat{\theta}_n(t_0)-\bar{\theta}_n\right\|^2\bar{p}_n(t_0)w(t_0)dt_0,$$
where $w$ is a non-negative weight function with $\textrm{supp}\, w\subseteq[\delta,T-\delta]$ for $\delta>0$  and $\bar{p}_n(t_0):=\int_0^TK_{h,t_0}(s)p_n(s)ds$ is the smoothed version of $p_n(t)=\IP(C_{n,ij}(t)=1)$. In contrast to \citet{HM93}, we know in advance that we test for a constant function. Therefore we can directly compare the parametric and non-parametric estimate and we do not require additional smoothing. For the statement of the following theorem define (note that under the following Assumption (A3, 1) the right hand side below does not depend on $(i,j)$)
\begin{equation}
\label{eq:def_sigma}
\Sigma(t,\theta):=-\IE\left(X_{n,ij}(t)X_{n,ij}(t)^T\exp\left(\theta^TX_{n,ij}(t)\right)\Big|C_{n,ij}(t)=1\right)
\end{equation}
with the abbreviation (on $H_0$) $\Sigma_t:=\Sigma(t,\theta_0)$. The following theorem gives the asymptotic distribution of the test statistic on the hypothesis $H_0$. The proof is given in Section \ref{subsec:proof_test} in the Appendix.
\begin{theorem}
\label{thm:test_asymptotics}
Under the Assumptions stated in the remainder of this section, on $\textrm{H}_0$
$$r_nh^{\frac{1}{2}}T_n-h^{-\frac{1}{2}}A_n\overset{d}{\to} N(0,B), \,\,n\to\infty,$$
\begin{align*}
\textrm{where }\quad A_n&:=\frac{1}{r_n}\sum_{(i,j)\in L_n}\int_0^TX_{n,ij}(s)^T\int_0^ThK_{h,t}(s)^2\Sigma_t^{-2}\frac{w(t)}{\bar{p}_n(t)}dtX_{n,ij}(s)dN_{n,ij}(s), \\
B&:=4K^{(4)}\int_0^T\textrm{trace}\left(\Sigma_t^{-2}\right)w^2(t)dt, \quad K^{(4)}:=\int_0^2\left(\int_{-1}^1K(v)K(u+v)dv\right)^2du.
\end{align*}
Note that $A_n$ can be approximated by using a plug in estimator for $\Sigma$ and $B$ can be approximated by Lemma \ref{lem_ass:var} in the Appendix.
\end{theorem}

In the following we firstly state an assumption and then discuss its meaning and the intuition behind it. All assumptions are formulated on $\textrm{H}_0$, in particular, $\theta_0$ denotes the true value of the constant parameter function. We use the abbreviation
$$\lambda_{n,ij}(t,\theta):=C_{n,ij}(t)\exp\left(\theta^TX_{n,ij}(t)\right)$$
such that $\lambda_{n,ij}(t,\theta_0)$ denotes the true intensity function on $H_0$.

\textbf{(A1) Boundary Cut-Off} \\
{\em 
$w\colon[0,T]\to[0,\infty)$ is continuous, bounded and $\IT:=\textrm{supp}(w)\subseteq[\delta,T-\delta]$ for some $\delta>0$.
}

\textbf{(A2) Exhaustiveness of $\Theta$} \\
{\em There is an open and bounded set $\Theta\subseteq\IR^q$ (denote the bound by $\tau$) such that $\theta_0\in\Theta$.}

Assumption (A1) allows to ignore convergence issues of the kernel estimator at the boundary and Assumption (A2) allows us to simplify some notation. Both assumptions are not very restrictive.

\textbf{(A3) Modelling Assumptions}
\emph{\begin{enumerate}
\item \label{ass:cond_ind} The conditional distribution of $(X_{n,ij}(s),N_{n,ij}(s))$ given $C_{n,ij}(s)=1$ is independent of $n$ and $(i,j)\in L_n$.
\item \label{ass:const_estimator} For $\bar{p}_n:=\int_0^T\bar{p}_n(s)ds$ the estimator $\bar{\theta}_n$ fulfils $\left\|\bar{\theta}_n-\theta_0\right\|=O_P\left((r_n\bar{p}_n)^{-\frac{1}{2}}\right)$.
\item \label{ass:boundedness} The covariates $X_{n,ij}$ are almost surely bounded by a constant $\hat{K}$. Together with (A2) this implies that $\lambda_{n,ij}(t,\theta)$ is almost surely bounded by a constant $\Lambda$ for all $\theta\in\Theta$.
\end{enumerate}
}

Assumption (A3, \ref{ass:cond_ind}) is identical to Assumption (A1) in \citet{KMP19}. It is reflecting our intuition about the asymptotics of the network: For growing networks we assume that the number of actors to whom a fixed actor has active connections remains bounded over time. In our intuition, the distribution of the covariates and events on an active edge is therefore only influenced by this group which is not growing. In consideration of this, we regard Assumption (A3, \ref{ass:cond_ind}) not restrictive. (A3, \ref{ass:const_estimator}) holds for example for the maximum likelihood estimator as introduced in Chapter VI.1.2. in \citet{ABGK93}. However, for our theory here, it is not required that $\bar{\theta}_n$ is the maximum likelihood estimator. For (A3, \ref{ass:boundedness}) we note that examples of covariates are the number of common friends, age difference, number of interactions in the past and so on. These quantities are naturally bounded e.g. if we believe that interactions and maintaining friendships requires time. More generally, we expect the intensity functions to be bounded if the actors have to invest time in the interactions (e.g. sending a message takes some time even though the actual event of sending is instantaneous). Because in this case, at least on average, actors will not cast arbitrarily many events in a given time frame.

\textbf{(A4) Kernel and Bandwidth}
{\em
\begin{enumerate}
\item \label{ass:bw} For $p_n:=\inf_{t\in[0,T]}p_n(t)$ the bandwidth $h$ fulfils $\frac{\sqrt{r_np_n}\cdot h}{(\log r_n)^{\frac{3}{2}}}\to\infty$ and $h(\log r_n)^2\to0$.
\item \label{ass:kernel_hoelder} The kernel $K:\IR\to[0,\infty)$ is supported on $[-1,1]$ and is Hoelder continuous with exponent $\alpha_K$ and constant $H_K$, i.e., $|K(x)-K(y)|\leq H_K\cdot|x-y|^{\alpha_K}$. As a consequence it is bounded by a constant which we also denote by $K$.
\end{enumerate}
}

(A4, \ref{ass:bw}) holds for example when $h\approx (p_nr_n)^{-\frac{1}{5}}$ is the asymptotically optimal bandwidth choice in most one-dimensional regression contexts (e.g. \citet{T09}), so they are standard for this type of problem. The Hoelder continuity of the kernel in (A4, \ref{ass:kernel_hoelder}) is a mild assumption which avoids technical problems later. For most simple kernels like Epanechnikov or a triangular kernel it is true.

\textbf{(A5) Invertibility of Fisher-Information} \\
\emph{The matrix $\Sigma_t=\Sigma(t,\theta_0)$ (cf. \eqref{eq:def_sigma}) is invertible for all $t\in[0,T]$ and $t\mapsto\Sigma_t$ is continuously differentiable. Particularly, $D:=\sup_{t\in[0,T]}\left\|\partial_t\Sigma_t\right\|<\infty$ and $t\mapsto\Sigma_t$ is uniformly continuous on $[0,T]$.
}

In (A5) we assume that the Fisher Information is invertible. This is a classical assumption. The assumption that $t\mapsto\Sigma_t$ is smooth reflects our believe that the behaviour of the network is also changing smoothly over time. Note that $\Sigma_t$ is a conditional expectation conditional on $C_{n,ij}(t)=1$, i.e., changes in the network itself (appearance or disappearance of edges) do not interfere with the smoothness of $\Sigma_t$.

\textbf{(A6) Behaviour of $p_n(t)$}\\
\emph{The quotient $\frac{\max_{s\in[0,T]}p_n(s)}{\min_{s\in[0,T]}p_n(s)}$ is bounded in $n\in\IN$ and the function $p_n(t)$ is Hoelder continuous with fixed exponent $\alpha_c$ but the constant $H_{n,c}$ may vary like a power of $n$.}

In this assumption we require that $p_n(t)$ lies for a given $n$ always on the same scale. The convergence rate of the non-parametric estimator at a given point in time $t$ depends on $r_np_n(t)$. Hence, we actually assume here that the non-parametric estimator has the same rate at all points in time. Note, however, that $p_n=\min_{s\in[0,T]} p_n(s)\to0$ is still allowed.

Before we can present the assumptions on the weak dependence structure, we introduce the concept of hubs. Informally speaking, a hub is a pair $(i,j)$ which is close to many other pairs.
\begin{definition}
\label{def:hubs}
Let $m>0$, $F\in\IN$ and $[a,b]\subseteq[0,T]$. For a subset of pairs $A\subseteq L_n$ we let
$$K_{m}^A(a,b):=\sup_{(k,l)\in A}\sum_{(i,j)\in L_n}\Ind(d_a^n((i,j),(k,l))< m)\cdot\sup_{u\in [a,b]}C_{n,ij}(u)$$
be the maximal number of active edges being close to pairs in $A$. A pair $(i,j)\in L_n$ is called a \emph{hub} on $[a,b]$ if $K_m^{(i,j)}(a,b)\geq F$.

Consider a collection $[a_t,b_t]\subseteq[0,T]$ for $t\in[0,T]$. Every random variable $H_{UB}^A\in\{0,1\}$ with
$$H_{UB}^A\geq\sup_{t\in[0,T]}\Ind\left(K_m^A(a_t,b_t)\geq F\right)$$
is called \emph{hub-ability} of the set $A$. By $N_{UB}:=\sum_{(i,j)\in L_n}H_{UB}^{ij}$ we denote an upper bound on the number of possible hubs in the networks $G_{n,t}$.
\end{definition}
The definitions of $H_{UB}^A$ and $N_{UB}$ depend on the choice of $([a_t,b_t])_{t\in[0,T]}$. In order to avoid notation clutter, we do not indicate this in the notation. Note that $K_m^{L_n}(a,b)$ denotes the size of the largest hub on $[a,b]$. We think about hubs in the following way: Consider a social media setting where every edge represents the connection between two people. In the works \citet{GWH07,HRW08} it is argued that in social media most of the \emph{friendships} between users are actually inactive in the sense that they do not interchange messages. This underpins the very much believable idea that every actor has only close contact to a bounded number of people. Having close contact means in our formulas that their distance is less than $m$. That means that most people interact with not more than, say $F$ people, regardless of the size of the network. Thus, if one edge exceeds the threshold of $F$, we call it a hub. In the following assumptions (H1) and (H2) we have to balance the size and frequency of hubs. This is necessary because if there was one pair which strongly influences the entire network, inference would be impossible.

\textbf{(H1) Hub Predictability}\\
\emph{For some $m\in\IN$ and for $a_t:=t-4h, b_t:=t+2h$, the random variables $K_m^{L_n}:=\sup_{t\in[0,T]}K_m^{L_n}(a_t,b_t)$ and $H_{UB}^{ij}$  are measurable with respect to $\mathcal{F}_0^n$. As a consequence also $N_{UB}$ is measurable with respect to $\mathcal{F}_0^n$ as well as $\mathcal{C}_{n,ij}:=F+3H_{UB}^{ij}\left(K_m^{L_n}\right)^2$.}

In this assumption we require that the fact whether a pair of vertices has the potential to become a hub is determined in the beginning of the observation period. Note that this does not require that every potential hub is a hub from the beginning. A pair can be close to few others in the beginning and then become a hub later. In addition, the maximal size of the hubs over time is assumed to be determined in the beginning as well (however it might be unknown). This latter assumption might be relaxed to an exponential growth condition.

\textbf{(H2) Hub size restriction}\\
\emph{Let $m\in\IN$ be as in (H1). The frequency of hubs is restricted to the following constraints:
\begin{align}
&\sup_{n\in\IN}\sup_{(i,j),(k,l)\in L_n}\sup_{t\in[0,T]}\IE\left(K_m^{L_n}H_{UB}^{ij,kl}\Big|\sup_{r\in[t-2h,t]}C_{n,kl}(r)C_{n,ij}(t)=1\right)<\infty \label{eq:HubS1}, \\
&\sup_{n\in\IN}\sup_{(i,j),(k,l)\in L_n}\sup_{t\in[0,T]}\sup_{r:|t-r|<2h}\IE\left(\left(K_m^{L_n}\right)^4H_{UB}^{ij}H_{UB}^{kl}\Big|C_{n,ij}(t)C_{n,kl}(r)=1\right)<\infty, \label{eq:HubS7} \\
&\int_0^T\IE\left(\int_{t-2h}^{t-}K_m^{L_n}H_{UB}^{ij}d|M_{n,ij}|(r)\Big|C_{n,ij}(t)=1\right)dt=O(1). \label{eq:AD9}
\end{align}}

The following assumptions refer to the dependence types we reviewed and introduced in Sections \ref{subsec:asymptotic_uncorrelation}-\ref{subsec:mixing_networks}. For a discussion of them we refer to the respective section.

\textbf{(D1) Momentary-$m$-Dependence}\\
\emph{Let $m>0$ be as in (H1). The processes are momentarily-$m$-dependent in the sense of Definition \ref{defin:m-dependence}. Moreover, the conditions \eqref{eq:cond1}-\eqref{eq:cond5} of Theorem \ref{thm:easy_non-pred} are fulfilled for
\begin{align*}
&\tilde{\phi}^J_{n,u_1u_2}(t,r) \\
:=&\frac{1}{r_n}\int_{0}^{\frac{r-t}{h}+2} \sum_{u\in L_n\setminus\{u_1,u_2\}}X_{n,u}(t+yh)^T\tilde{f}_{n}(t+yh,t)X_{n,u_1}(t)X_{n,u_2}(r)^T\tilde{f}_{n}(t+yh,r)^T \nonumber \\
&\,\,\,\,\,\times X_{n,u}(t+yh)C_{n,u}(t+yh)\lambda_{n,u}(t+yh,\theta_0)\Ind(d_{t-4h}^n(u,J)\geq m)dy,
\end{align*}
where $u_1,u_2\in L_n$, $J\subseteq L_n$ ($d_{t}^n(u,\emptyset)=\infty$) and
$$\tilde{f}_{n}(s,t):=\int_{\frac{s-T+\delta}{h}}^{\frac{s-\delta}{h}}K(y)K\left(\frac{t-s}{h}+y\right)\Sigma^{-T}_{s-yh}\Sigma_{s-yh}^{-1}\frac{w(s-yh)}{\bar{p}_n(s-yh)}dy$$
and $\phi_{n,u_1u_2}(t,r):=\tilde{\phi}_{n,u_1u_2}^{\emptyset}(t,r)$.
}

Proving the conditions \eqref{eq:cond1}-\eqref{eq:cond5} is very tedious. Therefore, we assume here that they hold and provide in Appendix \ref{subsec:fassumpD1} a list of technical but easy to believe assumptions under which they can be proven.

\textbf{(D2) Asymptotic Uncorrelation}\\
\emph{It holds that
\begin{align}
\sup_{t\in[0,T]}\frac{1}{r_n^2}\underset{(i,j)\neq (k,l)}{\sum_{(i,j),(k,l)\in L_n}}\frac{\IP(\sup_{r\in[t-2h,t]}C_{n,kl}(r)C_{n,ij}(t)=1)}{p_n(t)^2}=&O(1), \label{CA:AU1}  \\
\sup_{t,r\in[0,T]}\frac{1}{r_n^2}\sum_{(i,j),(k,l)\in L_n}\frac{\IP(C_{n,kl}(r)C_{n,ij}(t)=1)}{p_n(t)p_n(r)}=&O(1). \label{CA:AU2}
\end{align}}

\textbf{(D3) Mixing} \\
\emph{For any $a>0$, $n\in\IN$ and $t\in[0,T]$ there is a $\Delta_n$-partition with $\Delta_n=a\log n$ and $\mathcal{K}$ many types such that for all $(i,j)\in L_n$
$$\sum_{k=1}^{\mathcal{K}}\sum_{m\in\IN}I_{n,ij}^{k,m,t}(\Delta_n)\geq\sup_{r\in[t-2h,t+h]}C_{n,ij}(r)$$
and $I_{n,ij}^{k,m,t}(\Delta_n)$ is measurable with respect to $\mathcal{F}^n_{t-h}$. Define
$$E_{k,m}^{n,t}:=\sum_{(i,j)\in L_n}\int_{-\infty}^{\infty}L_{h,t}(s)\IE\left(I_{n,ij}^{k,m,t}(\Delta_n)C_{n,ij}(s)\right)ds\textrm{ and }S_{k,t}:=\max_{m}\sum_{(i,j)\in L_n}I_{n,ij}^{k,m,t}(\Delta_n),$$
where $L$ is either the kernel $K$ or $\tilde{K}(u)=\frac{1}{2}\Ind(u\in[-2,0])$ and $L_{h,t}$ is defined analogously to $K_{h,t}$. Suppose that for $p_n^*:=\sup_{t\in[0,T]}p_n(t)$ there is $c_3>0$ such that for
$$\Gamma_n^{t}:=\Ind\left(\forall k\in\{1,...,\mathcal{K}\}:\frac{S_{k,t}^2\cdot\log r_np_n^*}{r_np_n^*}\leq c_3,\, S_{k,t}\sqrt{h}\geq 1\right)$$
it holds that $\sup_{t\in[0,T]}\IP\left(\Gamma_n^t=0\right)$ vanishes exponentially fast. Also we suppose that there is a constant $c_2>0$ such that for all $t\in[0,T]$ and either choice of $L$
\begin{equation}
\label{eq:gt}
\frac{1}{r_np_n^*}\sum_{m}E_{k,m}^{n,t}\geq c_2.
\end{equation}
Let $\beta_t$ denote the mixing coefficients as in Definition \ref{defin:network-beta} of
\begin{align*}
Y_{n,ij}:=&\left(\mathcal{C}_{n,ij},\left(N_{n,ij}(r),X_{n,ij}(r),C_{n,ij}(r)\right)_{r\in[t-2h,t+h]}\right)\sup_{r\in[t-2h,t+h]}C_{n,ij}(r).
\end{align*}
We suppose that $\beta_t(\Delta)\leq \alpha_1\exp(-\alpha_2 \Delta)$ for some $\alpha_1,\alpha_2>0$. Let
\begin{align*}
q_n^L(t):=&\int_{-\infty}^{\infty}L_{h,t}(s)\IE(I_{n,ij}^{k,m}|C_{n,ij}(s)=1)ds, \\
H_{n,ij}(s,\theta_0):=&-C_{n,ij}(s)X_{n,ij}^{(r_1)}(s)X_{n,ij}^{(r_2)}(s)\exp(\theta_0^TX_{n,ij}(s))
\end{align*}
for any choice $r_1,r_2\in\{1,...,q\}$. Consider for each $t\in[\delta,T-\delta]$ and each $\theta\in\Theta$ either
\begin{align}
Y_{n,ij}&:=-\left(H_{n,ij}(t,\theta_0)-\IE\left(H_{n,ij}(t,\theta_0)|\Gamma_n^t=1\right)\right)\Gamma_n^t\textrm{ or} \label{eq:UUM1} \\
Y_{n,ij}&:=\int_0^TK_{h,t_0}(s)\left(H_{n,ij}(s,\theta)-\IE\left(H_{n,ij}(s,\theta)|\Gamma_n^t=1\right)\right)ds\cdot \Gamma_n^t \textrm{ or} \label{eq:UUM2} \\
Y_{n,ij}&:=\int_0^TK_{h,t}(s)\left(C_{n,ij}(s)-\IE(C_{n,ij}(s)|\Gamma_n^t=1)\right)ds\cdot\Gamma_n^t \textrm{ or} \label{eq:UUM3} \\
Y_{n,ij}&:=\frac{1}{H_n}\int_{-\infty}^{\infty}\tilde{K}_{h,t}\left(\mathcal{C}_{n,ij}\lambda_{n,ij}(r)-\IE\left(\mathcal{C}_{n,ij}\lambda_{n,ij}(r)|\Gamma_n^t=1\right)\right)dr\cdot \Gamma_n^t. \label{eq:UUM4}
\end{align}
Suppose that for either choice, there is $C>0$ such that for pairwise different vertices $i,j,k,l\in V_n$, all $n\in\IN$ and all $t\in[0,T]$ it holds that (use $L=K$ in the definition of $q_n^L$ for \eqref{eq:UUM1}-\eqref{eq:UUM3} and $L=\tilde{K}$ for \eqref{eq:UUM4})
\begin{equation*}
\begin{array}{ll}
\frac{1}{p_nq^L_n(t)}\textrm{Var}\left(Y_{n,ij}I_{n,ij}^{k,m,t}\right)&\leq C, \\
\frac{\sqrt{r_n}}{p_nq^L_n(t)}\textrm{Cov}\left(Y_{n,ij}I_{n,ij}^{k,m,t},Y_{n,jk}I_{n,jk}^{k,m,t}\right)&\leq C, \\
\frac{r_n}{p_nq^L_n(t)}\textrm{Cov}\left(Y_{n,ij}I_{n,ij}^{k,m,t},Y_{n,kl}I_{n,kl}^{k,m,t}\right)&\leq C.
\end{array}
\end{equation*}
}

The assumptions have been mostly discussed in Sections \ref{subsec:asymptotic_uncorrelation}-\ref{subsec:mixing_networks}. However, we would like to comment on $\Gamma_n^t$ and the measurability in Assumption (D3). The way $\Gamma_n^t$ is used ensures that the mixing property is only required if the partitioning of the network is reasonable. However, we also assume that the probability that the partitioning is reasonable is large. The inequality in \eqref{eq:gt} means that we assume that the percentage of the edges which are on average contained in the blocks of type $k$ is never negligible, i.e., that no block type is obsolete, a plausible assumption. We also tacitly assume that the number of block types $\Kappa$ is the same for all time points and does not change with $n$. This assumption reflects the idea that the network geometry is staying the same while the network size is increasing. The measurability assumption is required because of Lemma \ref{lem:exponential_inequality} in the Appendix. In the proof of the lemma we see that measurability is essential because we have to apply martingale results. In practice this means that the $\sigma$-field $\mathcal{F}_{t-h}^n$ contains the information which at the time $t-h$ inactive pairs $(i,j)$ (i.e., $C_{n,ij}(t-h)=0$) will possibly be active in the interval $[t-h,t+h]$ (i.e., $\sup_{r\in[t-h,t+h]}C_{n,ij}(r)=1$). Since there is $\geq$ in the condition in the beginning of (D3) the information is not required exactly: It is no problem if the partitioning contains a few pairs too many. When adding this information to the filtration we assume that the intensity process remains unaffected. This is plausible because we only add information about the future connectivity (not activity) of pairs which are currently known to be inactive (so they are known to not cast events among each other currently regardless of their future behaviour).

Denote for the next assumption
\begin{align*}
A_n(t):=&\sum_{(i,j)\in L_n}\sup_{u\in[-2,2]}C_{n,ij}(t+uh).
\end{align*}
The following set of assumptions looks very clumsy and difficult to check. However, the reader is politely asked to read the following assumptions by keeping in mind that Assumptions (AD, \ref{eq:AD8},\ref{eq:AD10}) are moment conditions which merely require a polynomial growth (but do not specify the exponent). Moreover, in (AD, \ref{eq:AD10}) the integral is over an interval of length $2h$, so it is to be expected that this integral is small. \\
\textbf{(AD) Additional Dependence} \\
\emph{For any given $k_0\in\IN$ we can choose $\xi>0$ such that
\begin{align}
&\IP\left(\sup_{t\in[0,T]}\frac{A_n(t)}{r_np_n}>\xi\right)=o(n^{-k_0}). \label{eq:AD5}
\end{align}
For the next assumptions we use the notation $d|M_{n,ij}|(s):=dN_{n,ij}(s)+\lambda_{n,ij}(s)ds$. There is $\kappa>0$ such that for all $\xi>1$, $(i,j)\in L_n$ it holds that
\begin{align}
&\sup_{t\in[0,T]}\IE\left[\left(\sum_{(i,j)\in L_n}\int_{0}^T\frac{\left(F+K_m^{L_n}H_{UB}^{ij}\right)}{r_np_n}d|M_{n,ij}|(r)\right)^2\left(\frac{A_n(t)}{r_np_n}\right)^2\Bigg|\frac{A_n(t)}{r_np_n(t)}>\xi_1\right]=O\left(n^{\kappa})\right), \label{eq:AD8} \\
&\sup_{t\in[0,T]}\frac{1}{p_n(t)}\IE\left(\frac{A_n(t)}{r_np_n(t)}\int_{t-2h}^{t-}d|M_{n,ij}|(r)C_{n,ij}(t)\Bigg| \frac{A_n(t)}{r_np_n}>\xi_1\right)=O(n^{\kappa}). \label{eq:AD10}
\end{align}
Additionally suppose that
\begin{align}
&\IE\left(\left(\frac{A_n(t)}{r_np_n}\right)^3\right)=O(1). \label{eq:AD11}
\end{align}}
Assumption (AD, \ref{eq:AD5}) requires the network to concentrate around its expected size. It could be proven on the expense of other technical assumptions. In order to prove \eqref{eq:AD5} we need an exponential inequality for averages of counting processes. Such an inequality can be shown by employing $\beta$-mixing as in the proof of Lemma \ref{lem:network_exp}. However, instead of using the Bernstein inequality  (see e.g. Proposition \ref{prop:bernstein}) we need a tail bound valid for independent sums of counting processes with bounded intensity functions. For our purposes it is sufficient to use a tail bound induced by using Chebyshev's Inequality in its exponential form. The remaining assumptions (AD, \ref{eq:AD8})-(AD, \ref{eq:AD11}) are moment growth conditions. Overall they appear to be weak because we only require that they do not grow super-polynomially. The main reason why we need these assumptions is that the martingales cannot be computed under the conditional probability.

\textbf{(AC) Additional Continuity}\\
\emph{Recall the definition of $H_{n,ij}(s,\theta_0)$ in (D3). For every choice of entries $r_1,r_2\in\{1,...,q\}$ there is $k>0$ such that
\begin{align}
\underset{|t-t^*|<n^{-k}}{\sup_{t,t^*\in[0,T]}}\left|\frac{1}{r_n}\sum_{(i,j)\in L_n}\left(\frac{H_{n,ij}(t,\theta_0)}{p_n(t)}-\frac{H_{n,ij}(t^*,\theta_0)}{p_n(t^*)}\right)\right|&\overset{P}{\to}0, \label{eq:AC1} \\
\underset{|t-t^*|<n^{-k}}{\sup_{t,t^*\in[0,T]}}\left|\IE\left(\frac{H_{n,ij}(t,\theta_0)}{p_n(t)}-\frac{H_{n,ij}(t^*,\theta_0)}{p_n(t^*)}\right)\right|&\overset{P}{\to}0. \label{eq:AC2}
\end{align}}
Instead of posing specific assumptions on the covariate processes $X_{n,ij}$, we choose to state the continuity which is required in the proofs directly. Assumption (AC, \ref{eq:AC2}) could be replaced by assuming that the conditional expectation function $\IE(X_{n,ij}(t)|C_{n,ij}(t)=1)$ is uniformly continuous. For Assumption (AC, \ref{eq:AC1}) we could for example assume that the sample paths of the covariate processes are continuous and that the number of edges which change their status in a small time interval is very small.

\section{Bike Data Illustration}
\label{sec:bikes}
In this section we apply the test from Section \ref{sec:model} to bike-sharing data from Washington D.C. The data is publicly available at {\tt https://www.capitalbikeshare.com/system-data}. For this small application we use the same setting as in Section 3.2 in \citet{KMP19}. In particular we consider the $527$ bike stations as vertices. The bike stations $i$ and $j$ \emph{interact} whenever there is a bike ride from station $i$ to station $j$. In contrast to \citet{KMP19} we consider only bike rides which happened on May, 5th 2018 between 5am in the morning and 10pm in the evening. We consider a short time span because for a longer time span it would be obvious that the parameter function is not constant (e.g. on weekends and weekdays). Without additional detailed information about short term effects (e.g. street closures due to accidents or increased biking due to festivals), it is difficult to observe the true dynamic network. We therefore use a non-dynamic \emph{conservative} network as described in Section 2.1 \citet{KMP19}: We consider two bike stations $i$ and $j$ connected by an edge if they are regularly used by which we mean that there were at least ten bike rides from $i$ to $j$ in April 2018 (that is, more than two rides per week). The true (but unobserved) time varying network is supposed to contain at each time more edges than the conservative network. But the above methodology could also be applied if we had a dynamic conservative network. Note finally that this small data application serves just as an illustration and is not meant to be and in-depth analysis of bike data which would particularly include a sensitivity analysis of the threshold for the network construction.

As covariates, we choose for this application the geographical distance between the bike stations. Let $d_{i,j}$ denote the logarithm of the distance (in minutes of bike time) between bike station $i$ and $j$. Then, we consider the following covariate vector
$$X_{n,ij}=(1,\max(d_{i,j},0),\max(d_{i,j},0)^2)^T.$$

We suppose that the weak dependence concepts which we introduced in Sections \ref{subsec:asymptotic_uncorrelation}-\ref{subsec:mixing_networks} are applicable here because the bike stations have an underlying geographic structure and it is very plausible that bike connections which are geographically far away can be treated more or less independently. Therefore, we use a distance function which is related to the geographical distance (recall that we do not need the actual values of the distance function to apply the technique). To be more specific: If we observe the bike rides between two bike stations $i$ and $j$ during a short time-period $[t_0,t_1]$, we can make inference about other bike rides between other stations $k$ and $l$ in the same time period $[t_0,t_1]$ only if these stations lie geographically close to each other. If there is e.g. a sudden traffic incident which affects the bike rides between $i$ and $j$ it is likely that the bike rides between $k$ and $l$ are also affected, if they lie in the same area. However, if they lie in a different part of the city, there is no influence. Therefore, we assume that asymptotic uncorrelation and $\beta$-mixing are plausible assumptions. In order for the assumption of momentary-$m$-dependence to hold we need to assume that global events which effect the entire city, like big sport events, need to be included in the filtration as we condition on it. As a consequence special events should be included in the intensity function too. Since we restrict the data example to one day (May, 5th 2018) this is a plausible assumption too.

The bandwidth choice for the non-parametric estimator as defined above \eqref{eq:local_likelihood} is carried out in the same way as in \citet{KMP19} and for details about the procedure we refer to that paper: We compute firstly for different bandwidths $h$ a prediction of the bike rides per edge by using a locally-linear estimator with a one-sided kernel. The resulting prediction error is seen in Figure \ref{fig:bwchoice} for a discrete grid of choices of $h$ (we chose the grid for computational simplicity). It can be seen that the prediction error starts to flatten out roughly at $h=1$ and is minimal for $h=1.1$. So we take that bandwidth and transfer it to the case of a regular kernel estimator by dividing by $\rho\approx1.82$ (see \citet{KMP19} for details). Hence, the bandwidth we use is $h\approx0.604$. The non-parametric and parametric estimates in are shown in Figure \ref{fig:thetas}. In this scenario the centred and scaled test statistic yields a value of above $16$. At least asymptotically, we consider the centred and scaled test statistics to be $N(0,1)$ distributed if the underlying data generating process has indeed a constant parameter function. So from this point of view, we have provided evidence that the model with the time-varying parameter function fits the data better. When looking at Figure \ref{fig:thetas} this result is at least intuitively not surprising. If we focus on a shorter time period, e.g. 4pm to 8pm, the result is not as extreme. In Figure \ref{fig:zoomed_thetas} the corresponding estimators are shown. In this case the scaled and centred test statistic is about $-0.79$ which results in a p-value of about $0.43$ which is usually not regarded as significant.

\begin{figure}
\center
\includegraphics[width=0.5\textwidth]{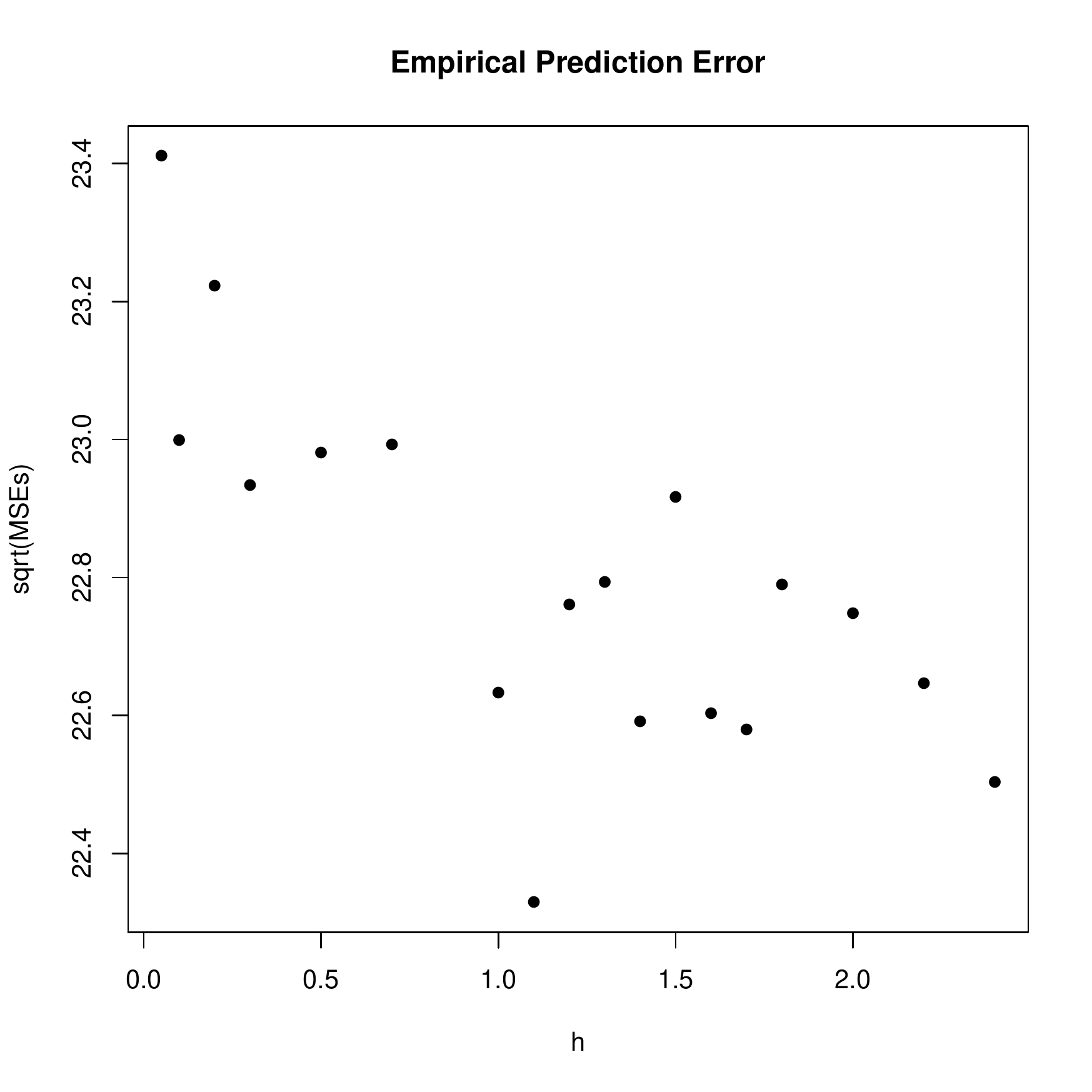}
\caption{Empirical Prediction Error for different bandwidths by using the estimate from a locally-linear estimator with a one-sided kernel.}
\label{fig:bwchoice}
\end{figure}

\begin{figure}
\begin{subfigure}{0.5\textwidth}
\includegraphics[width=\linewidth]{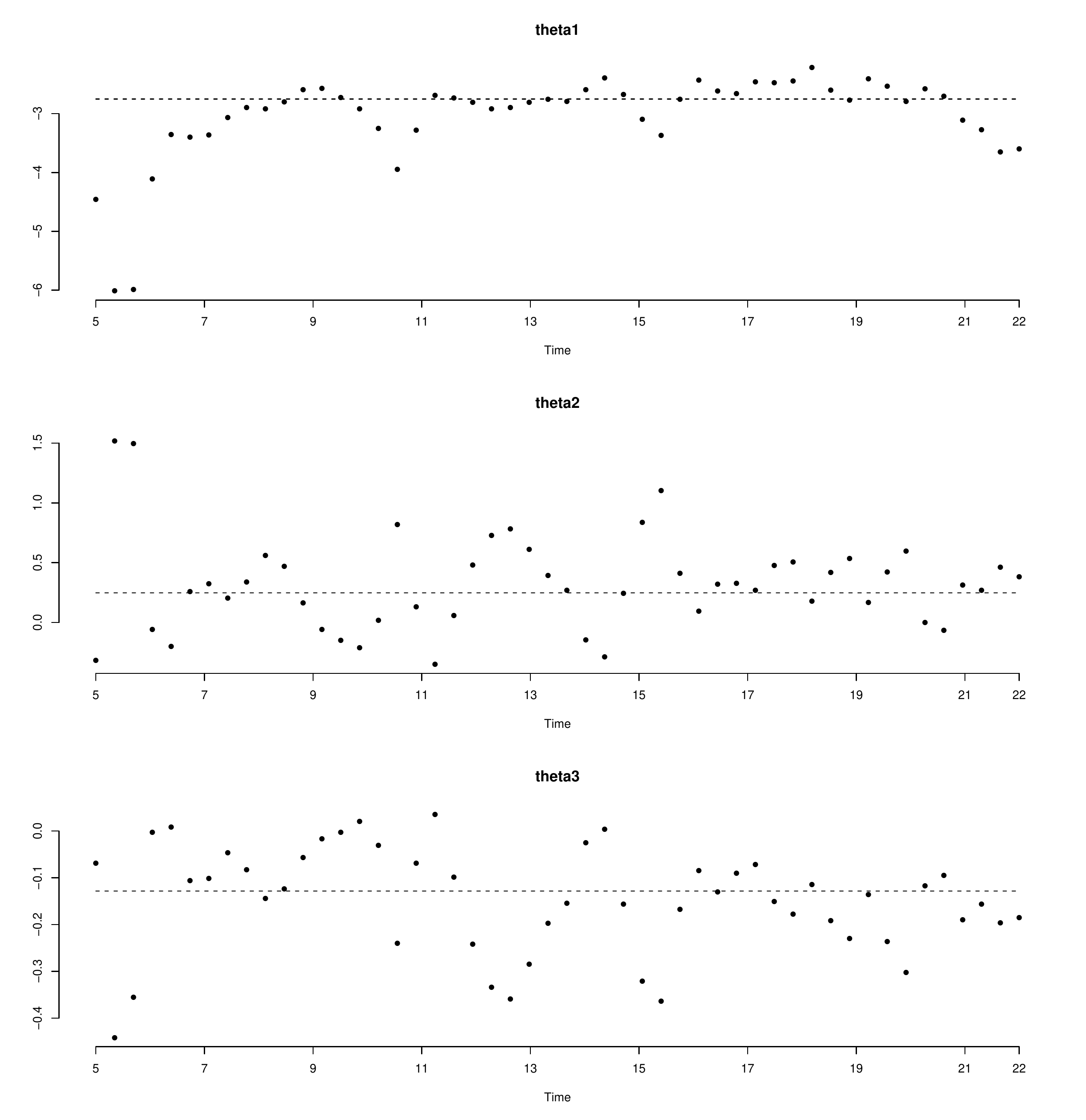}
\caption{Complete Day}
\label{fig:thetas}
\end{subfigure}%
\begin{subfigure}{0.5\textwidth}
\includegraphics[width=\linewidth]{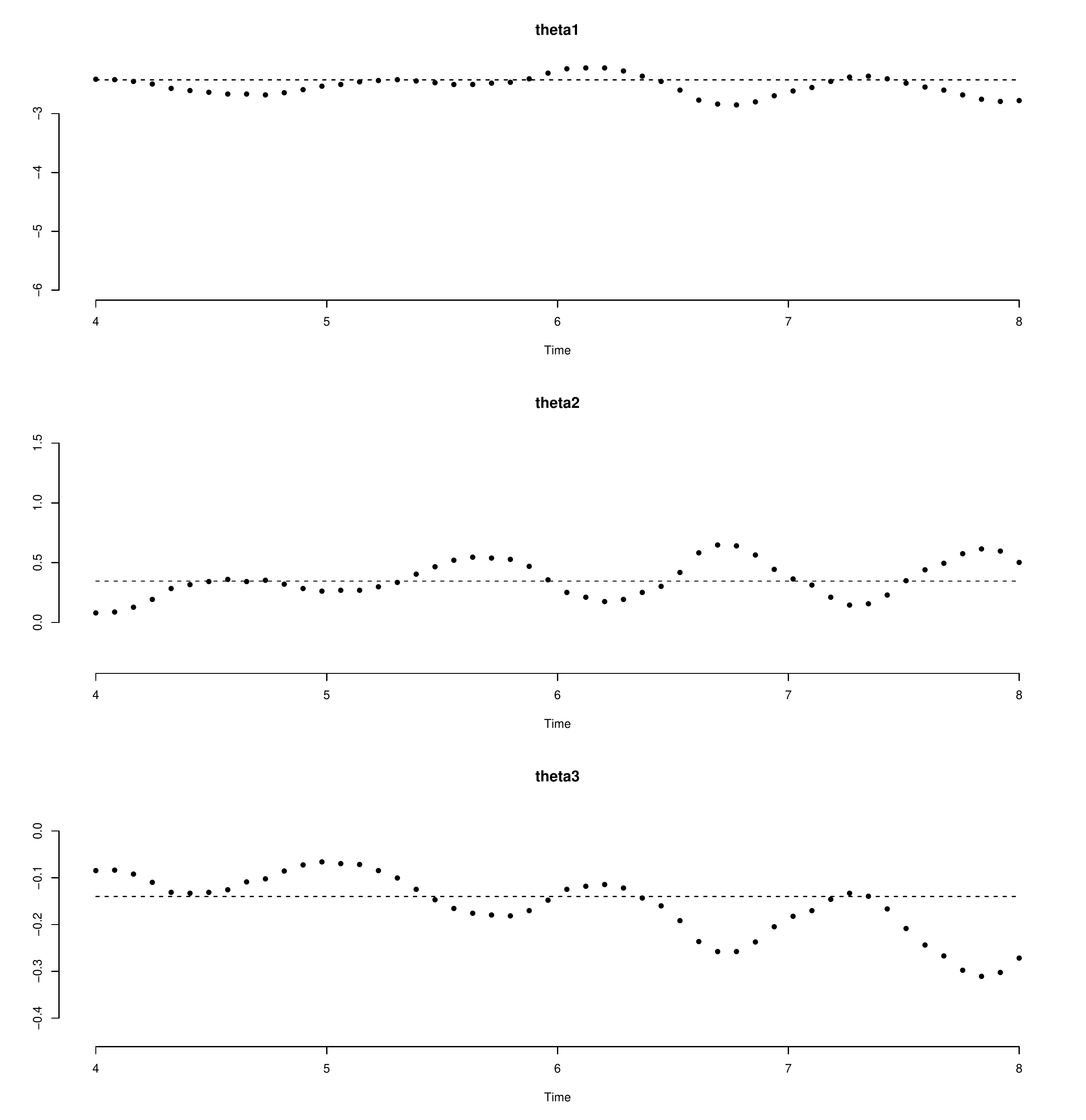}
\caption{Afternoon}
\label{fig:zoomed_thetas}
\end{subfigure}
\caption{The dots show evaluations of the non-parametric estimator for each covariate function, while the dotted lines refer to the estimator which assumes a constant parameter function.}
\end{figure}

\newpage
\section{Appendix}
\label{sec:appendix}
\subsection{Proofs of Section \ref{sec:main}}
\label{sec:proofs}
\begin{proof}[Proof of Theorem \ref{thm:easy_non-pred}]
The idea of the proof is to translate the convergence statement about $\phi_{n,u_1u_2}$ to statements about $\tilde{\phi}_{n,u_1u_2}^J$. This will be useful because the latter are partially predictable with respect to the short sighted filtration. Since we have certain processes which are martingales with respect to the short sighted filtration (cf. Lemma \ref{lem:martingale}) we can make use of martingale properties of the Itô Integral. For the first step, we see that the asymptotic behaviour of \eqref{eq:aim} is the same as the sum over the leave-$m$-out approximations, i.e.,
\begin{align}
&\eqref{eq:aim} \nonumber \\
=&\frac{1}{r_n}\underset{u_1\neq u_2}{\sum_{u_1,u_2\in L_n}}\int_0^T\int_{t-2\delta_n}^{t-}\phi_{n,u_1u_2}(t,r)-\tilde{\phi}_{n,u_1u_2}^{u_1u_2}(t,r)dM_{n,u_2}(r)dM_{n,u_1}(t) \label{eq:a1} \\
&+\frac{1}{r_n}\underset{u_1\neq u_2}{\sum_{u_1,u_2\in L_n}}\int_0^T\int_{t-2\delta_n}^{t-}\tilde{\phi}_{n,u_1u_2}^{u_1u_2}(t,r)dM_{n,u_2}(r)dM_{n,u_1}(t) \label{eq:a2}
\end{align}
and \eqref{eq:a1} converges to zero by \eqref{eq:cond1}. Hence, we only have to study \eqref{eq:a2}. $\tilde{\phi}_{n,u_1u_2}^{u_1u_2}(t,r)$ is partially-predictable with respect to the filtration $\mathcal{F}^{n,u_1u_2,m}_t$ and, by the assumption of Momentary $m$-Dependence (c.f. Definition \ref{defin:m-dependence} and Lemma \ref{lem:martingale}), $M_{n,ij}$ is a martingale with respect to $\mathcal{F}^{n,J,m}_t$ for all $J\subseteq L_n$ with $(i,j)\in J$. We will use this observation in order to prove that \eqref{eq:a2} converges to zero in probability by applying Markov's Inequality:
\begin{align}
&\IE(\eqref{eq:a2}^2) \nonumber \\
=&\frac{1}{r_n^2}\underset{u_1\neq u_2,u_3\neq u_4}{\sum_{u_1,u_2,u_3,u_4\in L_n}}\IE\Bigg[\int_0^T\int_{t-2\delta_n}^{t-}\tilde{\phi}_{n,u_1u_2}^{u_1u_2}(t,r)-\tilde{\phi}_{n,u_1u_2}^{u_1u_2u_3u_4}(t,r)dM_{n,u_2}(r)dM_{n,u_1}(t) \nonumber \\
&\quad\quad\times\int_0^T\int_{t-2\delta_n}^{t-}\tilde{\phi}_{n,u_3u_4}^{u_3u_4}(t,r)-\tilde{\phi}_{n,u_3u_4}^{u_1u_2u_3u_4}(t,r)dM_{n,u_4}(r)dM_{n,u_3}(t)\Bigg] \label{eq:Ia} \\
+&\frac{2}{r_n^2}\underset{u_1\neq u_2,u_3\neq u_4}{\sum_{u_1,u_2,u_3,u_4\in L_n}}\IE\Bigg[\int_0^T\int_{t-2\delta_n}^{t-}\tilde{\phi}_{n,u_1u_2}^{u_1u_2}(t,r)-\tilde{\phi}_{n,u_1u_2}^{u_1u_2u_3u_4}(t,r)dM_{n,u_2}(r)dM_{n,u_1}(t) \nonumber \\
&\quad\quad\quad\quad\quad\quad\times\int_0^T\int_{t-2\delta_n}^{t-}\tilde{\phi}_{n,u_3u_4}^{u_1u_2u_3u_4}(t,r)dM_{n,u_4}(r)dM_{n,u_3}(t)\Bigg] \label{eq:Ib} \\
+&\frac{1}{r_n^2}\underset{u_1\neq u_2,u_3\neq u_4}{\sum_{u_1,u_2,u_3,u_4\in L_n}}\IE\Bigg[\int_0^T\int_{t-2\delta_n}^{t-}\tilde{\phi}_{n,u_1u_2}^{u_1u_2u_3u_4}(t,r)dM_{n,u_2}(r)dM_{n,u_1}(t) \nonumber \\
&\quad\quad\quad\quad\quad\quad\times\int_0^T\int_{t-2\delta_n}^{t-}\tilde{\phi}_{n,u_3u_4}^{u_1u_2u_3u_4}(t,r)dM_{n,u_4}(r)dM_{n,u_3}(t)\Bigg] \label{eq:Ic}
\end{align}

We will treat the terms \eqref{eq:Ia}-\eqref{eq:Ic} separately. Note, that in contrast to \eqref{eq:aim}, all of the above expressions contain only the approximations with their predictability property. We will show in the following how this is useful.

\eqref{eq:Ia} converges to zero by \eqref{eq:cond2a}.

In order to see that \eqref{eq:Ib} converges to zero, we note firstly that the two stochastic integrals in \eqref{eq:Ib} (with respect to $M_{n,u_1}(t)$ and $M_{n,u_3}(t)$) are martingales with respect to the correct leave-$m$-out filtrations (namely $\mathcal{F}_t^{n,u_1,m}$ and $\mathcal{F}_t^{n,u_3,m}$, respectively). Although these two filtrations are in general not the same, we can make use of the fact that the leave-$m$-out filtrations allow future knowledge. Define furthermore for Lebesgue sets $A\subseteq \IR$
\begin{align*}
\bar{\Phi}_{n,u_1u_2}^{u_1u_2u_3u_4}(t,r)&:=\tilde{\phi}_{n,u_1u_2}^{u_1u_2}(t,r)-\tilde{\phi}_{n,u_1u_2}^{u_1u_2u_3u_4}(t,r) \\
I_{A}(u_1,u_2,u_3,u_4)&:=\int_{A\cap[0,T]}\int_{t-2\delta_n}^{t-}\bar{\Phi}_{n,u_1u_2}^{u_1u_2u_3u_4}(t,r)dM_{n,u_2}(r)dM_{n,u_1}(t), \\
J_A(u_1,u_2,u_3,u_4)&:=\int_{A\cap[0,T]}\int_{t-2\delta_n}^{t-}\tilde{\phi}_{n,u_3u_4}^{u_1u_2u_3u_4}(t,r)dM_{n,u_4}(r)dM_{n,u_3}(t).
\end{align*}
Note that $M_{n,u_3}$ and $M_{n,u_4}$ are adapted with respect to all leave-$m$-out filtrations. Since $\tilde{\phi}_{n,u_3u_4}^{u_1u_2u_3u_4}(t,r)$ is partially-predictable with respect to $\mathcal{F}_t^{n,u_1u_2u_3u_4,m}$, we get that
$$t\mapsto\int_{t-2j}^{t-}\tilde{\phi}_{n,u_3u_4}^{u_1u_2u_3u_4}(t,r)dM_{n,u_4}(r)$$
is predictable (cf. Definition \ref{def:preliminary_partially_predictable}) and as a consequence, $t\mapsto J_{[0,t)}(u_1,u_2,u_3,u_4)$ is predictable as well with respect to $\mathcal{F}_t^{n,u_1u_2u_3u_4,m}$.

With these definitions we have (if $\alpha>\beta$ we define $(\alpha,\beta]:=\emptyset$)
\begin{align}
&\eqref{eq:Ib} \nonumber \\
=&\frac{2}{r_n^2}\underset{u_1\neq u_2,u_3\neq u_4}{\sum_{u_1,u_2,u_3,u_4\in L_n}}\IE\Bigg[\int_0^T\int_{t-2\delta_n}^{t-}\bar{\Phi}_{n,u_1u_2}^{u_1u_2u_3u_4}(t,r)dM_{n,u_2}(r)\cdot J_{[t,t+2\delta_n]}(u_1,u_2,u_3,u_4)dM_{n,u_1}(t)\Bigg] \label{eq:case1} \\
&+\frac{2}{r_n^2}\underset{u_1\neq u_2,u_3\neq u_4}{\sum_{u_1,u_2,u_3,u_4\in L_n}}\IE\Bigg[\int_0^T\int_{t-2\delta_n}^{t-}\bar{\Phi}_{n,u_1u_2}^{u_1u_2u_3u_4}(t,r)dM_{n,u_2}(r)\cdot J_{[0,t)}(u_1,u_2,u_3,u_4)dM_{n,u_1}(t)\Bigg] \label{eq:case2} \\
&+\frac{2}{r_n^2}\underset{u_1\neq u_2,u_3\neq u_4}{\sum_{u_1,u_2,u_3,u_4\in L_n}}\IE\Bigg[\int_0^T\int_{t-2\delta_n}^{t-}\bar{\Phi}_{n,u_1u_2}^{u_1u_2u_3u_4}(t,r)dM_{n,u_2}(r)\cdot J_{(t+2\delta_n,T]}(u_1,u_2,u_3,u_4)dM_{n,u_1}(t)\Bigg] 
\label{eq:case3}
\end{align}
We show that this is $o(1)$ by considering the tree lines separately. Recall therefore that $F_{u_1}(t)=\{u_2\in L_n: d_t^n(u_2,u_1)\geq m\}$ is the set of pairs which are further away than $m$ from $u_1$ at time $t$.

For \eqref{eq:case1}, we prove firstly that for each $q\in[t-2\delta_n,t]$
$$J_{[t,t+2\delta_n]}(u_1,u_2,u_3,u_4)\Ind(u_3,u_4\in F_{u_1}(q))$$
is measurable with respect to $\mathcal{F}_{q}^{n,u_1,m}$. This follows from the following intermediate results:
\begin{enumerate}
\item The integrators $M_{n,u_3}$ and $M_{n,u_4}$ in $J_{[t,t+2\delta_n]}(u_1,u_2,u_3,u_4)$ are only considered up to time at most $t+2\delta_n$ and $\mathcal{F}_{q}^{n,u_1,m}$ contains information up to and including time $q+6\delta_n\geq t+4\delta_n$ for processes which are at time $q$ at least of distance $m$ to $u_1$.
\item We show that $\int_{a-2\delta_n}^{a-}\tilde{\phi}_{n,u_3u_4}^{u_1u_2u_3u_4}(a,r)dM_{n,u_2}(r)\Ind(u_3,u_4\in F_{u_1}(q))$ is measurable with respect to $\mathcal{F}_q^{n,u_1,m}$ for all $a\in[t,t+2\delta_n]$
\begin{enumerate}
\item $\tilde{\phi}_{n,u_3u_4}^{u_1u_2u_3u_4}(a,r)$ is partially-predictable with respect to $\tilde{\mathcal{F}}_{u_3u_4,a}^{n,u_1u_2u_3u_4,m}$ by assumption. In particular, it is measurable with respect to $\tilde{\mathcal{F}}_{u_3u_4,t+2\delta_n}^{n,u_1u_2u_3u_4,m}$ for all $r<a\leq t+2\delta_n$. Thus $\tilde{\phi}_{n,u_3u_4}^{u_1u_2u_3u_4}(a,r)\Ind(u_3,u_4\in F_{u_1}(q))$ requires two types of information: One on $X_{u_3}$ and $X_{u_4}$ up to time $t+2\delta_n$, and another type of information about the future (after $t+2\delta_n$) on all processes which are away from $u_1u_2u_3u_4$. Both are contained in $\mathcal{F}_q^{n,u_1,m}$ as we shall show in the following two steps.
\item The information about $X_{u_3}(\tau)$ and $X_{u_4}(\tau)$ for $\tau\leq t+2\delta_n$ is well included in $\mathcal{F}_q^{n,u_1,m}$ by the same arguments as in 1.
\item Let $s\leq t-2\delta_n$ and $r\leq s+6\delta_n$, then
$$[N_{n,ij}(r),X_{n,ij}(r),C_{n,ij}(r)]\cdot\Ind(d_s^n((i,j),\{u_1,u_2,u_3,u_4\})\geq m)$$
is measurable with respect to $\mathcal{F}_q^{n,u_1,m}$ because $s\leq t-2\delta_n\leq q$.
\end{enumerate}
\end{enumerate}
Together the above points imply that $J_{[t,t+2\delta_n]}(u_1,u_2,u_3,u_4)\Ind(u_3,u_4\in F_{u_1}(t-2\delta_n))$ is predictable with respect to $\mathcal{F}_{t}^{n,u_1,m}$. Moreover, $M_{n,u_1}$ is a martingale with respect to $\mathcal{F}_t^{n,u_1,m}$ by momentary $m$-dependence. Hence,

\begin{align}
&\eqref{eq:case1} \nonumber \\
=&\frac{2}{r_n^2}\underset{u_1\neq u_2,u_3\neq u_4}{\sum_{u_1,u_2,u_3,u_4\in L_n}}\IE\Bigg[\int_0^T\int_{t-2\delta_n}^{t-}\bar{\Phi}_{n,u_1u_2}^{u_1u_2u_3u_4}(t,r)dM_{n,u_2}(r)J_{[t,t+2\delta_n]}(u_1,u_2,u_3,u_4) \nonumber \\
&\quad\quad\times\Ind(\neg u_3,u_4\in F_{u_1}(t-2\delta_n))dM_{n,u_1}(t)\Bigg]. \nonumber
\end{align}
The last part is $o(1)$ by \eqref{eq:cond3}.

In \eqref{eq:case2}, we see that $J_{[0,t)}(u_1,u_2,u_3u_4)$ is predictable with respect to \\$\mathcal{F}_{t}^{n,u_1,m}\supseteq\mathcal{F}_t^{n,u_1u_2u_3u_4,m}$. Thus, we conclude by using that $M_{n,u_1}$ is a martingale with respect to $\mathcal{F}_{t}^{n,u_1,m}$ (with analogue arguments as in the first case): $\eqref{eq:case2}=0$.

For \eqref{eq:case3}, we note firstly that
\begin{align*}
&\int_0^T\int_{t-2\delta_n}^{t-}\bar{\Phi}_{n,u_1u_2}^{u_1u_2u_3u_4}(t,r)dM_{n,u_2}(r)\cdot J_{(t+2\delta_n,T]}(u_1,u_2,u_3,u_4)dM_{n,u_1}(t) \\
=&\int_0^T\int_{\xi-2\delta_n}^{\xi-}\tilde{\phi}_{n,u_3u_4}^{u_1u_2u_3u_4}(\xi,\rho)dM_{n,u_4}(\rho)\cdot I_{[0,\xi-2\delta_n)}(u_1,u_2,u_3,u_4)dM_{n,u_3}(\xi).
\end{align*}
Now, we can play a similar game: This time, $M_{n,u_3}$ is a martingale with respect to $\mathcal{F}_{\xi}^{n,u_3,m}$. Furthermore, $I_{[0,\xi-2\delta_n)}(u_1,u_2,u_3,u_4)$ requires knowledge of $M_{n,u_1}(\tau)$, $M_{n,u_2}(\tau)$, $X_{n,u_1}(\tau)$ and $X_{n,u_2}(\tau)$ for $\tau<\xi-2\delta_n$ which is included in $\mathcal{F}_{\xi}^{n,u_3,m}$ as well as knowledge of $[N_{n,ij}(r),X_{n,ij}(r),C_{n,ij}(r)]\cdot\Ind(d_s^n((i,j),\{u_1,u_2\})\geq m)$ for $s\leq\xi-6\delta_n$ and $r\leq s+6\delta_n$, i.e., $r\leq\xi$ which is again included in $\mathcal{F}_{\xi}^{n,u_3,m}$. Hence, $\xi\mapsto I_{[0,\xi-2\delta_n)}(u_1,u_2,u_3,u_4)$ is predictable with respect to $\mathcal{F}_{\xi}^{n,u_3,m}$. Hence, the integrand of \eqref{eq:case3} is a martingale and we obtain $\eqref{eq:case3}=0$. Thus, we have shown that $\eqref{eq:Ib}=o(1)$.

Finally, we consider \eqref{eq:Ic}. Therefore note firstly that $\tilde{\phi}_{n,u_1u_2}^{u_1u_2u_3u_4}(t,r)$ and $\tilde{\phi}_{n,u_3u_4}^{u_1u_2u_3u_4}(t,r)$ are both partially-predictable with respect to $\mathcal{F}_t^{n,u_1u_2u_3u_4,m}$. Moreover, $M_{n,u_1}$, $M_{n,u_2}$, $M_{n,u_3}$ and $M_{n,u_4}$ are all martingales with respect to $\mathcal{F}_t^{n,u_1u_2u_3u_4,m}$. Hence,
$$\int_{t-2\delta_n}^{t-}\tilde{\phi}_{n,u_1u_2}^{u_1u_2u_3u_4}(t,r)dM_{n,u_2}(r)$$
is also a predictable function in $t$ and
$$\int_0^T\int_{t-2\delta_n}^{t-}\tilde{\phi}_{n,u_1u_2}^{u_1u_2u_3u_4}(t,r)dM_{n,u_2}(r)dM_{n,u_1}(t)$$
is a martingale. The same holds when $M_{n,u_1}$ and $M_{n,u_2}$ are replaced by $M_{n,u_3}$ and $M_{n,u_4}$. Hence, for $u_1\neq u_3$
\begin{eqnarray*}
&&\IE\Bigg[\int_0^T\int_{t-2\delta_n}^{t-}\tilde{\phi}_{n,u_1u_2}^{u_1u_2u_3u_4}(t,r)dM_{n,u_2}(r)dM_{n,u_1}(t) \\
&&\quad\quad\quad\times\int_0^T\int_{t-2\delta_n}^{t-}\tilde{\phi}_{n,u_3u_4}^{u_1u_2u_3u_4}(t,r)dM_{n,u_4}(r)dM_{n,u_3}(t)\Bigg]=0.
\end{eqnarray*}
For $u_1=u_3$ we will apply firstly a martingale result to compute the covariance of the two stochastic integrals (first equality below), in the second equality below we employ a similar technique as in the computations for \eqref{eq:Ib}: Note that $C_{n,u_1}(t)\lambda_{n,u_1}(t)\Ind(u_2,u_4\in F_{u_1}(t-2\delta_n))$ is measurable with respect to $\mathcal{F}_{t-2\delta_n}^{n,u_2u_4,m}$, additionally $M_{n,u_2}$ and $M_{n,u_4}$ are martingales with respect to $\mathcal{F}_t^{n,u_2u_4,m}$. Hence,
\begin{align*}
&\IE\Bigg[\int_0^T\int_{t-2\delta_n}^{t-}\tilde{\phi}_{n,u_1u_2}^{u_1u_2u_4}(t,r)dM_{n,u_2}(r)dM_{n,u_1}(t)\int_0^T\int_{t-2\delta_n}^{t-}\tilde{\phi}_{n,u_1u_4}^{u_1u_2u_4}(t,r)dM_{n,u_4}(r)dM_{n,u_1}(t)\Bigg] \\
=&\int_0^T\IE\Bigg[\int_{t-2\delta_n}^{t-}\tilde{\phi}_{n,u_1u_2}^{u_1u_2u_4}(t,r)dM_{n,u_2}(r)\int_{t-2\delta_n}^{t-}\tilde{\phi}_{n,u_1u_4}^{u_1u_2u_4}(t,r')dM_{n,u_4}(r')C_{n,u_1}(t)\lambda_{n,u_1}(t)\Bigg]dt \\
=&\Ind(u_2=u_4)\int_0^T\int_{t-2\delta_n}^{t-}\IE\Bigg[\tilde{\phi}_{n,u_1u_2}^{u_1u_2}(t,r)^2C_{n,u_1}(t)\lambda_{n,u_1}(t)C_{n,u_2}(r)\lambda_{n,u_2}(r) \\
&\quad\quad\quad\times\Ind(u_2\in F_{u_1}(t-2\delta_n))]\Bigg]drdt \\
+&\int_0^T\IE\Bigg[\int_{t-2\delta_n}^{t-}\tilde{\phi}_{n,u_1u_2}^{u_1u_2u_4}(t,r)dM_{n,u_2}(r) \\
&\quad\quad\quad\times\int_{t-2\delta_n}^{t-}\tilde{\phi}_{n,u_1u_4}^{u_1u_2u_4}(t,r')dM_{n,u_4}(r')C_{n,u_1}(t)\lambda_{n,u_1}(t)\Ind(\neg u_2,u_4\in F_{u_1}(t-2\delta_n))\Bigg]dt 
\end{align*}
So we may rewrite
\begin{align*}
\eqref{eq:Ic}=&\frac{1}{r_n^2}\underset{u_1\neq u_2}{\sum_{u_1,u_2\in L_n}}\int_0^T\int_{t-2\delta_n}^{t-}\IE\Bigg[\tilde{\phi}_{n,u_1u_2}^{u_1u_2}(t,r)^2C_{n,u_1}(t)\lambda_{n,u_1}(t)C_{n,u_2}(r)\lambda_{n,u_2}(r) \\
&\quad\times\Ind(u_2\in F_{u_1}(t-2\delta_n))\Bigg]drdt \\
+&\frac{1}{r_n^2}\underset{u_1\neq u_2}{\sum_{u_1,u_2\in L_n}}\underset{u_4\neq u_2}{\sum_{u_4\in L_n}}\int_0^T\IE\Bigg[\int_{t-2\delta_n}^{t-}\tilde{\phi}_{n,u_1u_2}^{u_1u_2u_4}(t,r)dM_{n,u_2}(r) \\
&\quad\quad\quad\times\int_{t-2\delta_n}^{t-}\tilde{\phi}_{n,u_1u_4}^{u_1u_2u_4}(t,r')dM_{n,u_4}(r')C_{n,u_1}(t)\lambda_{n,u_1}(t)\Ind(\neg u_2,u_4\in F_{u_1}(t-2\delta_n))\Bigg]dt 
\end{align*}
By \eqref{eq:cond4} and \eqref{eq:cond5} we conclude $\eqref{eq:Ic}=o(1)$. Thus we have finally shown that $\eqref{eq:a2}\overset{\IP}{\to}0$ and hence the proof is complete.
\end{proof}

\begin{proof}[Proof of Lemma \ref{lem:network_exp}]
The proof of \eqref{eq:state1} is an immediate consequence of the following Proposition \ref{prop:network_bernstein} together with the assumptions:
\begin{eqnarray*}
&&\IP\left(\frac{1}{|E|_{n,t}}\sum_{i\in L_n}(Z_{n,i}-\IE(Z_{n,i}))\geq x\cdot\sqrt{\frac{\log|E|_{n,t}}{|E|_{n,t}}}\right) \\
&\leq&\IP\left(\sum_{i\in L_n}(Z_{n,i}-\IE(Z_{n,i}))\geq x\cdot\sqrt{\log|E|_{n,t}\cdot|E|_{n,t}}\right) \\
&\leq&\sum_{k=1}^\Kappa\exp\left(-\frac{c_2x^2\log|E|_{n,t}\cdot|E|_{n,t}}{2\left(|E|_{n,t}\sigma^2+E_k^{n,t}c_1x\cdot\sqrt{\log|E|_{n,t}\cdot|E|_{n,t}}\right)}\right)+\beta_t(\Delta_n)\cdot \Kappa r_n \\
&\leq&\Kappa\exp\left(-\frac{c_2x^2\log|E|_{n,t}}{2\left(\sigma^2+c_1c_3x\right)}\right)+\beta_t(\Delta_n)\cdot \Kappa r_n.
\end{eqnarray*}
\end{proof}

\begin{proposition}
\label{prop:network_bernstein}
Let $(Z_{n,ij})_{(i,j)\in L_n}$ be a set of random variables which fulfils \eqref{eq:GeneralCoverCondition} for a given $t\in[0,T]$. With the same notation as in Definition \ref{defin:network-beta} assume that there is a $\Delta$-partition such that for all $\rho\in\IN$ with $\rho\geq2$ and all $k\in\{1,...,\mathcal{K}\}$ and $m\in\{1,...,r_n\}$
$$\IE(|U_{k,m}^{n,t}(\Delta)|^{\rho})\leq\frac{\rho!}{2}E_{k,m}\sigma^2\cdot(E_kC)^{\rho-2},$$
for some numbers $\sigma^2, E_{k,m}, E_k$ and $C$ with $|E|_n:=\sum_{k=1}^{\Kappa}\sum_{m=1}^{r_n}E_{k,m}<+\infty$. Then,
$$\IP\left(\sum_{(i,j)\in L_n}(Z_{n,ij}-\IE(Z_{n,ij}))\geq x\right)\leq\sum_{k=1}^{\Kappa}\exp\left(-\frac{|E|_n^{-1}\sum_{m=1}^{r_n}E_{k,m}x^2}{2(|E|_n\sigma^2+E_kCx)}\right)+\beta_t(\Delta)\cdot \Kappa r_n.$$
\end{proposition}

\begin{proof}
With the definitions as in Definition \ref{defin:network-beta} we obtain by \eqref{eq:GeneralCoverCondition} that
$$Z_{n,ij}=\sum_k\sum_mI_{n,ij}^{k,m,t}(\Delta)Z_{n,ij}\textrm{ for all } (i,j)\in L_n.$$
Hence,
\begin{align}
\sum_{(i,j)\in L_n}(Z_{n,ij}-\IE Z_{n,ij})=&\sum_{k=1}^{\mathcal{K}}\sum_{m=1}^{r_n} \sum_{(i,j)\in L_n}Z_{n,ij} I_{n,ij}^{k,m,t}(\Delta)-\IE\left(Z_{n,ij}I_{n,ij}^{k,m,t}(\Delta)\right) \nonumber \\
=&\sum_{k=1}^{\mathcal{K}}\sum_{m=1}^{r_n}U_{k,m}^{n,t}(\Delta). \label{eq:sum_splitting}
\end{align}
In order to reduce notation, we omit $(\Delta)$ when talking about $U_{k,m}^{n,t}(\Delta)$. By Lemma \ref{lem:grouping_lemma} we can construct sequences $U_{k,m}^*$ as follows: We assume that the $\sigma$-field $\mathcal{F}_t^n$ is rich enough to allow for independent extra random variables $\delta_{k,m}$ which are uniformly distributed on $[0,1]$ and which are independent amongst each other and of everything else. The construction is the same for every $k$, so we only construct the sequence $U_{1,m}^*$, all other sequences $U_{k,m}^*$ for $k\geq2$ are constructed analogously. Define $U_{1,1}^*:=U_{1,1}$. For $m\geq2$ there is by Lemma \ref{lem:grouping_lemma} a function $f_m$ such that $U_{1,m}^*:=f_m(U_{1,1},...,U_{1,m-1},\delta_{1,m},U_{1,m})$ has the same distribution as $U_{1,m}$, is independent of $U_{1,1},...,U_{1,m-1}$ and
$$\IP(U_{1,m}\neq U_{1,m}^*)=\beta\left(\left(U_{1,1},...,U_{1,m-1}\right),U_{1,m}\right)\leq\beta_t(\Delta).$$
To sum it up, we have sequences $U_{k,m}^*$ with
\begin{enumerate}
\item For any $k$ and any fixed $R\in\IN$, $\left(U_{k,m}^*\right)_{m=1,...,R}$ is a sequence of independent random variables.
\item $U_{k,m}^*$ and $U_{k,m}$ have the same distribution.
\item For all $k=1,...,\mathcal{K}$: $\IP\left(\exists m\in\{1,...,r_n\}: U_{k,m}\neq U_{k,m}^*\right)\leq r_n\cdot\beta_t(\Delta)$.
\end{enumerate}
Denote by $R_k$ the random number of blocks $U_{k,m}$ of type $k$ which exist, i.e., such that for $m>R_k$ we have $U_{k,m}=0$. So we obtain by \eqref{eq:sum_splitting} for any $x\geq0$ and any sequence $(\alpha_k)_{k=1,...,\Kappa}$ with $\sum_{k=1}^{\Kappa}\alpha_k=1$ and $\alpha_k\geq0$:
\begin{eqnarray}
&&\IP\left(\sum_{(i,j)\in L_n}Z_{n,ij}-\IE(Z_{n,ij})\geq x\right) \nonumber \\
&\leq&\IP\left(\sum_{k=1}^{\Kappa}\sum_{m=1}^{R_k}U_{k,m}^*\geq x\right)+\IP\left(\exists k\in\{1,...,\Kappa\},m\in\{1,...,r_n\}:\, U_{k,m}\neq U_{k,m}^*\right) \nonumber \\
&\leq&\sum_{k=1}^{\Kappa}\IP\left(\sum_{m=1}^{R_k}U_{k,m}^*\geq\alpha_k\cdot x\right)+\beta_t(\Delta,v,v')\cdot \Kappa r_n \label{eq:ze}
\end{eqnarray}
For every $k$ the sequence $U_{k,m}^*$ is a sequence of independent random variables. Moreover, by definition $\IE(U_{k,m})=0$. So, the assumptions of Proposition \ref{prop:bernstein} are fulfilled with $\sigma_m^2:=E_{k,m}\sigma^2$ and $c:=E_kC$. So we can estimate the first part of \eqref{eq:ze} by
\begin{align}
\IP\left(\sum_{m=1}^{R_k}U_{k,m}^*\geq\alpha_k\cdot x\right)\leq&\exp\left(-\frac{\alpha_k^2x^2}{2\left(\sum_{m=1}^{r_n}E_{k,m}\sigma^2+E_k C\cdot\alpha_kx\right)}\right). \label{eq:ze2}
\end{align}
We chose $\alpha_k=|E|_n^{-1}\sum_{m=1}^{r_n}E_{k,m}$ and obtain by combining the equalities \eqref{eq:ze} and \eqref{eq:ze2},
\begin{align*}
\IP\left(\sum_{(i,j)\in L_n}(Z_{n,ij}-\IE(Z_{n,ij}))\geq x\right)\leq&\sum_{k=1}^{\Kappa}\exp\left(-\frac{\alpha_k^2x^2}{2\left(\sum_{m=1}^{r_n}E_{k,m}\sigma^2+E_kC\alpha_kx\right)}\right)+\beta_t(\Delta)\cdot \Kappa r_n \\
\leq&\sum_{k=1}^{\Kappa}\exp\left(-\frac{|E|_n^{-1}\sum_{m=1}^{r_n}E_{k,m}x^2}{2\left(|E|_n\sigma^2+E_kCx\right)}\right)+\beta_t(\Delta)\cdot \Kappa r_n
\end{align*}
\end{proof}

\begin{proof}[Proof of Lemma \ref{lem:sufficient_mixing}]
\label{subsubsec:pol}
Define $\epsilon:=x\cdot\sqrt{\frac{\log |E|_{n,t}}{|E|_{n,t}}}$. Then,
\begin{align}
&\IP\left(\frac{1}{|E|_{n,t}}\sum_{(i,j)\in L_n}\left[Z_{n,ij}-\IE(Z_{n,ij}(t))\right]\geq3\epsilon\right) \nonumber  \\
\leq&\IP\left(\frac{1}{|E|_{n,t}}\sum_{(i,j)\in L_n}\left[Z_{n,ij}(t)-\IE\left(Z_{n,ij}(t)|\Gamma_n^t=1\right)\right]\cdot \Gamma_n^t\geq\epsilon\right) \label{eq:exp1} \\
&\quad+\IP\left(\frac{1}{|E|_{n,t}}\sum_{(i,j)\in L_n}\left|\IE\left(Z_{n,ij}(t)|\Gamma_n^t=1\right)\right|\left(\Gamma_n^t-\IE(\Gamma_n^t)\right)\geq\epsilon\right) \label{eq:exp2} \\
&\quad+ \IP\left(\frac{1}{|E|_{n,t}}\sum_{(i,j)\in L_n}\IE\left(Z_{n,ij}(t)(\Gamma_n^t-1)\right)\geq\epsilon\right) \label{eq:exp3} \\
&\quad+\IP(\Gamma_n^t=0). \label{eq:exp4}
\end{align}
Line \eqref{eq:exp4} is part of the statement, so we just leave it as it is. For line \eqref{eq:exp3} we have
$$\left|\frac{r_n}{|E|_{n,t}}\IE(Z_{n,ij}(t)(\Gamma_n^t-1))\right|\leq M\IP(\Gamma_n^t=0)\cdot\frac{r_n}{|E|_{n,t}}.$$
Thus line \eqref{eq:exp3} equals zero by choice of $x$. For line \eqref{eq:exp2} we can make a similar argument
\begin{align*}
&\frac{r_n}{|E|_{n,t}}\left|\IE\left(Z_{n,ij}|\Gamma_n^t=1\right)\right|\left(\Gamma_n^t-\IE\left(\Gamma_n^t\right)\right)\geq\epsilon \\
\Rightarrow&\frac{Mr_n}{|E|_{n,t}}\left(\Gamma_n^t-\IE\left(\Gamma_n^t\right)\right)>M\IE\left(1-\Gamma_n^t\right)\frac{r_n}{|E|_{n,t}} \\
\Leftrightarrow&\Gamma_n^t>1.
\end{align*}
The last expression is a false statement and hence the first line cannot be true. Thus, $\eqref{eq:exp2}=0$. For line \eqref{eq:exp1} we apply Lemma \ref{lem:network_exp} to $Y_{n,ij}$ which is given in the statement of Lemma \ref{lem:sufficient_mixing}. We have
$$\left|\sum_{(i,j)\in L_n}Y_{n,ij}I_{n,ij}^{k,m,t}\right|\leq 2MS_k(t)\Gamma_n^t\leq 2Mc_3\sqrt{\frac{|E|_{n,t}}{\log(|E|_{n,t})}}=2M E_k^{n,t}.$$
In order to see that the conditions of Lemma \ref{lem:network_exp} hold, let $\rho\in\IN$ and greater or equal than two. Going on, we conclude for the grouping of $Y_{n,ij}$ by using the above estimation (recall that $\IE\left(U_{k,m}^{n,t}(\Delta_n)\right)=0$)
\begin{align*}
\IE\left(\left|U_{k,m}^{n,t}\right|^{\rho}\right)&= \IE\left(\left|\sum_{(i,j)\in L_n}Y_{n,ij}I_{n,ij}^{k,m,t}(\Delta_n)-\IE\left(\sum_{(i,j)\in L_n}Y_{n,ij} I_{n,ij}^{k,m,t}(\Delta_n)\right)\right|^{\rho-2}U_{k,m}^{n,t}(\Delta_n)^2\right) \\
&\leq\left(4ME_k^{n,t}\right)^{\rho-2}\cdot\textrm{Var}\left(U_{k,m}^{n,t}(\Delta_n)\right).
\end{align*}
Moreover, by assumption
\begin{align*}
&\frac{1}{E_{k,m}^{n,t}}\textrm{Var}\left(U_{k,m}^{n,t}(\Delta_n)\right) \\
\leq&\frac{1}{E_{k,m}^{n,t}}\sum_{(i,j)\in L_n}\textrm{Var}\left(Y_{n,ij}I_{n,ij}^{k,m,t}(\Delta_n)\right) \\
&\quad+\frac{1}{E_{k,m}^{n,t}}\underset{(i,j)\neq (k,l)}{\sum_{(i,j),(k,l)\in L_n}}\textrm{Cov}\left(Y_{n,ij}I_{n,ij}^{k,m,t}(\Delta_n),Y_{n,kl}I_{n,kl}^{k,m,t}(\Delta_n)\right) \\
\leq& 3CM^2.
\end{align*}
Thus the first requirement of Lemma \ref{lem:network_exp} holds for the definitions in the statement of this Lemma and $\sigma^2=3CM^2$ and $c_1=4M$. The first part of the second condition in Lemma \ref{lem:network_exp} holds by assumption and the second part holds by definition of $E_k^{n,t}$. Thus, we may apply Lemma \ref{lem:network_exp} and obtain for \eqref{eq:exp1}
\begin{align*}
&\IP\left(\frac{1}{|E|_{n,t}}\sum_{(i,j)\in L_n}Y_{n,ij}\geq x\cdot\sqrt{\frac{\log |E|_{n,t}}{|E|_{n,t}}}\right) \\
\leq&\mathcal{K}(|E|_{n,t})^{-\frac{c_2\cdot x^2}{2(3CM^2+4Mc_3x)}}+\beta_t(\Delta_n)\cdot\mathcal{K}r_n.
\end{align*}
This yields the statement.
\end{proof}

\subsection{Proof of Theorem \ref{thm:test_asymptotics}}
\label{subsec:proof_test}
Recall that $M_{n,ij}(t):=N_{n,ij}(t)-\int_0^t\lambda_{n,ij}(s)ds$ denotes the counting process martingale and decompose the likelihood as follows:
\begin{align}
P_n(\theta,t_0)&:=\frac{1}{r_np_n(t_0)}\sum_{(i,j)\in L_n}\int_0^TK_{h,t_0}(s)C_{n,ij}(s)\left[\log\lambda_{n,ij}(\theta,s)\cdot\lambda_{n,ij}(\theta_0,s)-\lambda_{n,ij}(\theta,s)\right]ds \label{eq:def_P} \\
\frac{\ell_n(\theta,t_0)}{r_np_n(t_0)}&=\frac{1}{r_np_n(t_0)}\sum_{(i,j)\in L_n}\int_0^TK_{h,t_0}(s)\log\lambda_{n,ij}(\theta,s)dM_{n,ij}(s)+P_n(\theta,t). \label{eq:decomp_ell}
\end{align}
Define moreover for $(i,j),(k,l)\in L_n$ and $s,t\in[0,T]$ (we use the convention $\Sigma_{t}^{-T}:=\left(\Sigma_t^{-1}\right)^T$)
\begin{equation}
\label{eq:deff}
f_{n,ij,kl}(s,t):=X_{n,ij}(s)^T\int_0^ThK_{h,t_0}(s)K_{h,t_0}(t)\Sigma_{t_0}^{-T}\Sigma_{t_0}^{-1}\frac{w(t_0)}{\bar{p}_n(t_0)}dt_0X_{n,kl}(t).
\end{equation}
We will also need the following functions defined for all $(i,j), (k,l)\in L_n$
\begin{align}
H_{n,ij}(s,\theta)&:=\left[-\partial_{\theta}^2\lambda_{n,ij}(s,\theta)\right]C_{n,ij}(s), \label{eq:def_H} \\
\tilde{H}_{n,ij}(s,\theta)&:=H_{n,ij}(s,\theta)-p_n(s)\Sigma(s,\theta), \label{eq:def_Htilde} \\
\tau_{n,ij,kl}(s)&:=\int_0^{s-}f_{n,ij,kl}(s,t)dM_{n,kl}(t), \label{eq:def_tau}
\end{align}
where $\int_0^{s-}$ denotes the integral over the set $[0,s)$. Most technical difficulties are contained in the proofs of the following Lemmas \ref{lem_ass:var}-\ref{lem_ass:exp2}. Their proofs are presented in Appendix \ref{sec:lemma_proofs}.

\begin{lemma}
\label{lem_ass:var}
It holds that
\begin{equation}
\label{eq:var1}
\frac{4}{hr_n^2}\sum_{(i,j)\in L_n}\int_0^T\underset{(k,l)\neq (i,j)}{\sum_{(k,l)\in L_n}}\tau_{n,ij,kl}(s)^2C_{n,ij}(s)\lambda_{n,ij}(s,\theta_0)ds \overset{\IP}{\to}B
\end{equation}
and
\begin{equation}
\label{eq:var2}
\frac{4}{hr_n^2}\sum_{(i,j)\in L_n}\int_0^T\underset{(k_1,l_1)\neq (k_2,l_2)}{\underset{(k_1,l_1),(k_2,l_2)\neq (i,j)}{\sum_{(k_1,l_1),(k_2,l_2)\in L_n}}}\tau_{n,ij,k_1l_1}(s)\tau_{n,ij,k_2l_2}(s)C_{n,ij}(s)\lambda_{n,ij}(s,\theta_0)ds\overset{\IP}{\to}0.
\end{equation}
The definition of $B$ is given in Theorem \ref{thm:test_asymptotics}.
\end{lemma}

\begin{lemma}
\label{lem_ass:jump}
For any $\epsilon>0$
$$\frac{2}{h^{\frac{1}{2}}r_n}\sum_{(i,j)\in L_n}\int_0^T\Ind\left(\left|\frac{2}{h^{\frac{1}{2}}r_n}\underset{(k,l)\neq (i,j)}{\sum_{(k,l)\in L_n}}\tau_{n,ij,kl}(s)\right|>\epsilon\right)\underset{(k,l)\neq (i,j)}{\sum_{(k,l)\in L_n}}\tau_{n,ij,kl}(s)dM_{n,ij}(s)\overset{\IP}{\to}0$$
\end{lemma}

\begin{lemma}
\label{lem_ass:K1}
There is a sequence $B_n$ with $B_n=O_P(1)$, such that for all $t_0\in\IT$
$$\left\|\left[\frac{1}{r_n\bar{p}_n(t_0)}\partial_{\theta}^2\ell_n(\theta_0;t_0)\right]^{-1}\right\|\leq B_n.$$
\end{lemma}

\begin{lemma}
\label{lem_ass:K2}
There is a sequence $K_n$ with $K_n=O_P(1)$ such that for all $\theta_1,\theta_2$ and $t\in\IT$
$$\frac{1}{r_n\bar{p}_n(t_0)}\left\|\partial_{\theta}^2\ell_n(\theta_1;t)-\partial_{\theta}^2\ell_n(\theta_2;t)\right\|\leq K_n\cdot\|\theta_1-\theta_2\|.$$
\end{lemma}

\begin{lemma}
\label{lem_ass:exp1}
Denote by $T_{n,k}$ for $n,k\in\IN$ the grid
\begin{equation}
\label{eq:grid1}
T_{n,k}:=\left\{\frac{j}{hn^k}:j\in\IN, \frac{j}{hn^k}\in[0,T]\right\}.
\end{equation}
Then, for any $k_0$ there is $C>0$ such that
$$\sup_{t_0\in T_{n,k_0}}\IP\left(\left\|\frac{\partial_{\theta}\ell_n(\theta_0;t_0)}{r_n\sqrt{\bar{p}_n(t_0)}}\right\|\geq C\sqrt{\frac{\log r_n}{r_nh}}\right)=o\left(h^{-1}n^{-k_0}\right).$$
\end{lemma}

\begin{lemma}
\label{lem_ass:exp2}
Define for $k,n\in\IN$ the grid
\begin{equation}
\label{eq:grid2}
T_{n,k}:=\left\{\left(\frac{j}{hn^{k}},\frac{j_1}{n^{k}},...,\frac{j_p}{n^{k}}\right)\in \IT\times\Theta:\,j,j_1,...,j_p\in\IZ\right\}.
\end{equation}
Then, for any $k_0\in\IN$, there is $C>0$ such that
\begin{eqnarray*}
&&\IP\left(\sup_{(t_0,\theta)\in T_{n,k_0}}\left|\frac{1}{r_n\bar{p}_n(t_0)}\sum_{(i,j)\in L_n}\int_0^TK_{h,t_0}\left(s\right)\tilde{H}_{n,ij}(s,\theta)ds\right|>C\cdot\sqrt{\frac{\log r_np_n}{r_np_n\cdot h}}\right)\to0.
\end{eqnarray*}
\end{lemma}

The above lemmas hold under the assumptions in Theorem \ref{thm:test_asymptotics}. Therefore, we can use all their statements in the following. We begin the proof of Theorem \ref{thm:test_asymptotics} by showing the following small lemmas.
\begin{lemma}
\label{lem:p_continuous}
Suppose (A4, \ref{ass:kernel_hoelder}) holds and that $p_n>0$. Let $\alpha_p:=\alpha_K$. Then it holds for any $t_0,t_1\in[h,T-h]$ and all $n\in\IN$ that
$$\left|\frac{1}{\bar{p}_n(t_0)}-\frac{1}{\bar{p}_n(t_1)}\right|\leq H_{n,p}\cdot|t_0-t_1|^{\alpha_p},\textrm{ where }H_{n,p}:=\frac{4H_K\max_{s\in[0,t]}p_n(s)}{h^{\alpha_K}p_n^2}$$
Suppose that, in addition, (A6) holds. Then,
$$\int_{-2}^2\frac{\bar{p}_n(t)}{\bar{p}_n(t-hv)}dv\textrm{ and }\frac{p_n(t)}{\bar{p}_n(t)}$$
are uniformly bounded in $n$ and $t$.
\end{lemma}
\begin{proof}
The proof is just a direct calculation: Note that $\bar{p}_n(t)\geq p_n$ for $t\in[h,T-h]$ and hence
\begin{align*}
\left|\frac{1}{\bar{p}_n(t_0)}-\frac{1}{\bar{p}_n(t_1)}\right|&\leq\frac{1}{p_n^2}\left|\bar{p}_n(t_1)-\bar{p}_n(t_0)\right| \leq\frac{1}{hp_n^2}\int_0^Tp_n(s)\left|K\left(\frac{t_0-s}{h}\right)-K\left(\frac{t_1-s}{h}\right)\right|ds \\
&\leq\frac{4H_K\max_{s\in[0,t]}p_n(s)}{h^{\alpha_K}p_n^2}|t_0-t_1|^{\alpha_K}.
\end{align*}
The second statement is now a direct consequence by noting that for $v\in[-2,2]$
$$\frac{\bar{p}_n(t)}{\bar{p}_n(t-hv)}\leq2^{\alpha_p}\bar{p}_n(t)h^{\alpha_p}H_{n,p}+1.$$
The right hand side is uniformly bounded under (A6). The boundedness of $\frac{p_n(t)}{\bar{p}_n(t)}$ is also a direct consequence of (A6).
\end{proof}

\begin{lemma}
\label{lem:bounded_inverse}
Suppose Assumption (A5) holds. There exist $M,\rho\in(0,\infty)$ such that for all $t$ and all matrices $X\in\IR^{p\times p}$ with $\|\Sigma(t,\theta_0)-X\|<\rho$ it holds that $X$ is invertible and $\|X^{-1}\|<M$.
\end{lemma}

\begin{proof}
We begin by showing that
\begin{equation}
\label{eq:rho_existence}
\exists\rho>0\,\forall t\in[0,T]\forall X\in\IR^{p\times p}, \|X-\Sigma(t,\theta_0)\|<\rho:\quad X\textrm{ is invertible}.
\end{equation}
Define $\rho_n:=\frac{1}{n}$ and suppose the statement was wrong. Then, we find for all $n\in\IN$ numbers $t_n\in[0,T]$ and matrices $X_n\in\IR^{p\times p}$ such that $\|X_n-\Sigma(t_n,\theta_0)\|<\rho_n$ but $X_n$ is not invertible. Since $(t_n)_{n\in\IN}\subseteq[0,T]$ and $[0,T]$ is compact, there is a subsequence $(t_{n_k})_{k\in\IN}$ such that $t_{n_k}\to t_0\in[0,T]$ for $k\to\infty$. By continuity of $\Sigma(t,\theta_0)$ in $t_0$ we conclude that
$$\|X_{n_k}-\Sigma(t_0,\theta_0)\|\leq\|X_{n_k}-\Sigma(t_{n_k},\theta_0)\|+\|\Sigma(t_{n_k},\theta_0)-\Sigma(t_0,\theta_0)\|\to0$$
and hence $X_{n_k}\to\Sigma(t_0,\theta_0)$ for $k\to\infty$. Note finally that the space of non-invertible matrices is given by $\textrm{det}^{-1}(\{0\})$. Since $\textrm{det}:\IR^{p\times p}\to\IR$ is continuous, the set of non-invertible matrices is closed. By construction $X_{n_k}$ is non-invertible and hence $\Sigma(t_0,\theta_0)$ is non-invertible, too. This is a clear contradiction to (A5).

In order to find $M>0$ choose $\rho$ in \eqref{eq:rho_existence} such that $\rho\cdot\sup_{t\in[0,T]}\|\Sigma(t,\theta_0)^{-1}\|\leq\frac{1}{2}$. This is possible because inverting a matrix is a continuous operation and by continuity of $t\mapsto\Sigma(t,\theta_0)$ as in (A5). Let now $t$ and $X$ be as in \eqref{eq:rho_existence}. By using the fact that the spectral-norm of a matrix is sub-multiplicative, we find
$$\|X^{-1}\|\leq\|X^{-1}\|\cdot\|X-\Sigma(t,\theta_0)\|\cdot\|\Sigma(t,\theta_0)^{-1}\|+\|\Sigma(t,\theta_0)^{-1}\|\leq\frac{1}{2}\|X^{-1}\|+\|\Sigma(t,\theta_0)^{-1}\|.$$
Hence, $\|X^{-1}\|\leq2\sup_{t\in[0,T]}\|\Sigma(t,\theta_0)^{-1}\|=:M<\infty$.
\end{proof}

\begin{lemma}
\label{lem:diff_integral}
Under (A3, \ref{ass:boundedness}), the functions $\ell_n(\theta,t)$ and $P_n(\theta,t)$ are twice differentiable with respect to $\theta$ and the derivatives can be computed by interchanging integral and differential.
\end{lemma}

\begin{proof}
The integral with respect to $M_{n,ij}$ can be split in an integral with respect to $N_{n,ij}$ (which is a sum) and a regular Lebesgue integral. Therefore, the stochastic integration is not inducing additional difficulties and we can apply standard theory for Lebesgue integration. The integrands are clearly differentiable with respect to $\theta$. Boundedness of the covariates guarantees that the derivatives can be bounded by an integrable function (which does not depend on $\theta$). Then the integral and derivative may be interchanged.
\end{proof}

\begin{lemma}
\label{lem:Fubini}
Under Assumptions (A4, \ref{ass:kernel_hoelder}), (A5) and (A3, \ref{ass:boundedness}) and $p_n>0$ we have that for any $(i,j),(k,l)\in L_n$ and any $r\in\{1,...,p\}$ the order of integration in the following integrals can be interchanged
\begin{align*}
&\int_{[0,T]^3}K_{h,t'}\left(s\right)K_{h,t'}\left(t\right)\frac{\partial_{\theta}\lambda_{n,ij}(s,\theta_0)^T}{\lambda_{n,ij}(s,\theta_0)}\Sigma_{t'}^{-T}\Sigma_{t'}^{-1}\frac{\partial_{\theta}\lambda_{n,kl}(t,\theta_0)}{\lambda_{n,kl}(t,\theta_0)}\frac{w(t')}{\bar{p}_n(t')}dt'dM_{n,ij}(s)dM_{n,kl}(t), \\
&\int_{[0,T]^2}K_{h,t_0}\left(\Sigma_{t_0}^{-1}\right)_{r\cdot}\partial_{\theta}\log\lambda_{n,ij}(\theta_0,t)w(t_0)dM_{n,ij}(t)dt_0.
\end{align*}
\end{lemma}

\begin{proof}
Note that similar to the proof of Lemma \ref{lem:diff_integral}, the integrals with respect to the martingales $M_{n,ij}$ can be split into two integrals. The integral with respect to the counting process is a sum and hence it is clear that the order of integration can be interchanged. For the other (Lebesgue) integrals we can apply Fubini's Theorem: We show that the iterated integrals exist even after taking the norm within the integral. For both iterated integrals we may apply Lemma \ref{lem:bounded_inverse} in order to remove $\Sigma$ from the consideration. Then, the innermost integral is in both cases an integral over the kernel, the weight function $w$ and in case of the first iterated integral of $\bar{p}_n(t)$. All these functions are bounded (cf Assumptions (A1), (A4, \ref{ass:kernel_hoelder}) and Lemma \ref{lem:p_continuous}) and hence the innermost integral can just be bounded by a constant. The outer integrals are now integrals over $\left\|\partial_{\theta}\lambda_{n,i}(t,\theta_0)\right\|$ or $\left\|\partial_{\theta}\lambda_{n,i}(t,\theta_0)\cdot\lambda_{n,i}(t,\theta_0)\right\|$, both of which are integrable by Assumption (A3, \ref{ass:boundedness}).
\end{proof}

\begin{lemma}
\label{lem:sigma_diff}
Suppose that (A2) and (A3, \ref{ass:boundedness}) hold true. Then, there is $\gamma_{\Sigma}:[0,T]\to(0,\infty)$ such that $\left\|\Sigma(s,\theta_1)-\Sigma(s,\theta_2)\right\|\leq\gamma_{\Sigma}(s)\|\theta_1-\theta_2\|$ for all $s\in[0,T]$ and all $\theta_1,\theta_2\in\Theta$, i.e., $\theta\mapsto\Sigma(t,\theta)$ is Lipschitz continuous in $\theta$ for every fixed $t$. Additionally,
$$\sup_{t\in[0,T]}\int_0^TK_{h,t}\left(s\right)\frac{p_n(s)}{\bar{p}_n(t)}\gamma_{\Sigma}(s)ds=O_P(1).$$
\end{lemma}
\begin{proof}
The proof is immediate since the covariates are bounded by Assumption (A3, \ref{ass:boundedness}) and the parameter space $\Theta$ is bounded by Assumption (A2).
\end{proof}

\begin{lemma}
\label{lem:bounded_g}
Suppose that (A1), (A3, \ref{ass:boundedness}), (A4), (A5) and (A6) hold. For $g_{n,ij}(s):=h^{-\frac{1}{2}}\int_{[0,s)}f_{n,ij,ij}(s,t)dM_{n,ij}(t)$, we have for $n\to\infty$
$$\frac{1}{r_n}\IE\left(\int_0^Tg_{n,ij}(s)^2C_{n,ij}(s)\lambda_{n,ij}(s,\theta_0)ds\right)\to0.$$
\end{lemma}
\begin{proof}
We use the bounds from (A1), (A3, \ref{ass:boundedness}) and Lemma \ref{lem:bounded_inverse} as well as the kernel properties (A4, \ref{ass:kernel_hoelder}) to obtain
$$|f_{n,ij,ij}(s,t)|\leq\frac{K\hat{K}^2\|w\|_{\infty}}{p_n}\Ind(|s-t|\leq 2h).$$
Using this estimate we can bound (denote $C:=K\hat{K}^2\|w\|_{\infty}$)
\begin{align*}
&\IE\left(\int_0^Tg_{n,ij}^2C_{n,ij}(s)\lambda_{n,ij}(s,\theta_0)ds\right)\leq\frac{\Lambda}{h}\int_0^T\IE\left(\left(\int_{[0,s)}f_{n,ij,ij}(s,t)dM_{n,ij}(t)\right)^2\right)ds \\
\leq&\frac{C^2\Lambda}{hp_n^2}\int_0^T\int_0^s\Ind(|s-t|\leq 2h)p_n(t)dtds\leq\frac{2C^2T\Lambda}{hp_n}\frac{\max_{s\in[0,T]}p_n(s)}{p_n}.
\end{align*}
The statement follows now by using the properties of $p_n(t)$ in (A6) and the bandwidth $h$ in (A4, \ref{ass:bw}).
\end{proof}

\begin{lemma}
Suppose that (H2) holds. Then, we have for
\begin{align}
&\sup_{n\in\IN}\sup_{(i,j)\in L_n}\sup_{t\in[0,T]}\frac{1}{r_np_n}\IE\left(\left(K_m^{L_n}\right)^4H_{UB}^{ij}\Big|C_{n,ij}(t)=1\right)<\infty \label{eq:HubS5}
\end{align}
\end{lemma}
\begin{proof}
Follows by applying (H2, \ref{eq:HubS7}).
\end{proof}

We continue with three more involved propositions. It is through these propositions how dependence structures enter the proof of Theorem \ref{thm:test_asymptotics}.
\begin{proposition}
\label{lem:1}
Under the same assumptions as in Theorem \ref{thm:test_asymptotics} we have
$$\sup_{t_0\in\IT}\left\|\frac{1}{r_n\sqrt{\bar{p}_n(t_0)}}\partial_{\theta}\ell_n(\theta_0,t_0)\right\|=O_P\left(\sqrt{\frac{\log r_n}{r_nh}}\right)$$
\end{proposition}

\begin{proof}
We note firstly that existence of the derivative of $\ell_n$ is ensured by Lemma \ref{lem:diff_integral} and we can compute the derivative by taking the derivative under the integral sign. Let $\delta_n:=\sqrt{\frac{\log r_n}{r_nh}}$ and recall the definition of the grid $T_{n,k}$ in \eqref{eq:grid1}. Denote by $\pi_{n,k}:[0,T]\to T_{n,k}$ the corresponding projection on $T_{n,k}$. Then $\#T_{n,k}\leq (T+1)\cdot hn^k$ and $|t-\pi_{n,k}(t)|\leq\frac{1}{hn^k}$. Using this projection we can estimate for $C>0$
\begin{align}
&\IP\left(\left\|\sup_{t_0\in\IT}\frac{1}{r_n\sqrt{\bar{p}_n(t_0)}}\partial_{\theta}\ell_n(\theta_0,t_0)\right\|\geq C\delta_n\right) \nonumber \\
\leq&\IP\Bigg(\sup_{t_0\in\IT}\left\|\frac{\partial_{\theta}\ell_n(\theta_0,t_0)}{r_n\sqrt{\bar{p}_n(t_0)}}-\frac{\partial_{\theta}\ell_n(\theta_0,\pi_{n,k}(t_0))}{r_n\sqrt{\bar{p}_n(\pi_{n,k}(t_0))}}\right\|+\sup_{t_0\in\IT}\left\|\frac{\partial_{\theta}\ell_n(\theta_0,\pi_{n,k}(t_0))}{r_n\sqrt{\bar{p}_n(\pi_{n,k}(t_0))}}\right\|\geq C\delta_n\Bigg) \nonumber \\
\leq&\IP\left(\underset{|s_0-t_0|\leq hn^{-k}}{\sup_{t_0,s_0\in\IT}}\left\|\frac{\partial_{\theta}\ell_n(\theta_0,t_0)}{r_n\sqrt{\bar{p}_n(t_0)}}-\frac{\partial_{\theta}\ell_n(\theta_0,s_0)}{r_n\sqrt{\bar{p}_n(s_0)}}\right\|\geq \frac{C}{2}\delta_n\right) \label{eq:P1} \\
&\quad+\IP\left(\sup_{t_0\in T_{n,k}}\left\|\frac{\partial_{\theta}\ell_n(\theta_0,t_0)}{r_n\sqrt{\bar{p}_n(t_0)}}\right\|\geq \frac{C}{2}\delta_n\right). \label{eq:P2}
\end{align}
We have to prove that both \eqref{eq:P1} and \eqref{eq:P2} converge to zero. We start with \eqref{eq:P1}. Denote therefore $g_{n,ij}(t,t_0)=hK_{h,t_0}\left(t\right)\partial_{\theta}\log\lambda_{n,i}(\theta_0,t)$, then
\begin{equation}
\label{eq:deriv_decomp}
\partial_{\theta}\ell_n(\theta_0,t_0)=\sum_{(i,j)\in L_n}\int_0^Th^{-1}g_{n,ij}(t,t_0)dM_{n,ij}(t)
\end{equation}
because $P_n'(\theta_0,t_0)=0$. Then we get
\begin{align}
&\eqref{eq:P1} \nonumber \\
\leq&\IP\left(\underset{|s_0-t_0|\leq hn^{-k}}{\sup_{t_0,s_0\in\IT}}\frac{1}{r_nh}\sum_{(i,j)\in L_n}\left\|\int_0^T\frac{g_{n,ij}(t,t_0)}{\sqrt{\bar{p}_n(t_0)}}dM_{n,ij}(t)-\int_0^T\frac{g_{n,ij}(t,s_0)}{\sqrt{\bar{p}_n(s_0)}}dM_{n,ij}(t)\right\|\geq\frac{C}{2}\delta_n\right) \nonumber \\
\leq&\IP\left(\frac{1}{r_nh}\sum_{(i,j)\in L_n}\int_0^T\underset{|s_0-t_0|\leq hn^{-k}}{\sup_{t_0,s_0\in\IT}}\left\|\frac{g_{n,ij}(t,t_0)}{\sqrt{\bar{p}_n(t_0)}}-\frac{g_{n,ij}(t,s_0)}{\sqrt{\bar{p}_n(s_0)}}\right\|dN_{n,ij}(t)\geq\frac{C}{4}\delta_n\right) \label{eq:part1} \\
+&\IP\left(\frac{1}{r_nh}\sum_{(i,j)\in L_n}\int_0^T\underset{|s_0-t_0|\leq hn^{-k}}{\sup_{t_0,s_0\in\IT}}\left\|\frac{g_{n,ij}(t,t_0)}{\sqrt{\bar{p}_n(t_0)}}-\frac{g_{n,ij}(t,s_0)}{\sqrt{\bar{p}_n(s_0)}}\right\|C_{n,ij}(t)\lambda_{n,ij}(\theta_0,t)dt\geq\frac{C}{4}\delta_n\right). \label{eq:part2_test}
\end{align}
For \eqref{eq:part1} we apply Lenglart's inequality (cf. Lemma \ref{lem:Lenglart} in the Appendix) to obtain for any choice of $c^*>0$
\begin{align*}
&\IP\left(\frac{1}{r_nh}\sum_{(i,j)\in L_n}\int_0^T\underset{|s_0-t_0|\leq hn^{-k}}{\sup_{t_0,s_0\in\IT}}\left\|\frac{g_{n,ij}(t,t_0)}{\sqrt{\bar{p}_n(t_0)}}-\frac{g_{n,ij}(t,s_0)}{\sqrt{\bar{p}_n(s_0)}}\right\|dN_{n,ij}(t)\geq\frac{C}{4}\delta_n\right) \\
\leq&\frac{c^*}{C}+\IP\left(\frac{1}{r_nh}\sum_{(i,j)\in L_n}\int_0^T\underset{|s_0-t_0|\leq hn^{-k}}{\sup_{t_0,s_0\in\IT}}\left\|\frac{g_{n,ij}(t,t_0)}{\sqrt{\bar{p}_n(t_0)}}-\frac{g_{n,ij}(t,s_0)}{\sqrt{\bar{p}_n(s_0)}}\right\|C_{n,ij}(t)\lambda_{n,ij}(\theta_0,t)dt\geq\frac{c^*}{4}\delta_n\right).
\end{align*}
If we restrict to $c^*<C$ we obtain furthermore
\begin{align}
&\eqref{eq:part1}+\eqref{eq:part2_test} \nonumber \\
\leq&\frac{c^*}{C}+2\IP\left(\frac{1}{r_nh}\sum_{(i,j)\in L_n}\int_0^T\underset{|s_0-t_0|\leq hn^{-k}}{\sup_{t_0,s_0\in\IT}}\left\|\frac{g_{n,ij}(t,t_0)}{\sqrt{\bar{p}_n(t_0)}}-\frac{g_{n,ij}(t,s_0)}{\sqrt{\bar{p}_n(s_0)}}\right\|C_{n,ij}(t)\lambda_{n,ij}(\theta_0,t)dt\geq\frac{c^*}{4}\delta_n\right). \label{eq:1415}
\end{align}
Since for any $x,y\geq0$, $|\sqrt{x}-\sqrt{y}|\leq\sqrt{|x-y|}$, Lemma \ref{lem:p_continuous} implies that $\sqrt{\frac{1}{\bar{p}_n(t_0)}}$ is Hoelder continuous with exponent $\frac{\alpha_p}{2}$ and constant $\sqrt{H_{n,p}}$. Moreover, we have $\sup_{t_0\in\mathbb{T}}\frac{1}{\sqrt{\bar{p}_n(t_0)}}\leq\frac{1}{\sqrt{p_n}}$ and Hoelder continuity of the kernel K by Assumption (A4, \ref{ass:kernel_hoelder}) (we denote the bound on the kernel also by $K$). Combining all these, we obtain for $|t_0-s_0|\leq hn^{-k}$
\begin{align*}
&\left\|\frac{g_{n,ij}(t,t_0)}{\sqrt{\bar{p}_n(t_0)}}-\frac{g_{n,ij}(t,s_0)}{\sqrt{\bar{p}_n(s_0)}}\right\| \\
=&\left\|\partial_{\theta}\log\lambda_{n,ij}(\theta_0,t)\right\|\cdot h\left|\bar{p}_n(t_0)^{-\frac{1}{2}}K_{h,t_0}\left(t\right)-\bar{p}_n(s_0)^{-\frac{1}{2}}K_{h,s_0}\left(t\right)\right| \\
\leq&\left\|\partial_{\theta}\log\lambda_{n,ij}(\theta_0,t)\right\|\cdot h\left[\bar{p}_n(t_0)^{-\frac{1}{2}}\left|K_{h,t_0}\left(t\right)-K_{h,s_0}\left(t\right)\right|+K_{h,s_0}\left(t\right)\left|\bar{p}_n(t_0)^{-\frac{1}{2}}-\bar{p}_n(s_0)^{-\frac{1}{2}}\right|\right] \\
\leq& \|\partial_{\theta}\log\lambda_{n,ij}(\theta_0,t)\|\left[p_n^{-\frac{1}{2}}\cdot H_K n^{-k\alpha_K}+K\cdot\sqrt{H_{n,p}}n^{-k\cdot\frac{\alpha_p}{2}}h^{\frac{\alpha_p}{2}}\right] \\
\leq&\|\partial_{\theta}\log\lambda_{n,ij}(\theta_0,t)\|\cdot p_n^{-\frac{1}{2}}\left[H_K+K\sqrt{H_{n,p}p_nh^{\alpha_p}}\right]n^{-k\cdot\alpha_p/2},
\end{align*}
since $\alpha_p=\alpha_K$ by Lemma \ref{lem:p_continuous}. So we get
\begin{align*}
&\frac{1}{r_nh}\sum_{(i,j)\in L_n}\int_0^T\underset{|s_0-t_0|\leq hn^{-k}}{\sup_{t_0,s_0\in\IT}}\left\|\frac{g_{n,ij}(t,t_0)}{\sqrt{\bar{p}_n(t_0)}}-\frac{g_{n,ij}(t,s_0)}{\sqrt{\bar{p}_n(s_0)}}\right\|C_{n,ij}(t)\lambda_{n,ij}(\theta_0,t)dt \\
\leq&\frac{1}{r_nh}\sum_{(i,j)\in L_n}\int_0^T\|\partial_{\theta}\log\lambda_{n,ij}(\theta_0,t)\|C_{n,ij}(t)\lambda_{n,ij}(\theta_0,t)dt \\
&\quad\quad\quad\times p_n^{-\frac{1}{2}}\left[H_K+K\sqrt{H_{n,p}p_nh^{\alpha_p}}\right]n^{-k\cdot\alpha_p/2} \\
&=O_P\left(h^{-1}p_n^{-1/2}n^{-k\alpha_p/2}\right).
\end{align*}
Since by definition of $H_{n,p}$ in Lemma \ref{lem:p_continuous}, we have $H_{n,p}p_nh^{\alpha_p}=O(1)$ and the covariates are bounded by (A3, \ref{ass:boundedness}). Hence, we get that \eqref{eq:1415} is small, because we can choose $k=k_0$ such that for large enough $c^*$ the probability is small for all $n\in\IN$ and then we can choose $C$ large enough such that the whole expression is small. Then, also \eqref{eq:P1} is small, for this good choice $k=k_0$ which we keep fixed from now on.

Let us now turn to \eqref{eq:P2}. Here we take the supremum over a finite set and so we can estimate by applying union bound and Lemma \ref{lem_ass:exp1} for $C>0$ large enough
\begin{align*}
&\eqref{eq:P2} \\
\leq&\# T_{n,k_0}\,\cdot\,\sup_{t_0\in T_{n,k_0}}\IP\left(\left\|\frac{\partial_{\theta}\ell_n(\theta_0,t_0)}{r_n\sqrt{\bar{p}_n(t_0)}}\right\|\geq\frac{C}{2}\delta_n\right)\to 0.
\end{align*}
\end{proof}

Having established Proposition \ref{lem:1}, we can quickly show the following result.

\begin{lemma}
\label{lem:aux1}
Under the same assumptions as in Theorem \ref{thm:test_asymptotics} we have
$$\sup_{t_0\in\IT}\left\|\theta_0-\hat{\theta}_n(t_0)\right\|=O_p\left(\sqrt{\frac{\log r_n}{r_np_nh}}\right).$$
\end{lemma}

\begin{proof}
By Lemmas \ref{lem_ass:K1} and \ref{lem_ass:K2} we have that for any choice of $t_0\in\IT$
\begin{align*}
\left\|\left[\frac{1}{r_n\bar{p}_n(t_0)}\partial_{\theta}^2\ell_n(\theta_0,t_0)\right]^{-1}\right\|\leq& B_n \\
\frac{1}{r_n\bar{p}_n(t_0)}\left\|\partial_{\theta}^2\ell_n(\theta_1,t_0)-\partial_{\theta}^2\ell_n(\theta_2,t_0)\right\|\leq&K_n\cdot\|\theta_1-\theta_2\|,
\end{align*}
where $B_n, K_n=O_P(1)$. Thus, we find by Proposition \ref{lem:1} that
$$\eta_n:=\sup_{t_0\in\IT}\left\|\left[\frac{\partial_{\theta}^2\ell_n(\theta_0,t_0)}{r_n\bar{p}_n(t_0)}\right]^{-1}\frac{\partial_{\theta}\ell_n(\theta_0,t_0)}{r_n\bar{p}_n(t_0)}\right\|\leq\frac{B_n}{\sqrt{p_n}}\sup_{t_0\in\IT}\left\|\frac{\partial_{\theta}\ell_n(\theta_0,t_0)}{r_n\sqrt{p_n(t_0)}}\right\|=O_P\left(\sqrt{\frac{\log r_n}{p_nr_nh}}\right).$$
Hence, we can apply Kantorovich's Theorem (cf. Theorem \ref{thm:kantorovich}) for all $t_0\in\IT$ with the same choice of $B_n, K_n$ and $\eta_n$ as above. Thus, there is $\hat{\theta}_n(t_0)$ such that for all $t_0$
$$\left\|\theta_0-\hat{\theta}_n(t_0)\right\|\leq2\eta_n=O_P\left(\sqrt{\frac{\log r_n}{r_np_nh}}\right).$$
\end{proof}

\begin{corollary}
\label{cor:theta_interior}
The probability of the event \emph{for all $t_0\in\IT$ it holds that $\hat{\theta}_n(t_0)\in\Theta$} converges to one.
\end{corollary}
\begin{proof}
By Assumption (A2) it holds $\theta_0\in\Theta$ and hence by Lemma \ref{lem:aux1} all estimates $\hat{\theta}_n(t_0)$ lie also in $\Theta$.
\end{proof}

\begin{proposition}
\label{lem:2}
Under the same assumptions as in Theorem \ref{thm:test_asymptotics} for any choice of \\$\theta_1^*(t_0),...,\theta^*_p(t_0)\in[\theta_0,\hat{\theta}(t_0)]$ (where for $a,b\in\IR^p$ we denote by $[a,b]$ the connecting line between $a$ and $b$), define the matrix
$$\ell_n^*(t_0):=\begin{pmatrix}
\partial_{\theta}^2\ell_{n,1\cdot}(\theta^*_1(t_0),t_0) \\
\vdots \\
\partial_{\theta}^2\ell_{n,q\cdot}(\theta^*_q(t_0),t_0)
\end{pmatrix},$$
where $\partial_{\theta}^2\ell_{n,r\cdot}$ denotes for $r\in\{1,...,p\}$ the $r$-th row of the second derivative of $\ell_n$ with respect to $\theta$. The matrix $\ell_n^*(t_0)$ concentrates around $\Sigma(\theta_0,t_0)$ (cf. \eqref{eq:def_sigma}), i.e.,
$$\sup_{t_0\in\IT}\left\|\frac{1}{r_n\bar{p}_n(t_0)}\ell_n^*(t_0)-\Sigma(t_0,\theta_0)\right\|=O_P\left(\sqrt{\frac{\log r_n}{r_np_n\cdot h}}+h\right).$$
Furthermore, $\ell_n^*(t_0)$ is invertible and
$$\sup_{t_0\in\IT}\left\|\left[\frac{1}{r_n\bar{p}_n(t_0)}\ell_n^*(t_0)\right]^{-1}-\Sigma(t_0,\theta_0)^{-1}\right\|=O_P\left(\sqrt{\frac{\log r_n}{r_np_n\cdot h}}+h\right).$$
\end{proposition}

\begin{proof}
We begin by rewriting $\ell_n^*$ in terms of the second derivatives
\begin{equation}
\label{eq:intvalsigma}
\left\|\frac{1}{r_n\bar{p}_n(t_0)}\ell_n^*(t_0)-\Sigma(t_0)\right\|^2\leq\sum_{r=1}^p\left\|\frac{1}{r_n\bar{p}_n(t_0)}\partial_{\theta}^2\ell_n(\theta^*_r(t_0),t_0)-\Sigma(t_0)\right\|^2.
\end{equation}
Since $p$ doesn't vary in $n$, it is enough to consider each term in the sum on the right hand side above separately. In order to reduce notation, we do not indicate which intermediate value $\theta^*_r(t_0)$ we consider and write simply $\theta^*(t_0)$ instead. Recall the definitions of $H_{n,ij}$, $\tilde{H}_{n,ij}$ and $\Sigma(s,\theta)$ in \eqref{eq:def_H}, \eqref{eq:def_Htilde} and \eqref{eq:def_sigma}, respectively. It holds that $\Sigma(s,\theta)=\IE(H_{n,ij}|C_{n,ij}(s)=1)$. Now, we can separate the problem as follows: Recall the abbreviation $\Sigma_t:=\Sigma(t,\theta_0)$
\begin{align}
&\left\|\frac{1}{r_n\bar{p}_n(t_0)}\partial_{\theta}^2\ell_n(\theta^*(t_0),t_0)-\Sigma_{t_0}\right\| \nonumber \\
\leq&\left\|\frac{1}{r_n\bar{p}_n(t_0)}\sum_{(i,j)\in L_n}\int_0^TK_{h,t_0}\left(s\right)\left[H_{n,ij}(s,\theta^*(t_0))-\Sigma(s,\theta^*(t_0))p_n(s)\right]ds\right\| \nonumber \\
&+\left\|\frac{1}{r_n}\sum_{(i,j)\in L_n}\int_0^TK_{h,t_0}\left(s\right)\left[\Sigma(s,\theta^*(t_0))\frac{p_n(s)}{\bar{p}_n(t_0)}-\Sigma_{t_0}\right]ds\right\| \nonumber \\
\leq&\sup_{\theta\in\Theta}\left\|\frac{1}{r_n\bar{p}_n(t_0)}\sum_{(i,j)\in L_n}\int_0^TK_{h,t_0}\left(s\right)\left[H_{n,ij}(s,\theta)-\Sigma(s,\theta)p_n(s))\right]ds\right\| \label{eq:L1} \\
&\quad+\left\|\int_0^TK_{h,t_0}\left(s\right)\left(\Sigma(s,\theta^*(t_0))-\Sigma_s\right)\frac{p_n(s)}{\bar{p}_n(t_0)}ds\right\| \label{eq:L2} \\
&\quad+\left\|\int_0^TK_{h,t_0}\left(s\right)\left(\Sigma_s-\Sigma_{t_0}\right)\frac{p_n(s)}{\bar{p}_n(t_0)}ds\right\| \label{eq:L3} \\
&\quad+\left\|\int_0^TK_{h,t_0}\left(s\right)\Sigma_{t_0}\left(\frac{p_n(s)}{\bar{p}_n(t_0)}-1\right)ds\right\|. \label{eq:L4}
\end{align}
We note firstly that $\eqref{eq:L4}=0$ by definition of $\bar{p}_n(t_0)$. Moreover, after taking the $\sup$ over all $t_0$, the convergence rate of line \eqref{eq:L2} equals $O_P\left(\sqrt{\frac{\log r_n}{r_np_nh}}\right)$, because of the Lipschitz continuity of $\Sigma$ in Lemma \ref{lem:sigma_diff} and Lemma \ref{lem:aux1} (recall that $\theta^*(t_0)$ is an intermediate value between $\hat{\theta}_n(t_0)$ and $\theta_0$ in Taylor's Formula). The expression in \eqref{eq:L3} can be handled by Assumption (A5) which states boundedness of $\partial_t\Sigma_t$ together with a Taylor expansion in the time parameter:
\begin{align*}
\sup_{t_0\in[0,T]}\left\|\int_0^TK_{h,t_0}\left(s\right)\left(\Sigma_s-\Sigma_{t_0}\right)\frac{p_n(s)}{\bar{p}_n(t_0)}ds\right\|\leq&\sup_{t_0\in[0,T]}\int_0^TK_{h,t_0}\left(s\right)\sup_{t\in[0,T]}\|\partial_t\Sigma_t\|\cdot|s-t_0|\frac{p_n(s)}{\bar{p}_n(t_0)}ds \\
\leq&h\cdot D,
\end{align*}
where we used in the last step that the kernel is supported on $[-1,1]$ and hence $|s-t_0|\leq h$. So \eqref{eq:L3} is of order $h$.

To deal with the first expression, line \eqref{eq:L1}, we let $\delta_n:=\sqrt{\frac{\log r_np_n}{r_np_n\cdot h}}$ and $C>0$ and denote by $T_{n,k_0}$ the discrete grid covering $\IT\times\Theta$ as defined in \eqref{eq:grid2}. We apply the same splitting technique as in \eqref{eq:P1} and \eqref{eq:P2} and obtain
\begin{align}
&\IP\left(\underset{\theta\in\Theta}{\sup_{t_0\in\IT}}\left\|\frac{1}{r_n\bar{p}_n(t_0)}\sum_{(i,j)\in L_n}\int_0^TK_{h,t_0}\left(s\right)\tilde{H}_{n,ij}(s,\theta)ds\right\|>C\delta_n\right) \nonumber \\
\leq&\IP\left(\underset{\|\theta_1-\theta_2\|\leq n^{-{k_0}}}{\underset{|t_0-t_0'|\leq hn^{-{k_0}},}{\sup_{t_0,t_0'\in\IT,\theta\in\Theta}}}\left\|\frac{1}{r_n\bar{p}_n(t_0)}\sum_{(i,j)\in L_n}\int_0^TK_{h,t_0}\left(s\right)\tilde{H}_{n,ij}(s,\theta_1)ds\right.\right. \nonumber \\
&\quad\quad\quad\quad\quad\left.\left.-\frac{1}{r_n\bar{p}_n(t_0')}\sum_{(i,j)\in L_n}\int_0^TK_{h,t_0'}\left(s\right)\tilde{H}_{n,ij}(s,\theta_2)ds\right\|>\frac{C}{2}\delta_n\right) \label{eq:L11} \\
&+\IP\left({\sup_{(t_0,\theta)\in T_{n,{k_0}}}}\left|\frac{1}{r_n\bar{p}_n(t_0)}\sum_{(i,j)\in L_n}\int_0^TK_{h,t_0}\left(s\right)\tilde{H}_{n,ij}(s,\theta)ds\right|>\frac{C}{2}\delta_n\right). \label{eq:L12}
\end{align}
In order to show that \eqref{eq:L11} converges to zero, we note that for $|t_0-t_0'|\leq hn^{-{k_0}},|\theta_1-\theta_2|\leq n^{-{k_0}}$. Note that by Lemma \ref{lem:sigma_diff} and Assumption (A3, \ref{ass:boundedness}) $\tilde{H}_{n,ij}(s,\theta)$ is Hoelder continuous with exponent $\alpha_H$ and random, time dependent constant $\gamma_{n,ij}(s)$ which is uniformly bounded. Thus, we get by Hoelder continuity of the kernel (Assumption (A4, \ref{ass:kernel_hoelder})) and of $\bar{p}_n(t_0)^{-1}$ (Lemma \ref{lem:p_continuous})
\begin{align*}
&\frac{1}{r_n}\sum_{(i,j)\in L_n}\int_0^T\left\|\frac{1}{\bar{p}_n(t_0)}\cdot K_{h,t_0}\left(s\right)\tilde{H}_{n,ij}(s,\theta_1)-\frac{1}{\bar{p}_n(t_0')}\cdot K_{h,t_0'}\left(s\right)\tilde{H}_{n,ij}(s,\theta_2)\right\|ds \\
\leq&\frac{1}{r_n}\sum_{(i,j)\in L_n}\int_0^T\frac{1}{\bar{p}_n(t_0)}\cdot\frac{1}{h}\left|K\left(\frac{s-t_0}{h}\right)-K\left(\frac{s-t_0'}{h}\right)\right|\cdot\|\tilde{H}_{n,ij}(s,\theta_1)\| \\
&\quad\quad+\left|\frac{1}{\bar{p}_n(t_0)}-\frac{1}{\bar{p}_n(t_0')}\right|\cdot K_{h,t_0'}\left(s\right)\cdot\|\tilde{H}_{n,ij}(s,\theta_1)\| \\
&\quad\quad+\frac{1}{\bar{p}_n(t_0')}\cdot K_{h,t_0'}\left(s\right)\left\|\tilde{H}_{n,ij}(s,\theta_1)-\tilde{H}_{n,ij}(s,\theta_2)\right\|ds \\
\leq&\frac{H_K}{p_n}\cdot\frac{1}{h}n^{-{k_0}\alpha_K}\frac{1}{r_n}\sum_{(i,j)\in L_n}\int_0^T\|\tilde{H}_{n,ij}(s,\theta_1)\|ds \\
&\quad\quad+H_{n,p}h^{\alpha_p}n^{-{k_0}\alpha_p}\cdot\frac{1}{r_n}\sum_{(i,j)\in L_n}\int_0^TK_{h,t_0'}\left(s\right)\|\tilde{H}_{n,ij}(s,\theta_1)\|ds \\
&\quad\quad+\frac{1}{p_n}\cdot\frac{1}{r_n}\sum_{(i,j)\in L_n}\int_0^TK_{h,t_0'}\left(s\right)\gamma_{n,ij}(s)n^{-{k_0}\alpha_H}ds,
\end{align*}
which converges to zero when $k_0$ is chosen large enough. \eqref{eq:L12} converges to zero by Statement \ref{lem_ass:exp2}. Thus we have shown the first part of the proposition. To prove that inversion preserves the rate, we denote $X_n(t_0):=\frac{1}{r_n\bar{p}_n(t_0)}\ell_n^*(t_0)$. Since we have just shown above that $X_n(t_0)$ converges in probability to $\Sigma_{t_0}$ we conclude by Lemma \ref{lem:bounded_inverse} that firstly $X_n(t_0)$ is with probability converging to one invertible and $\|X_n(t_0)^{-1}\|\leq M$. Thus, on this event,
$$\left\|X_n(t_0)^{-1}-\Sigma(t_0,\theta_0)^{-1}\right\|=\left\|X_n(t_0)^{-1}\left(\Sigma(t_0,\theta_0)-X_n(t_0)\right)\Sigma(t_0,\theta_0)^{-1}\right\|\leq M^2\left|\Sigma(t_0,\theta_0)-X_n(t_0)\right|$$
which concludes the proof of the proposition.
\end{proof}

\begin{proposition}
\label{prop:L1conv}
Under Assumption (A3, \ref{ass:const_estimator})
$$\sqrt{r_nh}\int_0^T\left(\hat{\theta}_n(t_0)-\theta_0\right)\frac{\bar{p}_n(t_0)}{\sqrt{\bar{p}_n}}w(t_0)dt_0=o_P(1).$$
\end{proposition}

\begin{proof}
To begin with, we use the Taylor expansion from equation \eqref{eq:entwicklung} which is shown there without reference to this Proposition. By using also the Cauchy-Schwarz Inequality we get for every entry $r\in\{1,...,p\}$
\begin{align}
&\left|\sqrt{r_nh}\int_0^T\left(\hat{\theta}_n^{(r)}(t_0)-\theta_0^{(r)}\right)\frac{\bar{p}_n(t_0)}{\sqrt{\bar{p}_n}}w(t_0)dt_0\right| \nonumber \\
\leq&\left|\int_0^T\left(\Sigma^{-1}_{t_0}-\left[\frac{1}{r_n\bar{p}_n(t_0)}\ell_n^*(t_0)\right]^{-1}\right)_{r\cdot}\cdot\sqrt{\frac{h}{r_n\bar{p}_n}}\partial_{\theta}\ell_n(\theta_0,t_0)w(t_0)dt_0\right| \nonumber \\
&\quad+\left|\int_0^T\left(\Sigma^{-1}_{t_0}\right)_{r\cdot}\cdot\sqrt{\frac{h}{r_n\bar{p}_n}}\partial_{\theta}\ell_n(\theta_0,t_0)w(t_0)dt_0\right| \nonumber \\
\leq&\left(\int_0^T\left\|\left(\Sigma^{-1}_{t_0}-\left[\frac{1}{r_n\bar{p}_n(t_0)}\ell^*_n(t_0)\right]^{-1}\right)_{r\cdot}\right\|^2w(t_0)dt_0\cdot\int_0^T\frac{h}{r_n\bar{p}_n}\left\|\partial_{\theta}\ell_n(\theta_0,t_0)\right\|^2w(t_0)dt_0\right)^{\frac{1}{2}} \label{eq:f1}\\
&\quad+\left|\int_0^T\left(\Sigma^{-1}_{t_0}\right)_{r\cdot}\cdot\sqrt{\frac{h}{r_n\bar{p}_n}}\partial_{\theta}\ell_n(\theta_0,t_0)w(t_0)dt_0\right| \label{eq:f2}.
\end{align}
We show now that \eqref{eq:f1} and \eqref{eq:f2} are both $o_P(1)$. We begin with \eqref{eq:f1}. Let $\epsilon,\eta>0$ be arbitrary, then for any $C>0$
\begin{align}
&\IP\left(\eqref{eq:f1}>\epsilon\right) \nonumber \\
\leq&\IP\left(\int_0^T\frac{h}{r_n\bar{p}_n}\left\|\partial_{\theta}\ell_n(\theta_0,t_0)\right\|^2w(t_0)dt_0>\frac{\epsilon^2}{C^2}\right) \nonumber \\
&\quad+\IP\left(\int_0^T\left\|\left(\Sigma^{-1}_{t_0}-\left[\frac{1}{r_n\bar{p}_n(t_0)}\ell^*_n(t_0)\right]^{-1}\right)_{r\cdot}\right\|^2w(t_0)dt_0>C^2\right). \label{eq:gehtgleichweiter}
\end{align}

In order to deal with the first part, we use Markov's Inequality. The resulting expectation is written in terms of \eqref{eq:deriv_decomp} and can be bounded by using the fact that the counting process martingales are uncorrelated and that everything is identically distributed. More precisely, we obtain for $h<\delta/2$ (cf. Assumption (A1))
\begin{align*}
&\IE\left(\int_0^T\frac{h}{r_n\bar{p}_n}\left\|\partial_{\theta}\ell_n(\theta_0,t_0)\right\|^2w(t_0)dt_0\right) \\
=&\frac{h}{r_n\bar{p}_n}\sum_{i\in E_n}\int_0^T\int_0^TK_{h,t_0}\left(t\right)^2\IE\left(\left\|\partial_{\theta}\log\lambda_{n,ij}(\theta_0,t)\right\|^2C_{n,ij}(t)\lambda_{n,ij}(\theta_0,t)\right)dtw(t_0)dt_0 \\
\leq&\int_0^T\int_0^ThK_{h,t_0}\left(t\right)^2w(t_0)dt_0\IE\left(\left\|\partial_{\theta}\log\lambda_{n,ij}(\theta_0,t)\right\|^2\lambda_{n,ij}(\theta_0,t)\Big|C_{n,ij}(t)=1\right)\frac{p_n(t)}{\bar{p}_n}dt \\
\leq&\|w\|_{\infty}K\int_{\delta/2}^{T-\delta/2}\frac{p_n(t)}{\bar{p}_n}dt\cdot\sup_{t\in[0,T]}\IE\left(\left\|\partial_{\theta}\log\lambda_{n,ij}(\theta_0,t)\right\|^2\lambda_{n,ij}(\theta_0,t)|C_{n,ij}(t)=1\right)=:C^*,
\end{align*}
where $K$ is the bound on the kernel from (A4, \ref{ass:kernel_hoelder}), $\|w\|_{\infty}<\infty$ by (A1), the supremum is finite by Assumption (A3, \ref{ass:boundedness}) and $\int_{\delta/2}^{T-\delta/2}\frac{p_n(t)}{\bar{p}_n}dt\leq1$ by definition. By Proposition \ref{lem:2} and the assumptions on $h$ in (A4, \ref{ass:bw}) we find that for all $C>0$ and thus in particular $C=\frac{\epsilon\sqrt{\eta}}{\sqrt{2C^*}}$ it holds for $n$ large enough that
$$\IP\left(\int_0^T\left\|\left(\Sigma^{-1}_{t_0}-\left[\frac{1}{r_n\bar{p}_n(t_0)}\ell^*_n(t_0)\right]^{-1}\right)_{r\cdot}\right\|^2w(t_0)dt_0>C^2\right)\leq\frac{\eta}{2}.$$
Now, by using all previous considerations we may estimate by using \eqref{eq:gehtgleichweiter} for $n$ large enough by
$$\IP(\eqref{eq:f1}>\epsilon)\leq\frac{C^2}{\epsilon^2}\IE\left(\int_0^T\frac{h}{r_n\bar{p}_n}\left\|\partial_{\theta}\ell_n(\theta_0,t_0)\right\|^2w(t_0)dt_0\right)+\frac{\eta}{2}\leq\eta.$$
Since $\epsilon,\eta>0$ were chosen arbitrarily, we have shown that $\eqref{eq:f1}=o_P(1)$.

We continue with \eqref{eq:f2}. This is easier to handle because $\Sigma^{-1}_{t_0}$ is deterministic and thus in particular predictable. It is therefore not necessary to separate first and second derivative as we did in \eqref{eq:f1}. Let $\epsilon>0$ be arbitrary, then we find by applying Lemma \ref{lem:Fubini} again for $h<\delta/2$
\begin{align*}
&\IP\left(\eqref{eq:f2}>\epsilon\right) \\
\leq&\frac{h}{\epsilon^2r_n\bar{p}_n}\IE\left(\left(\sum_{(i,j)\in L_n}\int_0^T\int_0^T\left(\Sigma^{-1}_{t_0}\right)_{r\cdot}K_{h,t_0}\left(t\right)w(t_0)dt_0\partial_{\theta}\log\lambda_{n,ij}(\theta_0,t)dM_{n,ij}(t)\right)^2\right) \\
\leq&\frac{h}{\epsilon^2r_n\bar{p}_n}\sum_{(i,j)\in L_n}\int_0^T\IE\left(\left\|\int_0^TK_{h,t_0}\left(t\right)\left(\Sigma^{-1}_{t_0}\right)_{r\cdot}w(t_0)dt_0\right\|^2\left\|\partial_{\theta}\log\lambda_{n,ij}(\theta_0,t)\right\|^2C_{n,ij}(t)\lambda_{n,ij}(\theta_0,t)\right)dt \\
\leq&\frac{h\|w\|_{\infty}^2}{\epsilon^2\bar{p}_n}\int_{\delta/2}^{T-\delta/2}p_n(t)dt\cdot\sup_{t\in[0,T]}\IE\left(\left\|\partial_{\theta}\log\lambda_{n,ij}(\theta_0,t)\right\|^2\lambda_{n,ij}(\theta_0,t)\Big|C_{n,ij}(t)=1\right)\cdot\sup_{t\in[0,T]}\left\|\left(\Sigma_{t}^{-1}\right)_{r\cdot}\right\|.
\end{align*}
Since $h\to0$, this converges to zero by Assumption (A3, \ref{ass:boundedness}), Lemma \ref{lem:bounded_inverse} and using once again that $\int_{\delta/2}^{T-\delta/2}\frac{p_n(t)}{\bar{p}_n}dt\leq1$. Thus, also $\eqref{eq:f2}=o_P(1)$.
\end{proof}

\begin{proof}[Proof of Theorem \ref{thm:test_asymptotics}]
We note firstly that we may replace the estimator $\bar{\theta}_n$ in the test statistic with $\theta_0$ because for $T_{0,n}:=\int_0^T\left\|\hat{\theta}_n(t_0)-\theta_0\right\|^2\bar{p}_n(t_0)w(t_0)dt_0$ it holds that
\begin{align*}
r_nh^{\frac{1}{2}}T_n=&r_nh^{\frac{1}{2}}T_{0,n}+r_nh^{\frac{1}{2}}\int_0^T\left\|\bar{\theta}_n-\theta_0\right\|^2\bar{p}_n(t_0)w(t_0)dt_0 \\
&+2r_nh^{\frac{1}{2}}\int_0^T\left(\hat{\theta}_n(t_0)-\theta_0\right)^T\left(\bar{\theta}_n-\theta_0\right)\bar{p}_n(t_0)w(t_0)dt_0.
\end{align*}
By the Assumptions (A1), (A3, \ref{ass:const_estimator}), Proposition \ref{prop:L1conv} and $h\to0$ the last two terms may be asymptotically neglected. Hence, the limiting distribution of $r_nh^{\frac{1}{2}}T_n$ can be found by studying $r_nh^{\frac{1}{2}}T_{0,n}$.

By Corollary \ref{cor:theta_interior}, $\hat{\theta}(t_0)\in\Theta$ with high probability. Since $\ell_n$ is differentiable by Lemma \ref{lem:diff_integral}, we thus have $\partial_{\theta}\ell_n(\hat{\theta}(t_0),t_0)=0$ on this event. As we are concerned with convergence in the distribution, we may restrict to this event. By a Taylor expansion there are $\theta^*_r(t_0)$ which lie on the connecting line between $\hat{\theta}(t_0)$ and $\theta_0$ such that
\begin{equation}
\label{eq:standard_taylor}
0=\partial_{\theta}\ell_{n,r}(\hat{\theta},t_0)=\partial_{\theta}\ell_{n,r}(\theta_0(t_0),t_0)+\partial_{\theta}^2\ell_{n,r\cdot}(\theta^*_r(t_0),t_0)(\hat{\theta}(t_0)-\theta_0),
\end{equation}
where $\partial_{\theta}\ell_{n,r}$ is the $r$-th component of the gradient (with respect to $\theta$) of $\ell_n$ and $\partial_{\theta}^2\ell_{n,r\cdot}$ denotes the $r$-th row of the Hessian Matrix of $\ell_n$ with respect to $\theta$. Define
$$\ell_n^*(t_0):=\begin{pmatrix}
\partial_{\theta}^2\ell_{n,1\cdot}(\theta^*_1(t_0),t_0) \\
\vdots \\
\partial_{\theta}^2\ell_{n,q\cdot}(\theta^*_q(t_0),t_0)
\end{pmatrix}.$$
By Proposition \ref{lem:2} and (A4, \ref{ass:bw}) we have that $\ell_n^*(t_0)$ is uniformly close to $\Sigma_{t_0}$. Hence, by Lemma \ref{lem:bounded_inverse} we find that with probability tending to one $\ell^*_n(t_0)$ is invertible for all $t_0\in[0,T]$. Thus, \eqref{eq:standard_taylor} is equivalent to
\begin{equation}
\hat{\theta}_n(t_0)-\theta_0=-\left[\ell_n^*(t_0)\right]^{-1}\,\cdot\,\partial_{\theta}\ell_n(\theta_0,t_0). \label{eq:entwicklung}
\end{equation}
Using this expansion and by applying Propositions \ref{lem:1} and \ref{lem:2}, we obtain (use also the properties of $w$ in (A1))
\begin{align*}
T_{0,n}=&\int_{\delta}^{T-\delta}\left\|\left[\frac{1}{r_n\bar{p}_n(t_0)}\ell_n^*(t_0)\right]^{-1}\,\cdot\,\frac{1}{r_n\bar{p}_n(t_0)}\partial_{\theta}\ell_n(\theta_0,t_0)\right\|^2w(t_0)\bar{p}_n(t_0)dt_0 \\
=&\int_{\delta}^{T-\delta}\left\|\Sigma_{t_0}^{-1}\frac{\partial_{\theta}\ell_n(\theta_0,t_0)}{r_n\bar{p}_n(t_0)}+\left(\left[\frac{\ell_n^*(t_0)}{r_n\bar{p}_n(t_0)}\right]^{-1}-\Sigma_{t_0}^{-1}\right)\frac{\partial_{\theta}\ell_n(\theta_0,t_0)}{r_n\bar{p}_n(t_0)}\right\|^2w(t_0)\bar{p}_n(t_0)dt_0 \\
=&\int_{\delta}^{T-\delta}\left\|\Sigma_{t_0}^{-1}\frac{\partial_{\theta}\ell_n(\theta_0,t_0)}{r_n\bar{p}_n(t_0)}\right\|^2w(t_0)\bar{p}_n(t_0)dt_0 \\
&\,\,\,+O_P\left(\frac{\log r_n}{r_nh}\left(\sqrt{\frac{\log r_n}{r_np_nh}}+h+\left(\sqrt{\frac{\log r_n}{r_np_nh}}+h\right)^2\right)\right) \\
=&\int_{\delta}^{T-\delta}\left\|\Sigma_{t_0}^{-1}\frac{\partial_{\theta_0}\ell_n'(\theta_0,t_0)}{r_n\bar{p}_n(t_0)}\right\|^2w(t_0)\bar{p}_n(t_0)dt_0+O_P\left(\frac{\log r_n}{r_nh}\left(\sqrt{\frac{\log r_n}{r_np_nh}}+h\right)\right).
\end{align*}
Hence,
\begin{align*}
r_nh^{\frac{1}{2}}T_{0,n}=&r_nh^{\frac{1}{2}}\int_{\delta}^{T-\delta}\left\|\Sigma_{t_0}^{-1}\frac{\partial_{\theta}\ell_n(\theta_0,t_0)}{r_n\bar{p}_n(t_0)}\right\|^2w(t_0)\bar{p}_n(t_0)dt_0+O_P\left(\frac{(\log r_n)^{\frac{3}{2}}}{\sqrt{r_np_n}\cdot h}+\left(h(\log r_n)^2\right)^{\frac{1}{2}}\right),
\end{align*}
where the $O_p$ part is $o_P(1)$ by Assumption (A4, \ref{ass:bw}) on the bandwidth. Thus, for the asymptotic considerations, we have to investigate only the first part. By noting that $\log x\cdot y-x\leq \log y\cdot y-y$ for all $x,y>0$ we see that $\theta\mapsto\log\lambda_{n,ij}(\theta,s)\cdot\lambda_{n,ij}(\theta_0,s)-\lambda_{n,ij}(\theta,s)$ is maximal for $\theta=\theta_0$ and this holds for all $(i,j)\in L_n$ and and all $t\in[0,T]$. Hence, $\theta\mapsto P_n(\theta,t_0)$ defined in \eqref{eq:def_P} has a local maximum at $\theta_0$. By differentiability of $P_n$ (cf. Lemma \ref{lem:diff_integral}) we conclude that $\partial_{\theta}P_n(\theta_0,t_0)=0$ for all $t_0\in[0,T]$. Using this together with the decomposition of the likelihood in \eqref{eq:decomp_ell} we obtain (the order of integration may be interchanged by Lemma \ref{lem:Fubini})
\begin{align}
&r_nh^{\frac{1}{2}}\int_{\delta}^{T-\delta}\left\|\Sigma^{-1}_{t_0}\frac{\partial_{\theta}\ell_n(\theta_0,t_0)}{r_n\bar{p}_n(t_0)}\right\|^2w(t_0)\bar{p}_n(t_0)dt_0 \nonumber \\
=&r_nh^{\frac{1}{2}}\int_{\delta}^{T-\delta}\left\|\frac{\Sigma^{-1}_{t_0}}{r_n\bar{p}_n(t_0)}\sum_{(i,j)\in L_n}\int_0^TK_{h,t_0}\left(t\right)\frac{\partial_{\theta}\lambda_{n,ij}(\theta_0,t)}{\lambda_{n,ij}(\theta_0,t)}dM_{n,ij}(t)\right\|^2w(t_0)\bar{p}_n(t_0)dt_0 \nonumber \\
=&\frac{1}{h^{\frac{1}{2}}r_n}\sum_{(i,j),(k,l)\in L_n}\int_{[0,T]^3}hK_{h,t_0}\left(s\right)K_{h,t_0}\left(t\right)X_{n,ij}(s)^T\Sigma_{t_0}^{-T}\Sigma_{t_0}^{-1}X_{n,kl}(t)\frac{w(t_0)}{\bar{p}_n(t_0)}dt_0dM_{n,ij}(s)dM_{n,kl}(t) \nonumber \\
=&\frac{1}{h^{\frac{1}{2}}r_n}\sum_{(i,j),(k,l)\in L_n}\int_0^T\int_0^Tf_{n,ij,kl}(s,t)dM_{n,ij}(s)dM_{n,kl}(t), \label{eq:a}
\end{align}
where $f_{n,ij,kl}$ was defined in \eqref{eq:deff}. Note that $f_{n,ij,kl}(s,t)=f_{n,kl,ij}(t,s)$. Then (in the second line we reorder integration by Lemma \ref{lem:Fubini}, and the third equality is not term-wise the same but for the whole sum),
\begin{align}
\eqref{eq:a}=&\frac{1}{h^{\frac{1}{2}}r_n}\sum_{(i,j),(k,l)\in L_n}\int_0^T\int_0^Tf_{n,ij,kl}(s,t)\left(\Ind_{t<s}+\Ind_{t>s}+\Ind_{t=s}\right)dM_{n,ij}(s)dM_{n,kl}(t) \nonumber \\
=&\frac{1}{h^{\frac{1}{2}}r_n}\sum_{(i,j),(k,l)\in L_n}\Bigg[\int_0^T\int_0^{s-}f_{n,ij,kl}(s,t)dM_{n,kl}(t)dM_{n,ij}(s)+\int_0^T\int_0^{t-}f_{n,ij,kl}(s,t)dM_{n,ij}(s)dM_{n,kl}(t) \nonumber \\
&\,\,\,\,\,\quad\quad\quad+\int_0^T\int_{\{s\}}f_{n,ij,kl}(s,t)dM_{n,kl}(t)dM_{n,ij}(s)\Bigg] \nonumber \\
=&\frac{1}{h^{\frac{1}{2}}r_n}\sum_{(i,j)\in L_n}\left[2\int_0^T\int_0^{s-}f_{n,ij,ij}(s,t)dM_{n,ij}(t)dM_{n,ij}(s)+\int_0^T\int_{\{s\}}f_{n,ij,ij}(s,t)dM_{n,ij}(t)dM_{n,ij}(s)\right] \label{eq:line1} \\
+&\frac{1}{h^{\frac{1}{2}}r_n}\underset{(i,j)\neq (k,l)}{\sum_{(i,j),(k,l)\in L_n}}\Bigg[2\int_0^T\int_0^{s-}f_{n,ij,kl}(s,t)dM_{n,kl}(t)dM_{n,ij}(s)+\int_0^T\int_{\{s\}}f_{n,ij,kl}(s,t)dM_{n,kl}(t)dM_{n,ij}(s)\Bigg]. \label{eq:line2}
\end{align}
We will consider lines \eqref{eq:line1} and \eqref{eq:line2} separately. We start with line \eqref{eq:line1} and in there, we start with the second integral: Note that the martingales $M_{n,ij}$ have jumps of height exactly one at those positions where the counting processes $N_{n,ij}$ jump (this is because we assume a continuous integrated intensity process). Hence we have
\begin{equation}
\label{eq:opi}
\int_{\{s\}}f_{n,ij,ij}(s,t)dM_{n,ij}(t)=\Ind_{\Delta N_{n,ij}(s)=1}f_{n,ij,ij}(s,s),
\end{equation}
and furthermore
\begin{align*}
\int_0^T\int_{\{s\}}f_{n,ij,ij}(s,t)dM_{n,ij}(t)dM_{n,ij}(s)=\int_0^T\Ind_{\Delta N_{n,ij}(s)=1}f_{n,ij,ij}(s,s)dM_{n,ij}(s)=\int_0^Tf_{n,ij,ij}(s,s)dN_{n,ij}(s).
\end{align*}
Using the above equality, we obtain
$$\eqref{eq:line1}=\frac{1}{h^{\frac{1}{2}}r_n}\sum_{(i,j)\in L_n}\left[2\int_0^T\int_0^{s-}f_{n,ij,ij}(s,t)dM_{n,ij}(t)dM_{n,ij}(s)+\int_0^Tf_{n,ij,ij}(s,s)dN_{n,ij}(s)\right].$$
The first sum is a sum of uncorrelated martingales and so it will converge to zero in probability by an application of Markov's inequality: Denote by $g_{n,ij}(s)$ a sequence of identically distributed, predictable functions, then in general $\IE\left(\int_0^Tg_{n,ij}(s)dM_{n,ij}(s)\right)=0$ and for $(i,j),(k,l)\in L_n$
\begin{align*}
&\IE\left(\int_0^Tg_{n,ij}(s)dM_{n,ij}(s)\cdot\int_0^Tg_{n,kl}(s)dM_{n,kl}(s)\right)=0, \textrm{ for }(i,j)\neq (k,l), \\
&\IE\left(\left(\int_0^Tg_{n,ij}(s)dM_{n,ij}(s)\right)^2\right)=\int_0^T\IE\left(g_{n,ij}(s)^2C_{n,ij}(s)\lambda_{n,ij}(s,\theta_0)\right)ds.
\end{align*}
So we get for any $\epsilon>0$
\begin{align*}
&\IP\left(\left|\frac{1}{r_n}\sum_{(i,j)\in L_n}\int_0^Tg_{n,ij}(s)dM_{n,ij}(s)\right|>\epsilon\right)\leq\frac{1}{\epsilon^2}\frac{1}{r_n}\IE\left(\int_0^Tg_{n,ij}(s)^2C_{n,ij}(s)\lambda_{n,ij}(\theta_0,s)ds\right).
\end{align*}
When letting $g_{n,ij}(s)=h^{-\frac{1}{2}}\int_0^{s-}f_{n,ij,ij}(s,t)dM_{n,ij}(t)$, we have by Lemma \ref{lem:bounded_g} that the above converges to zero. Moreover, by definition
$$h^{-\frac{1}{2}}A_n=\frac{1}{h^{\frac{1}{2}}r_n}\sum_{(i,j)\in L_n}\int_0^Tf_{n,ij,ij}(s,s)dN_{n,ij}(s).$$
Combining these considerations yields $\eqref{eq:line1}=o_p(1)+h^{-\frac{1}{2}}A_n$.

Next we consider \eqref{eq:line2}. Firstly, we note, using an analogue of \eqref{eq:opi}, that the second integral in \eqref{eq:line2} equals zero because the two martingales $M_{n,ij}$ and $M_{n,kl}$ never jump simultaneously because $(i,j)\neq (k,l)$. To investigate the first integral we simplify notation by using the predictable functions $\tau_{n,ij,kl}$ defined in \eqref{eq:def_tau}. Then
$$\eqref{eq:line2}=\frac{1}{h^{\frac{1}{2}}r_n}\sum_{(i,j)\in L_n}\int_0^T2\left(\underset{(k,l)\neq (i,j)}{\sum_{(k,l)\in L_n}}\tau_{n,ij,kl}(s)\right)dM_{n,ij}(s)$$
is a martingale in $T$. We intent to show convergence to a normal distribution by using Rebolledo's martingale central limit theorem (Theorem \ref{thm:Rebolledo}). To this end, we need to prove the convergence of the variation towards a deterministic quantity and that the jump parts of the process converge to zero. We start with the quadratic variation (note that $M_{n,ij}$ and $M_{n,kl}$ are uncorrelated whenever $(i,j)\neq (k,l)$):
\begin{align*}
&\left<\frac{1}{h^{\frac{1}{2}}r_n}\sum_{(i,j)\in L_n}\int_0^T2\left(\underset{j\neq i}{\sum_{(k,l)\in L_n}}\tau_{n,ij,kl}(s)\right)dM_{n,ij}(s)\right> \\
=&\frac{4}{hr_n^2}\sum_{(i,j)\in L_n}\int_0^T\left(\underset{(k,l)\neq (i,j)}{\sum_{(k,l)\in L_n}}\tau_{n,ij,kl}(s)\right)^2C_{n,ij}(s)\lambda_{n,ij}(\theta_0,s)ds \\
=&\frac{4}{hr_n^2}\sum_{(i,j)\in L_n}\int_0^T\underset{(k_1,l_1),(k_2,l_2)\neq (i,j)}{\sum_{(k_1,l_1),(k_2,l_2)\in L_n}}\tau_{n,ij,k_1l_1}(s)\tau_{n,ij,k_2l_2}(s)C_{n,ij}(s)\lambda_{n,ij}(\theta_0,s)ds \\
=&\frac{4}{hr_n^2}\sum_{(i,j)\in L_n}\int_0^T\underset{(k,l)\neq (i,j)}{\sum_{(k,l)\in L_n}}\tau_{n,ij,kl}(s)^2C_{n,ij}(s)\lambda_{n,ij}(\theta_0,s)ds \\
&\,\,\,+\frac{4}{hr_n^2}\sum_{(i,j)\in L_n}\int_0^T\underset{(k_1,l_1)\neq (k_2,l_2)}{\underset{(k_1,l_1),(k_2,l_2)\neq (i,j)}{\sum_{(k_1,l_1),(k_2,l_2)\in L_n}}}\tau_{n,ij,k_1l_1}(s)\tau_{n,ij,k_2l_2}(s)C_{n,ij}(s)\lambda_{n,ij}(\theta_0,s)ds \\
\overset{\IP}{\to}&B,
\end{align*}
by Lemma \ref{lem_ass:var}. Now the jump process (the process which contains all jumps of size greater than or equal to $\epsilon>0$) is given by (note that no two martingales jump at the same time)
$$\frac{2}{h^{\frac{1}{2}}r_n}\sum_{(i,j)\in L_n}\int_0^T\Ind\left(\left|\frac{2}{h^{\frac{1}{2}}r_n}\underset{(k,l)\neq (i,j)}{\sum_{(k,l)\in L_n}}\tau_{n,ij,kl}(s)\right|>\epsilon\right)\underset{(k,l)\neq (i,j)}{\sum_{(k,l)\in L_n}}\tau_{n,ij,kl}(s)dM_{n,ij}(s),$$
which converges to zero by Lemma \ref{lem_ass:jump}. Hence, by Rebolledo's martingale central limit theorem (see Theorem \ref{thm:Rebolledo} in the Appendix)
$$\eqref{eq:line2}\overset{d}{\to}N(0,B)$$
and the statement of the theorem is shown.
\end{proof}

\subsection{Proof of Lemmas \ref{lem_ass:var}-\ref{lem_ass:exp2}}
\label{sec:lemma_proofs}

\begin{proof}[Proof of Lemma \ref{lem_ass:var}]
Recall that for $u,v\in L_n$
$$f_{n,u,v}(s,t)=X_{n,u}(s)^T\int_{\delta}^{T-\delta}hK_{h,t_0}(s)K_{h,t_0}(t)\Sigma_{t_0}^{-T}\Sigma_{t_0}^{-1}\frac{w(t_0)}{\bar{p}_n(t_0)}dt_0X_{n,v}(t).$$
By substituting $y=\frac{s-t_0}{h}$ we obtain
\begin{align*}
&f_{n,u,v}(s,t) \\
=&X_{n,u}(s)^T\int_{(s-T)/h}^{s/h}K(y)K\left(\frac{t-s}{h}+y\right)\Sigma^{-T}_{s-yh}\Sigma_{s-yh}^{-1}\frac{w(s-yh)}{\bar{p}_n(s-yh)}dyX_{n,v}(t).
\end{align*}
For ease of notation we denote
\begin{equation}
\label{eq:ftilde}
\tilde{f}_{n}(s,t):=\int_{\frac{s-T+\delta}{h}}^{\frac{s-\delta}{h}}K(y)K\left(\frac{t-s}{h}+y\right)\Sigma^{-T}_{s-yh}\Sigma_{s-yh}^{-1}\frac{w(s-yh)}{\bar{p}_n(s-yh)}dy.
\end{equation}
Note firstly that $\tilde{f}_n$ is not random and secondly that for $t<s-2h$ the above expression equals zero because we assume that the kernel is supported on $[-1,1]$ (cf. Assumption (A4, \ref{ass:kernel_hoelder})). So we obtain for $\tau_{n,u,v}$:
\begin{equation}
\label{eq:taurep}
\tau_{n,u,v}(s)=X_{n,u}(s)^T\int_{s-2h}^{s-}\tilde{f}_{n}(s,t)X_{n,v}(t)dM_{n,v}(t).
\end{equation}
The integral in the above display is over a vector-valued integrand. Such integrals are always understood element-wise. We begin with proving \eqref{eq:var2}. By using the representation of $\tau_{n,u,v}$ in \eqref{eq:taurep} we obtain
\begin{align*}
&\frac{4}{hr_n^2}\sum_{u\in L_n}\underset{u_1\neq u_2}{\underset{u_1,u_2\neq u}{\sum_{u_1,u_2\in L_n}}}\int_0^T\tau_{n,u,u_1}(s)\tau_{n,u,u_2}(s)C_{n,u}(s)\lambda_{n,u}(s)ds \\
=&\frac{4}{hr_n^2}\sum_{u\in L_n}\underset{u_1\neq u_2}{\underset{u_1,u_2\neq u}{\sum_{u_1,u_2\in L_n}}}\int_0^TX_{n,u}(s)^T\int_{s-2h}^{s-}\tilde{f}_{n}(s,t)X_{n,u_1}(t)dM_{n,u_1}(t) \\
&\,\,\,\times\left(\int_{s-2h}^{s-}\tilde{f}_{n}(s,t)X_{n,u_2}(t)dM_{n,u_2}(t)\right)^TX_{n,u}(s)C_{n,u}(s)\lambda_{n,u}(s,\theta_0)ds.
\end{align*}
In order to study the behaviour of these integrals, we write the product of the two integrals as a sum. The equation from before continues (we can interchange the order of integration because the integrand after the second equality is non-negative and the martingales can be split into the counting process integral and a regular Lebesgue integral):
\begin{align*}
=&\frac{4}{hr_n^2}\sum_{u\in L_n}\underset{u_1\neq u_2}{\underset{u_1,u_2\neq u}{\sum_{u_1,u_2\in L_n}}}\int_0^TX_{n,u}(s)^T\int_{s-2h}^{s-}\tilde{f}_{n}(s,t)X_{n,u_1}(t)dM_{n,u_1}(t) \\
&\,\,\,\times\int_{s-2h}^{s-}X_{n,u_2}^T(t)\tilde{f}_{n}(s,t)^T(t)dM_{n,u_2}(t)X_{n,u}(s)C_{n,u}(s)\lambda_{n,u}(s,\theta_0)ds \\
=&\frac{4}{hr_n^2}\sum_{u\in L_n}\underset{u_1\neq u_2}{\underset{u_1,u_2\neq u}{\sum_{u_1,u_2\in L_n}}}\int_0^T\int_0^T\int_0^T\Ind_{t<s}\Ind_{t\geq s-2h}\Ind_{r<s}\Ind_{r\geq s-2h}X_{n,u}(s)^T\tilde{f}_{n}(s,t)X_{n,u_1}(t) \\
&\,\,\,\times X_{n,u_2}(r)^T\tilde{f}_{n}(s,r)^TX_{n,u}(s)C_{n,u}(s)\lambda_{n,u}(s,\theta_0)dsdM_{n,u_2}(r)dM_{n,u_1}(t) \\
=&\frac{4}{hr_n^2}\sum_{u\in L_n}\underset{u_1\neq u_2}{\underset{u_1,u_2\neq u}{\sum_{u_1,u_2\in L_n}}}\int_0^T\int_0^T\int_{\max(t,r)}^{\min(t,r)+2h}X_{n,u}(s)^T\tilde{f}_{n}(s,t)X_{n,u_1}(t) \\
&\,\,\,\times X_{n,u_2}(r)^T\tilde{f}_{n}(s,r)^TX_{n,u}(s)C_{n,u}(s)\lambda_{n,u}(s,\theta_0)dsdM_{n,u_2}(r)dM_{n,u_1}(t) \\
=&\frac{4}{hr_n^2}\sum_{u\in L_n}\underset{u_1\neq u_2}{\underset{u_1,u_2\neq u}{\sum_{u_1,u_2\in L_n}}}\int_0^T\int_0^T(\Ind_{r<t}+\Ind_{r>t})\int_{\max(t,r)}^{\min(t,r)+2h}X_{n,u}(s)^T\tilde{f}_{n}(s,t)X_{n,u_1}(t) \\
&\,\,\,\times X_{n,u_2}(r)^T\tilde{f}_{n}(s,r)^TX_{n,u}(s)C_{n,u}(s)\lambda_{n,u}(s,\theta_0)dsdM_{n,u_2}(r)dM_{n,u_1}(t) \\
\end{align*}
Note that we do not need the indicator $\Ind_{t=r}$ because $u_1\neq u_2$ and hence the martingales $M_{n,u_1}$ and $M_{n,u_2}$ will not jump simultaneously almost surely. We continue (for the second equality interchange the roles of $u_1$ and $u_2$ as well as the roles of $t$ and $r$)
\begin{align}
=&\frac{4}{hr_n^2}\sum_{u\in L_n}\underset{u_1\neq u_2}{\underset{u_1,u_2\neq u}{\sum_{u_1,u_2\in L_n}}}\int_0^T\int_0^{t-}\int_{\max(t,r)}^{\min(t,r)+2h}X_{n,u}(s)^T\tilde{f}_{n}(s,t)X_{n,u_1}(t) \nonumber \\
&\,\,\,\times X_{n,u_2}(r)^T\tilde{f}_{n}(s,r)^TX_{n,u}(s)C_{n,u}(s)\lambda_{n,u}(s,\theta_0)dsdM_{n,u_2}(r)dM_{n,u_1}(t) \nonumber \\
+&\frac{4}{hr_n^2}\sum_{u\in L_n}\underset{u_1\neq u_2}{\underset{u_1,u_2\neq u}{\sum_{u_1,u_2\in L_n}}}\int_0^T\int_0^{r-}\int_{\max(t,r)}^{\min(t,r)+2h}X_{n,u}(s)^T\tilde{f}_{n}(s,t)X_{n,u_1}(t) \nonumber \\
&\,\,\,\times X_{n,u_2}(r)^T\tilde{f}_{n}(s,r)^TX_{n,u}(s)C_{n,u}(s)\lambda_{n,u}(s,\theta_0)dsdM_{n,u_1}(t)dM_{n,u_2}(r) \nonumber \\
=&\frac{8}{hr_n^2}\sum_{u\in L_n}\underset{u_1\neq u_2}{\underset{u_1,u_2\neq u}{\sum_{u_1,u_2\in L_n}}}\int_0^T\int_{t-2h}^{t-}\int_{t}^{r+2h}X_{n,u}(s)^T\tilde{f}_{n}(s,t)X_{n,u_1}(t) \nonumber \\
&\,\,\,\times X_{n,u_2}(r)^T\tilde{f}_{n}(s,r)^TX_{n,u}(s)C_{n,u}(s)\lambda_{n,u}(s,\theta_0)dsdM_{n,u_2}(r)dM_{n,u_1}(t) \nonumber \\
=&\frac{8}{r_n}\underset{u_1\neq u_2}{\sum_{u_1,u_2\in L_n}}\int_0^T\int_{t-2h}^{t-}\int_{0}^{\frac{r-t}{h}+2}\frac{1}{r_n}\underset{u\neq u_1,u_2}{\sum_{u\in L_n}}X_{n,u}(t+yh)^T\tilde{f}_{n}(t+yh,t)X_{n,u_1}(t)X_{n,u_2}(r)^T \nonumber \\
&\,\,\,\times \tilde{f}_{n}(t+yh,r)^TX_{n,u}(t+yh)C_{n,u}(t+yh)\lambda_{n,u}(t+yh,\theta_0)dydM_{n,u_2}(r)dM_{n,u_1}(t) \nonumber \\
=& \frac{8}{r_n}\underset{u_1\neq u_2}{\sum_{u_1,u_2\in L_n}}\int_0^T\int_{t-2h}^{t-}\phi_{n,u_1u_2}(t,r)dM_{n,u_2}(r)dM_{n,u_1}(t), \label{eq:p1}
\end{align}
here we have introduced the notation $\phi_{n,u_1u_2}(t,r):=\tilde{\phi}_{n,u_1u_2}^{\emptyset}(t,r)$, where for any set of pairs $I\subseteq L_n$ ($d_{t}^n(u,\emptyset)=\infty$)
\begin{align}
&\tilde{\phi}^I_{n,u_1u_2}(t,r) \nonumber \\
:=&\frac{1}{r_n}\int_{0}^{\frac{r-t}{h}+2} \sum_{u\neq u_1,u_2}X_{n,u}(t+yh)^T\tilde{f}_{n}(t+yh,t)X_{n,u_1}(t)X_{n,u_2}(r)^T\tilde{f}_{n}(t+yh,r)^T \nonumber \\
&\,\,\,\,\,\times X_{n,u}(t+yh)C_{n,u}(t+yh)\lambda_{n,u}(t+yh,\theta_0)\Ind(d_{t-4h}^n(u,I)\geq m)dy \label{eq:phi_tilde}
\end{align}
The functions $\tilde{\phi}^I_{n,u_1u_2}(t,r)$ are partially predictable with respect to $\tilde{\mathcal{F}}_{u_1u_2,t}^{n,I,m}$ for $I\supseteq\{u_1,u_2\}$ because the integrand above has the product structure as mentioned in Definition \ref{def:preliminary_partially_predictable} and summation and integration is preserving this property as it is a measurability property. In order to prove that \eqref{eq:p1} converges to zero we apply Theorem \ref{thm:easy_non-pred}. To this end we have to prove that \eqref{eq:cond1}-\eqref{eq:cond5} with $\delta_n=h$ hold. This is either true by Assumption (D1) or Lemma \ref{lem:pred_cond_hold}. Thus, we may apply Theorem \ref{thm:easy_non-pred} and thus we can conclude that \eqref{eq:var2} holds.

To prove \eqref{eq:var1} in Lemma \ref{lem_ass:var} we will apply very similar techniques as before. In fact, we can use almost exactly the same steps with $u_2=u_1$ we have taken in order to arrive at \eqref{eq:p1} with one exception: At some point we said that we can ignore the indicator function $\Ind_{t=r}$ because $u_1\neq u_2$, this is not true now and we need to take care of this. We obtain
\begin{align}
&\frac{4}{hr_n^2}\sum_{u_1\in L_n}\int_0^T\underset{u_2\neq u_1}{\sum_{u_2\in L_n}}\tau_{n,u_1u_2}(s)^2C_{n,u_1}(s)\lambda_{n,u_1}(s)ds \nonumber \\
=&\frac{8}{r_n}\sum_{u\in L_n}\int_0^T\int_{t-2h}^{t-}\phi_{n,u}(t,r)dM_{n,u}(r)dM_{n,u}(t) \label{eq:offdiag} \\
&+\frac{4}{r_n}\sum_{u\in L_n}\int_0^T\phi_{n,u}(t)\cdot\Delta M_{n,u}(t)dM_{n,u}(t), \label{eq:diag}
\end{align}
where we used the abbreviations $\phi_{n,u}(r,t):=\phi_{n,uu}(r,t)$ and $\phi_{n,u}(t):=\phi_{n,u}(t,t)$. We prove that \eqref{eq:offdiag} converges to zero in probability by applying similar techniques as before. We start by approximating $\phi_{n,u}$ by its measurable approximation $\tilde{\phi}_{n,u}^u$:
\begin{align}
&\eqref{eq:offdiag} \nonumber \\
=&\frac{8}{r_n}\sum_{u\in L_n}\int_0^T\int_{t-2h}^{t-}\phi_{n,u}(t,r)-\tilde{\phi}_{n,u}^u(t,r)dM_{n,u}(r)dM_{n,u}(t) \label{eq:offdiag1} \\
&+\frac{8}{r_n}\sum_{u\in L_n}\int_0^T\int_{t-2h}^{t-}\tilde{\phi}_{n,u}^u(t,r)dM_{n,u}(r)dM_{n,u}(t). \label{eq:offdiag2}
\end{align}
We now use again the approximation \eqref{eq:h2} and obtain by using martingale properties (recall that $K_m^{L_n}$ is measurable by Assumption (H1)) and Markov's Inequality for any $\epsilon>0$
\begin{align*}
&\IP(|\eqref{eq:offdiag1}|>\epsilon) \\
\leq&\frac{1}{\epsilon}\IE\left(\frac{8}{r_n}\sum_{u\in L_n}\int_0^T\int_{t-2h}^{t-}\frac{2C^*}{r_np_n(t)^2}\left(F+K_m^{L_n}H_{UB}^u\right)d|M_{n,u}|(r)d|M_{n,u}|(t)\right) \\
\leq&\frac{16C^*\Lambda}{\epsilon r_np_n}\int_0^T\IE\left(\int_{t-2h}^{t-}\left(F+K_m^{L_n}H_{UB}^u\right)d|M_{n,u}|(r)\Big|C_{n,u}(t)=1\right)dt \\
\to&0
\end{align*}
by Assumption (H2, \ref{eq:AD9}).

For \eqref{eq:offdiag2}, we recall that $\tilde{\phi}_{n,u}^u$ is partially predictable with respect to $\mathcal{F}_t^{n,u,m}$ (cf. remark after \eqref{eq:phi_tilde}) and thus we may apply Lemma \ref{lem:martingale}. Together with \eqref{eq:phi_bound} we get
\begin{align*}
&\IP(|\eqref{eq:offdiag2}|>\epsilon) \\
\leq&\frac{16\Lambda}{\epsilon r_n}\sum_{u\in L_n}\int_0^T\IE\left(\int_{t-2h}^{t-}|\phi_{n,u}^u(t,r)|d|M_{n,u}|(r)C_{n,u}(t)\right)dt \\
\leq&\frac{32C^*\Lambda}{\epsilon r_n^2}\sum_{u\in L_n}\int_0^T\IE\left(\frac{1}{p_n^2}\int_{t-2h}^{t-}d|M_{n,u}|(r)\cdot A_n(t)\cdot C_{n,u}(t)\right)dt \\
\leq&\frac{32C^*\Lambda}{\epsilon}\int_0^T\IE\left(\frac{A_n(t)}{r_np_n)}\cdot\int_{t-2h}^{t-}d|M_{n,u}|(r)\Bigg| C_{n,u}(t)=1\right)dt \cdot\frac{\sup_{t\in[0,T]}p_n(t)}{p_n}.
\end{align*}
This converges to zero by conditioning on $A_n(t)/r_np_n>\alpha$ and Assumptions (AD, \ref{eq:AD5}, \ref{eq:AD10}).

We study now the convergence behaviour of \eqref{eq:diag}. Note firstly that
\begin{align}
\eqref{eq:diag}=&\frac{4}{r_n}\sum_{u\in L_n}\int_0^T\phi_{n,u}(t)dM_{n,u}(t) \label{eq:diag1} \\
&+\frac{4}{r_n}\sum_{u\in L_n}\int_0^T\phi_{n,u}(t)C_{n,u}(t)\lambda_{n,u}(t)dt \label{eq:diag2}
\end{align}
The first part, \eqref{eq:diag1}, converges to zero by an application of Proposition \ref{prop:single_non-pred}. We get by said Proposition with $\tilde{\phi}_{n,u}^I(t):=\tilde{\phi}_{n,u}^I(t,t)$
\begin{align}
&\IE\left(\eqref{eq:diag1}^2\right) \nonumber \\
=&\frac{16\Lambda}{r_n^2}\sum_{u_1\in L_n}\int_0^T\IE\left(\tilde{\phi}_{n,u_1}^{u_1}(t)^2C_{n,u_1}(t)\right)dt \label{eq:diag11} \\
&+\frac{32}{r_n^2}\sum_{u_1,u_2\in L_n}\IE\left(\int_0^T\tilde{\phi}_{n,u_1}^{u_1u_2}(t)dM_{n,u_1}(t)\cdot\int_0^T\left(\phi_{n,u_2}(t)-\tilde{\phi}_{n,u_2}^{u_1u_2}(t)\right)dM_{n,u_2}(t)\right) \label{eq:diag12} \\
&+\frac{16}{r_n^2}\sum_{u_1,u_2\in L_n}\IE\left(\int_0^T\left(\phi_{n,u_1}(t)-\tilde{\phi}_{n,u_1}^{u_1u_2}(t)\right)dM_{n,u_1}(t)\cdot\int_0^T\left(\phi_{n,u_2}(t)-\tilde{\phi}_{n,u_2}^{u_1u_2}(t)\right)dM_{n,u_2}(t)\right). \label{eq:diag13}
\end{align}
We apply estimates \eqref{eq:h2} and \eqref{eq:phi_bound} to show that the three lines above converge to zero. We have by exchangeability of the network
\begin{align*}
\eqref{eq:diag11}\leq&16\Lambda\int_0^T\frac{p_n(t)}{r_np_n^2}\IE\left(\left(\frac{2C^*A_n(t)}{r_np_n}\right)^2\Bigg|C_{n,u}(t)=1\right)dt \\
\leq&64(C^*)^2\Lambda\int_0^T\frac{p_n(t)}{r_np_n^2}\IE\left(\left(\frac{A_n(t)}{r_np_n}\right)^3\right)dt.
\end{align*}
Convergence to zero of the above expression is implied by the fact that the third moment of $A_n(t)$ exist by Assumption (AD, \ref{eq:AD11}). We continue with \eqref{eq:diag12} and \eqref{eq:diag13} to get
\begin{align*}
&\eqref{eq:diag12} \\
\leq&\frac{32}{r_n^2}\sum_{u_1,u_2\in L_n}\IE\left(\int_0^T\frac{2C^*A_n(t)}{r_np_n^2}d|M_{n,u_1}|(t)\cdot \int_0^T\frac{4C^*\left(F+K_m^{L_n}H_{UB}^{u_1u_2}\right)}{r_np_n(t)^2}d|M_{n,u_2}|(t)\right) \\
\leq&\frac{256(C^*)^2}{r_n^2p_n^2}\IE\left(\left(\sum_{u_1\in L_n}\int_0^T\left(F+K_m^{L_n}H_{UB}^{u_1}\right)d|M_{n,u_1}|(t)\right)^2\sup_{t\in[0,T]}\left(\frac{A_n(t)}{r_np_n}\right)^2\right)
\end{align*}
which converges to zero by conditioning and Assumptions (H2, \ref{eq:HubS7}) and (AD, \ref{eq:AD5}, \ref{eq:AD8}). Moreover,
\begin{align*}
\eqref{eq:diag13}\leq&\frac{32}{r_n^2}\sum_{u_1,u_2\in L_n}\IE\Bigg(\int_0^T\frac{4C^*}{r_np_n(t)^2}\left(F+K_m^{L_n}H_{UB}^{u_1u_2}\right)d|M_{n,u_1}|(t) \\
&\quad\quad\quad\times \int_0^T\frac{4C^*}{r_np_n(t)^2}\left(F+K_m^{L_n}H_{UB}^{u_1u_2}\right)d|M_{n,u_2}|(t)\Bigg) \\
\leq&\frac{512(C^*)^2}{r_n^2p_n^2}\IE\left(\left(\sum_{u\in L_n}\int_0^T\frac{\mathcal{C}_{n,u}}{r_np_n(t)}d|M_{n,u}|(t)\right)^2\right).
\end{align*}
It can be shown that the above converges to zero by using martingale properties and the Assumptions (H2, \ref{eq:HubS7}), \ref{eq:HubS5} and (D2, \ref{CA:AU2}). And we conclude that \eqref{eq:diag1} converges to zero.

So we have left to prove convergence of \eqref{eq:diag2} which we do as follows: Denote by superscripts entries of vectors or matrices, i.e., $X^{r_1}_{n,u}(t)^2$ refers to the square of the $r_1$-th entry of $X_{n,u}(t)$ and $\tilde{f}^{r_1,r_2}_n(t+\xi h,t)$ refers to the entry in row $r_1$ and column $r_2$ of the matrix $\tilde{f}_n(t+\xi h,t)$. Then
\begin{align}
&\eqref{eq:diag2} \\
=&\frac{4}{r_n}\sum_{u_1\in L_n}\int_0^T\phi_{n,u_1}(t,t)C_{n,u_1}(t)\lambda_{n,u_1}(t)dt \nonumber \\
=&4\underset{r'_1,r'_2=1}{\sum_{r_1,r_2=1}^q}\int_0^T\int_0^2\frac{1}{r_n}\sum_{u_2\in L_n}\Delta_{n,u_2}^{r_1r_1',r_1r_2}(t+\xi h,t,t+\xi h)\cdot\frac{1}{r_n}\sum_{u_1\in L_n}\Delta_{n,u_1}^{r_2r_2',r_1'r_2'}(t,t,t+\xi h)d\xi dt \nonumber \\
&-\frac{4}{r_n}\sum_{r_1,r_2,r'_1,r'_2=1}^q\int_0^T\int_0^2\frac{1}{r_n}\sum_{u_2\in L_n}\Delta_{n,u_2}^{r_1r_1',r_1r_2}(t+\xi h,t,t+\xi h)\Delta_{n,u_2}^{r_2r_2',r_1'r_2'}(t,t,t+\xi h)d\xi dt, \label{eq:idontcare}
\end{align}
where for all $a,b,c,d\in\{1,...,q\}$
\begin{align*}
\Delta_{n,u}^{ab,cd}(\tau,t,s)&:=X_{n,u}^{a}(\tau)X_{n,u}^{b}(\tau)\tilde{f}^{c,d}_n(s,t)C_{n,u}(\tau)\lambda_{n,u}(\tau) \\
\tilde{\Delta}_{n,u}^{ab,cd}(\tau,t,s)&:=\Delta_{n,u}^{ab,cd}(\tau,t,s)-\IE\left(\Delta_{n,u}^{ab,cd}(\tau,t,s)\right) \\
&=\tilde{f}_n^{c,d}(s,t)\Big[X_{n,u}^a(\tau)X_{n,u}^b(\tau)C_{n,u}(\tau)\lambda_{n,u}(\tau) \\
&\quad\quad\quad\quad-\IE\left(X_{n,u}^a(\tau)X_{n,u}^b(\tau)C_{n,u}(\tau)\lambda_{n,u}(\tau)\right)\Big].
\end{align*}
We keep this in mind and prove now for all $r_1,r_1',r_2,r_2'\in\{1,...,q\}$
\begin{align}
&\sup_{t\in[0,T],\sigma,\tau\in[0,2]}\Bigg|\frac{1}{r_n}\sum_{u\in L_n}\tilde{\Delta}_{n,i}^{r_1r_1',r_1r_2}(t+\tau h,t,t+\sigma h)\Bigg|=o_P(1) \label{eq:usop}
\end{align}
via exponential inequality techniques. Since the argument is the same for all indices, we omit $r_1,r_1',r_2,r_2'$ in the notation. Let therefore $\IT_n$ denote a grid of $[0,T]\times[0,2]^2$ with mesh $H_{n,p}^{-\frac{1}{\alpha_p}}n^{-k_X}$ (where $k_X$ is chosen later) and let $(t^*,\sigma^*,\tau^*)$ be the projection of $(t,\sigma,\tau)\in[0,T]\times[0,2]^2$ onto $\IT_n$, i.e., $\|(t,\sigma,\tau)-(t^*,\sigma^*,\tau^*)\|\leq H_{n,p}^{-\frac{1}{\alpha_p}}n^{-k_X}$. We obtain
\begin{align}
&\sup_{t\in[0,T],\sigma ,\tau \in[0,2]}\Bigg|\frac{1}{r_n}\sum_{u\in L_n}\tilde{\Delta}_{n,u}(t+\tau h,t,t+\sigma h)\Bigg| \nonumber \\
\leq&\sup_{t\in[0,T],\sigma ,\tau \in[0,2]}\Bigg|\frac{1}{r_n}\sum_{u\in L_n}\left(\tilde{\Delta}_{n,u}(t+\tau h,t,t+\sigma h)-\tilde{\Delta}_{n,u}(t^*+\tau^*h,t^*,t^*+\sigma^*h)\right)\Bigg| \label{eq:last1} \\
&+\sup_{t\in[0,T],\sigma ,\tau \in[0,2]}\Bigg|\frac{1}{r_n}\sum_{u\in L_n}\tilde{\Delta}_{n,u}(t^*+\tau^*h,t^*,t^*+\sigma^*h)\Bigg|. \label{eq:last2}
\end{align}
For \eqref{eq:last1} we define $\xi_{n,u}(t):=X_{n,u}(t)X_{n,u}(t)C_{n,u}(t)\lambda_{n,u}(t)$. Then, we find
\begin{align*}
&\left|\frac{1}{r_n}\sum_{u\in L_n}\left(\tilde{\Delta}_{n,u}(t+\tau h,t,t+\sigma h)-\tilde{\Delta}_{n,u}(t^*+\tau^*h,t^*,t^*+\sigma^*h)\right)\right| \\
\leq&\left|p_n(t+\tau h)\tilde{f}_n(t+\sigma h,t)-p_n(t^*+\tau^*h)\tilde{f}_n(t^*+\sigma^*h,t^*)\right| \\
&\quad\quad\quad\times\left|\frac{1}{r_n}\sum_{u\in L_n}\left(\frac{\xi_{n,u}(t+\tau h)}{p_n(t+\tau h)}-\IE\left(\frac{\xi_{n,u}(t+\tau h)}{p_n(t+\tau h)}\right)\right)\right| \\
&+p_n(t^*+\tau^*h)\tilde{f}_n(t^*+\sigma^*h,t^*)\left|\frac{1}{r_n}\sum_{u\in L_n}\left(\frac{\xi_{n,u}(t+\tau h)}{p_n(t+\tau h)}-\frac{\xi_{n,u}(t^*+\tau^*h)}{p_n(t^*+\tau^*h)}\right)\right| \\
&+p_n(t^*+\tau^*h)\tilde{f}_n(t^*+\sigma^*h,t^*)\left|\IE\left(\frac{\xi_{n,u}(t+\tau h)}{p_n(t+\tau h)}-\frac{\xi_{n,u}(t^*+\tau^*h)}{p_n(t^*+\tau^*h)}\right)\right|.
\end{align*}
By the Assumptions (A1), (A5), (A6) and Lemma \ref{lem:p_continuous} we have that the functions $\Sigma_t$, $w$, $p_n$ and $\bar{p}_n$ are all continuous on the compact interval $[0,T]$. Therefore $p_n$ and $\tilde{f}_n$ are uniformly continuous on $[0,T]$. Hence, we can choose $k_X$ large enough such that the first term converges to zero (recall also that the covariates are bounded by Assumption (A3, \ref{ass:boundedness}). The second and third term converge to zero by Assumptions (AC, \ref{eq:AC1}, \ref{eq:AC2}), respectively, after possibly increasing $k_X$ further. Keep this choice of $k_X$ fixed for the remainder of the proof.

Lastly we need to discuss \eqref{eq:last2}. To this end, we apply a standard union bound technique together with Lemma \ref{lem:sufficient_mixing}. We can estimate when noting the sup in \eqref{eq:last2} is actually only taken over $\IT_n$ that for every $\epsilon>0$ by \eqref{eq:ftildebound}
\begin{align*}
&\IP\left(\eqref{eq:last2}>\epsilon\right) \\
=&\IP\Bigg(\sup_{(t,\sigma ,\tau)\in\IT_n}\left|\tilde{f}_n(t+\sigma h,t)\right|\cdot\frac{1}{r_n}\Bigg|\sum_{u\in L_n}\Big[\xi_{n,u}(t+\tau h)-\IE\left(\xi_{n,u}(t+\tau h)\right)\Big]\Bigg|>\epsilon\Bigg) \\
\leq&\IP\Bigg(\sup_{(t,\sigma ,\tau)\in\IT_n}\frac{1}{r_np_n(t)}\Bigg|\sum_{u\in L_n}\Big[\xi_{n,u}(t+\tau h)-\IE\left(\xi_{n,u}(t+\tau h)\right)\Big]\Bigg|>\frac{\epsilon}{c}\Bigg) \\
\leq&\#\IT_n\cdot\sup_{(t,\sigma ,\tau)\in\IT_n}\IP\left(\frac{1}{r_np_n(t)}\sum_{u\in L_n}\Big[\xi_{n,u}(t+\tau h)-\IE\left(\xi_{n,u}(t+\tau h)\right)\Big]>\frac{\epsilon}{c}\right) \\
&+\#\IT_n\cdot\sup_{(t,\sigma ,\tau)\in\IT_n}\IP\left(\frac{1}{r_np_n(t)}\sum_{u\in L_n}\Big[\xi_{n,u}(t+\tau h)-\IE\left(\xi_{n,u}(t+\tau h)\right)\Big]<-\frac{\epsilon}{c}\right).
\end{align*}
The two lines above work completely analogously and hence, we continue only with the first line. The proof of the second line is then identical, we just have to replace $\xi_{n,u}$ by $-\xi_{n,u}$. We will also replace now $\frac{\epsilon}{C}$ by $\epsilon$ for notational convenience, i.e., we show that
\begin{align}
&\#\IT_n\cdot\sup_{(t,\sigma,\tau )\in\IT_n}\IP\left(\frac{1}{r_np_n(t)}\sum_{u\in L_n}\Big[\xi_{n,u}(t+\tau h)-\IE\left(\xi_{n,u}(t+\tau h)\right)\Big]>\epsilon\right)\to0. \label{eq:last21}
\end{align}
To this end, we apply Lemma \ref{lem:sufficient_mixing} to the array of random variables
$$Z_{n,u}=\xi_{n,u}(t)$$
which is bounded by $M:=\hat{K}^2\Lambda$. Consider the sequence of $\Delta_n$-partitions as in Assumption (D3). Since $\xi_{n,u}(t)=\xi_{n,u}(t)C_{n,u}(t)$, we have that $(Z_{n,u})_{u\in L_n}$ fulfils \eqref{eq:GeneralCoverCondition} with respect to this $\Delta_n$-partitioning. The further requirements of Lemma \ref{lem:sufficient_mixing} on the partitioning were required in (D3). The asymptotic uncorrelation condition of Lemma \ref{lem:sufficient_mixing} holds by Assumption (D3, \ref{eq:UUM1}). Note that we show in the proof of Lemma \ref{lem:exponential_inequality} that $|E|_{n,t}=r_n\bar{p}_n(t)$. Thus, we may apply Lemma \ref{lem:sufficient_mixing} with $M=\hat{K}^2\Lambda$. We use that $\frac{p_n}{p_n^*}>\frac{1}{c_0}$ for some $c_0>0$ by Assumption (A6) and that $x\mapsto \frac{x}{\log x}$ is monotonically increasing to obtain
\begin{align*}
&\eqref{eq:last21} \\
\leq&\#\IT_n\left(2\mathcal{K}_n(r_np_n)^{-\frac{c_2\frac{\epsilon^2}{c_0^2}\sqrt{\frac{r_np_n}{\log r_np_n}}}{2\left(3CM^2+4Mc_3\frac{\epsilon}{c_3}\right)}}+\beta_t(\Delta_n)\cdot\mathcal{K}_nr_n+\IP(\Gamma_n^t=0)\right).
\end{align*}
The exponent of $r_np_n$ can be chosen arbitrarily small. By Assumption (D3) $\IP(\Gamma_n^t=0)\to0$ converges to zero exponentially fast and we can choose $\Delta_n$ such that $\beta_t(\Delta_n)$ vanishes as fast as we want. Thus, also the product with $\#\IT_n$ converges to zero. This was the last piece for establishing \eqref{eq:usop}.

We can now continue to compute \eqref{eq:diag2} or equivalently \eqref{eq:idontcare}. Note that by \eqref{eq:ftildebound} it holds uniformly over all indices that
\begin{align*}
\frac{1}{r_n}\Delta^{r_2r_2',r_1r_2'}_{n,i}\leq\frac{1}{r_np_n}C\hat{K}^2\Lambda\to0
\end{align*}
as $n\to\infty$. Thus, the second line in \eqref{eq:idontcare} converges to zero. The limit of the first line of \eqref{eq:idontcare} is by using \eqref{eq:usop} the same as
\begin{align*}
&4\sum_{r_1,r_2,r_1',r_2'=1}^q\int_0^T\int_0^2\IE\left(\Delta_{n,j}^{r_1r_1',r_1r_2}(t+uh,t,t+uh)\right)\cdot\IE\left(\Delta_{n,j}^{r_2r_2',r_1'r_2'}(t,t,t+uh)\right)dudt \\
\to&4\sum_{r_1,r_2,r_1',r_2'=1}^qK^{(4)}\int_0^T\Sigma_t^{r_1,r_1'}\Sigma_t^{r_2,r_2'}[\Sigma^{-2}]^{r_1,r_2}(t)[\Sigma^{-T}\Sigma^{-1}]^{r'_1,r'_2}(t)w(t)^2dt \\
=&4K^{(4)}\int_0^T\textrm{trace}\left(\Sigma_t^{-2}\right)w^2(t)dt,
\end{align*}
where we used continuity of $\Sigma_t$ from Assumption (A5) and where
\begin{align*}
K^{(4)}&:=\int_0^2\left(\int_{-1}^1K(v)K(u+v)dv\right)^2du.
\end{align*}
Thus we have proven \eqref{eq:var1} and the proof of Lemma \ref{lem_ass:var} is complete.
\end{proof}

\begin{lemma}
\label{lem:help2}
Suppose that (A1), (A3, \ref{ass:boundedness}), (A4, \ref{ass:kernel_hoelder}) and (A6) hold. Then, there is a constant $C^*>0$ such that for all $I,J\subseteq L_n$ and all $r,t\in[\delta,T-\delta]$ with $r\in[t-2h,t]$ it holds for $\tilde{\phi}$ defined in \eqref{eq:phi_tilde} that
\begin{equation}
\sup_{u_1,u_2}\left|\tilde{\phi}_{n,u_1u_2}^{I}(t,r)-\tilde{\phi}_{n,u_1u_2}^J(t,r)\right|\leq\frac{2C^*|I\Delta J|}{r_np_n(t)^2}\left(F+K_m^{L_n}H_{UB}^{I\Delta J}\right), \label{eq:h2}
\end{equation}
where $I\Delta J:=(I\setminus J) \cup (I\setminus J)$ denotes the symmetric difference of $I$ and $J$ and $K_m^{L_n}$ and $H_{UB}^{I\Delta J}$ are to be understood with respect to $U_t:=[t-4h,t+2h]$.
\end{lemma}
\begin{proof}
By Lemma \ref{lem:bounded_inverse}, Assumptions (A1), (A4, \eqref{ass:kernel_hoelder}) and (A6) we find a constant $c>0$ such that for all $r\in[t-2h,t]$ and $t\in[\delta,T-\delta]$ (note that then $\frac{r-t}{h}\leq 0$)
\begin{align}
\|\tilde{f}_n(t+yh,r)\|\leq&M^2\|w\|_{\infty}\int_{-\infty}^{\infty}K(x)K\left(\frac{r-t}{h}+x-y\right)\frac{1}{\bar{p}_n(t+h(y-x))}dx \nonumber \\
\leq&\frac{c}{p_n(t)}. \label{eq:ftildebound}
\end{align}
By the assumption of bounded covariates (A3, \ref{ass:boundedness}), we get for any index sets $I,J\subseteq L_n$ and $r\in[t-2h,t]$ (i.e. $\frac{r-t}{h}\in[-2,0]$), that
\begin{align}
&\sup_{u_1,u_2\in L_n}\left|\tilde{\phi}_{n,u_1u_2}^{I}(t,r)-\tilde{\phi}_{n,u_1u_2}^J(t,r)\right| \nonumber \\
\leq&\frac{C^*}{r_np_n(t)^2}\int_0^{\frac{r-t}{h}+2}\sum_{u\in L_n}C_{n,u}(t+yh)\left|\Ind(d_{t-4h}^n(u,I)\geq m)-\Ind(d_{t-4h}^n(u,J)\geq m)\right|dy \nonumber \\
\leq&\frac{C^*}{r_np_n(t)^2}\int_0^{\frac{r-t}{h}+2}\sum_{u\in L_n}C_{n,u}(t+yh)\sum_{k\in I\Delta J}\Ind(d_{t-4h}^n(u,k)\leq m)dy \nonumber \\
\leq&\frac{2C^*}{r_np_n(t)^2}\sum_{k\in I\Delta J}\Bigg(\sum_{u\in L_n}\sup_{y\in [0,2]}C_{n,u}(t+yh)\Ind(d_{t-4h}^n(u,k)\leq m)\Ind(K_m^k\leq F) \nonumber \\
&+\quad\sum_{u\in L_n}\sup_{y\in[0,2]}C_{n,u}(t+yh)\Ind(d_{t-4h}^n(u,k)\leq m)\Ind(K_m^k>F)\Bigg) \nonumber \\
\leq&\frac{2C^*|I\Delta J|}{r_np_n(t)^2}\left(F+K_m^{L_n}H_{UB}^{I\Delta J}\right) \nonumber
\end{align}
\end{proof}

\begin{proof}[Proof of Lemma \ref{lem_ass:jump}]
Let $\epsilon>0$ be arbitrary. We have to show that a martingale evaluated at a certain time point $T$ converges to zero in probability. By Lenglart's Inequality as in Corollary \ref{cor:Lenglart} it is sufficient to prove that the quadratic variation converges to zero in probability. Simply taking the $\sup$ yields for the quadratic variation
\begin{align*}
&\frac{4}{hr_n^2}\sum_{u\in L_n}\int_0^T\Ind\left(\left\|\frac{2}{h^{\frac{1}{2}}r_n}\sum_{v\neq u}\tau_{n,uv}(s)\right\|>\epsilon\right)\left(\sum_{v\neq u}\tau_{n,uv}(s)\right)^2\lambda_{n,u}(s)ds \\
\leq&\Ind\left(\sup_{s\in[0,T], u\in L_n}\left\|\frac{2}{h^{\frac{1}{2}}r_n}\sum_{v\neq u}\tau_{n,uv}(s)\right\|>\epsilon\right)\frac{4}{hr_n^2}\sum_{u\in L_n}\int_0^T\left(\sum_{v\neq u}\tau_{n,uv}(s)\right)^2\lambda_{n,u}(s)ds.
\end{align*}
Lemma \ref{lem_ass:var} is stating that the second part is converging and hence it is sufficient to prove that the indicator function is converging to zero in probability which is equivalent of proving uniform convergence in probability (uniform in $u$ and $s$) of
$$\frac{1}{h^{\frac{1}{2}}r_n}\sum_{v\neq u}\tau_{n,uv}(s)$$
to zero. We are going to employ Lemma \ref{lem:exponential_inequality}. To this end note firstly that $\tau_{n,uv}(s)$ has the following structure
$$\tau_{n,uv}(s)=X_{n,u}(s)^T\int_0^{s-}\tilde{f}_n(s,t)X_{n,v}(t)dM_{n,v}(t),$$
where $\tilde{f}_n(s,t)$ is defined in \eqref{eq:ftilde} and can be written as
$$\tilde{f}_n(s,t)=\int_{\delta}^{T-\delta}hK_{h,t_0}(s)K_{h,t_0}(t)\Sigma_{t_0}^{-T}\Sigma_{t_0}^{-1}\frac{w(t_0)}{\bar{p}_n(t_0)}.$$
We can simplify the expression by interchanging the integrals, taking the norm inside and using the boundedness properties from Assumption (A3, \ref{ass:boundedness}) and Lemma \ref{lem:bounded_inverse}:
\begin{align}
&\sup_{s\in[0,T], u\in L_n}\frac{1}{h^{\frac{1}{2}}r_n}\left\|\sum_{v\neq u}\tau_{n,uv}(s)\right\| \nonumber \\
\leq&\sup_{s\in[0,T], u\in L_n}h^{\frac{1}{2}}\int_{\delta}^{T-\delta}\|X_{n,u}(s)\|K_{h,t_0}(s)\left\|\Sigma_{t_0}^{-T}\Sigma_{t_0}^{-1}\right\|w(t_0) \nonumber \\
&\quad\quad\times\frac{1}{r_n\bar{p}_n(t_0)}\sum_{v\neq u}\int_0^{s-}K_{h,t_0}(t)\|X_{n,v}(t)\|d|M_{n,v}|(t)dt_0. \nonumber \\
\leq&h^{\frac{1}{2}}\cdot\sup_{s\in[0,T]}\int_{\delta}^{T-\delta}K_{h,t_0}(s)\frac{\hat{K}M^2\|w\|_{\infty}}{r_n\bar{p}_n(t_0)}\sum_{v\in L_n}\int_0^TK_{h,t_0}(t)\|X_{n,v}(t)\|d|M_{n,v}|(t)dt_0 \nonumber \\
\leq&h^{\frac{1}{2}}\cdot\sup_{t_0\in\IT}\frac{\hat{K}M^2\|w\|_{\infty}T}{r_n\bar{p}_n(t_0)}\sum_{v\in L_n}\int_0^TK_{h,t_0}(t)\|X_{n,v}(t)\|dM_{n,v}(t) \label{eq:lem212ready} \\
&+h^{\frac{1}{2}}\cdot\sup_{t_0\in\IT}\frac{2\hat{K}^2M^2\Lambda\|w\|_{\infty}T}{r_n\bar{p}_n(t_0)}\sum_{v\in L_n}\int_0^TK_{h,t_0}(t)C_{n,v}(t)dt. \label{eq:cubready}
\end{align}
Now, $\eqref{eq:lem212ready}=o_P(1)$ because the expression in \eqref{eq:lem212ready} is the same as in $\ell_n'(\theta_0,t_0)$ but with $X_{n,v}(t)$ replaced by $\|X_{n,v}(t)\|$. Moreover, all mixing properties valid for $X_{n,v}(t)$ hold for $\|X_{n,v}(t)\|$ as well and, of course, $\|X_{n,v}(t)\|$ is also bounded. Thus, we may repeat the proof of Proposition \ref{lem:1} and all subsidiary results (which proofs do not require this Lemma) word by word and \eqref{eq:lem212ready} converges to zero in probability. We also have that \eqref{eq:cubready} is $O_P\left(h^{\frac{1}{2}}\right)=o_P(1)$ by the later proven equation \eqref{eq:cub}. Hence, we have shown that
$$\sup_{s\in[0,T], u\in L_n}\frac{1}{h^{\frac{1}{2}}r_n}\left\|\sum_{v\neq u}\tau_{n,uv}(s)\right\|=o_P(1)$$
and this finalizes the proof of the Lemma.
\end{proof}

\begin{proof}[Proof of Lemma \ref{lem_ass:K1}]
Consider the following event
$$A_n:=\left\{\forall t_0\in[\delta,T-\delta]:\,\left\|\frac{1}{r_n\bar{p}_n(t_0)}\ell_n''(\theta_0,t_0)-\Sigma(t_0,\theta_0)\right\|\leq\rho\right\},$$
where $\rho$ is the same as in Lemma \ref{lem:bounded_inverse} and suppose for the moment that $\IP(A_n)\to1$. On $A_n$, we find
\begin{align*}
&\left\|\left[\frac{1}{r_n\bar{p}_n(t_0)}\ell_n''(\theta_0,t_0)\right]^{-1}-\Sigma(t_0,\theta_0)^{-1}\right\| \\
\leq&\left\|\left[\frac{1}{r_n\bar{p}_n(t_0)}\ell_n''(\theta_0,t_0)\right]^{-1}\right\|\cdot \left\|\Sigma(t_0,\theta_0)-\frac{1}{r_n\bar{p}_n(t_0)}\ell_n''(\theta_0,t_0)\right\|\cdot\left\|\Sigma(t_0,\theta_0)^{-1}\right\| \\
\leq&M^2\rho
\end{align*}
by Lemma \ref{lem:bounded_inverse}. Hence, we conclude
$$\sup_{t_0\in[\delta,T-\delta]}\left\|\left[\frac{1}{r_n\bar{p}_n(t_0)}\ell_n''(\theta_0,t_0)\right]^{-1}\right\|=O_P(1)$$
as required. In order to prove $\IP(A_n)\to1$, we have to prove the uniform convergence of $\frac{1}{r_n\bar{p}_n(t_0)}\ell_n''(\theta_0,t_0)$ to $\Sigma(\theta_0,t_0)$. Denote therefore by $T_n$ a grid of $\IT:=[\delta,T-\delta]$ with mesh $n^{-k}$ for some $k$ and let for $t_0\in\IT$ be $t_0^*$ be the projection of $t_0$ on $T_n$, i.e., we have $|t_0-t_0^*|\leq n^{-k}$. Then we obtain
\begin{align}
&\sup_{t_0\in\IT}\left\|\frac{1}{r_n\bar{p}_n(t_0)}\ell_n''(\theta_0,t_0)-\Sigma(\theta_0,t_0)\right\| \nonumber \\
\leq&\sup_{t_0\in\IT}\left\|\frac{1}{r_n\bar{p}_n(t_0)}\ell_n''(\theta_0,t_0)-\frac{1}{r_n\bar{p}_n(t_0^*)}\ell_n''(\theta_0,t_0^*)\right\|+\sup_{t_0\in\IT}\left\|\Sigma(\theta_0,t_0)-\Sigma(\theta_0,t_0^*)\right\| \label{eq:L29cont} \\
&\quad+\sup_{t_0\in\IT}\left\|\frac{1}{r_n\bar{p}_n(t_0^*)}\ell_n''(\theta_0,t_0^*)-\Sigma(\theta_0,t_0^*)\right\|. \label{eq:L29exp}
\end{align}
The second $\sup$ in \eqref{eq:L29cont} is converging to zero for $k$ chosen large enough because $\sup_{t_0\in\IT}|t_0-t_0^*|\to0$ as $n\to\infty$ and by uniform continuity of $t\mapsto\Sigma(\theta_0,t)$ (cf. Assumption (A5)). To prove that the first part of \eqref{eq:L29cont} is $o_P(1)$, we note that by the boundedness Assumption (A3, \ref{ass:boundedness})
\begin{align*}
&\left\|\frac{1}{r_n\bar{p}_n(t_0)}\ell_n''(\theta_0,t_0)-\frac{1}{r_n\bar{p}_n(t_0^*)}\ell_n''(\theta_0,t_0^*)\right\| \\
\leq&\hat{K}^2e^{\tau\cdot\hat{K}}\int_0^T\left|\frac{1}{h\cdot \bar{p}_n(t_0)}K\left(\frac{s-t_0}{h}\right)-\frac{1}{h\cdot \bar{p}_n(t_0^*)}K\left(\frac{s-t_0^*}{h}\right)\right|ds.
\end{align*}
Since $K$ and $\bar{p}_n$ are Hoelder continuous by Assumption (A4, \ref{ass:kernel_hoelder}) and Lemma \ref{lem:p_continuous}, respectively, the above converges to zero as $n\to\infty$ after possibly increasing $k$ further. For this choice of $k$, which we keep fixed from now on, $\eqref{eq:L29cont}=o_P(1)$. So finally, we have to prove that \eqref{eq:L29exp} is also $o_P(1)$. To this end, we firstly note that the $\sup$ is actually only taken over $T_n$ because we only consider $t_0^*$. So we apply a standard union bound technique to get the $\sup$ out of the probability and we include $\Gamma_n^{t_0}$: Let $x>0$ and recall the Definition of $H_{n,u}(s,\theta)=-C_{n,u}(s)X_{n,u}(s)X_{n,u}(s)^Te^{\theta^TX_{n,u}(s)}$ from the proof of Proposition \ref{lem:2}. Then,
\begin{align}
&\IP\left(\sup_{t_0\in T_n}\left\|\frac{1}{r_n\bar{p}_n(t_0)}\ell_n''(\theta_0,t_0)-\Sigma(\theta_0,t_0)\right\|>x\right) \nonumber \\
\leq&\# T_n\cdot\sup_{t_0\in\IT}\IP\left(\left\|\frac{1}{r_n\bar{p}_n(t_0)}\ell_n''(\theta_0,t_0)-\Sigma(\theta_0,t_0)\right\|>x\right) \nonumber \\
\leq&\# T_n\cdot\sup_{t_0\in\IT}\IP\left(\left\|\frac{1}{r_n\bar{p}_n(t_0)}\sum_{u\in L_n}\int_0^TK_{h,t_0}(s)\left(H_{n,u}(s,\theta_0)-\IE\left(H_{n,i}(s,\theta_0)\right)\right)ds\right\|>\frac{x}{2}\right) \label{eq:smcs2} \\
&+\# T_n\cdot\sup_{t_0\in\IT}\IP\left(\left\|\int_0^TK_{h,t_0}(s)\Sigma_s\frac{p_n(s)}{\bar{p}_n(t_0)}ds-\Sigma_{t_0}\right\|\geq\frac{x}{2}\right). \label{eq:smcs4}
\end{align}
The probability in line \eqref{eq:smcs4} equals zero for $n$ large enough because for $t_0\in[\delta,T-\delta]$ we have by the definition of $\bar{p}_n(s)$
\begin{align*}
&\int_0^TK_{h,t_0}(s)\Sigma_s\frac{p_n(s)}{\bar{p}_n(t_0)}ds-\Sigma_{t_0} \\
=&\int_0^TK_{h,t_0}(s)\left(\Sigma_s-\Sigma_{t_0}\right)\frac{p_n(s)}{\bar{p}_n(t_0)}ds=\eqref{eq:L3}=o(1).
\end{align*}
Finally, line \eqref{eq:smcs2} will be treated by applying Lemma \ref{lem:sufficient_mixing}. Note therefore firstly that we may work element-wise because we can estimate the norm from above by the 1-norm and consider each term separately (note that the dimension of the covariates is not increasing). Thus, we may pretend for the following that $H_{n,u}(s,\theta_0)$ is a number rather than a matrix. Moreover, we can repeat the following proof word by word for $-H_{n,u}(s,\theta_0)$ and thus we may consider $H_{n,u}(s,\theta_0)$ instead of $|H_{n,u}(s,\theta)|$. We apply Lemma \ref{lem:sufficient_mixing} to
$$Z_{n,u}:=\int_0^TK_{h,t_0}(s)H_{n,u}(s,\theta_0)ds.$$
$Z_{n,u}$ is bounded by $M:=K\hat{K}^2\Lambda$ by (A3, \ref{ass:boundedness}). The assumptions on the $\Delta_n$ partitions and asymptotic uncorrelation are fulfilled by Assumption (D3, \ref{eq:UUM2}) with $C=O(1)$. Then, the upper bound provided by Lemma \ref{lem:sufficient_mixing} converges faster to zero than any power of $n$ because of the properties of the bandwidth $h$ in (A4, \ref{ass:bw}) and the properties of $\Gamma_n^{t_0}$ and the $\beta$-mixing coefficients from Assumption (D3). This proves that \eqref{eq:smcs2} converges to zero for any choice of $x>0$ and thus $\eqref{eq:L29exp}=o_P(1)$ and the proof of the Lemma is complete.
\end{proof}

\begin{proof}[Proof of Lemma \ref{lem_ass:K2}]
We note firstly that by Assumption (A3, \ref{ass:boundedness}) for all $t_0\in\IT$ and all $\theta_1,\theta_2\in\Theta$ we can estimate by a Taylor approximation
\begin{align}
&\left\|X_{n,u}(s)X_{n,u}(s)^TC_{n,u}(s)\left(e^{\theta_1^TX_{n,u}(s)}-e^{\theta_2^TX_{n,u}(s)}\right)\right\| \nonumber \\
\leq&\hat{K}^3e^{\tau\hat{K}}\|\theta_1-\theta_2\|. \label{eq:transfer}
\end{align}
Hence, we obtain for all $t_0\in\IT$ and all $\theta_1,\theta_2\in\Theta$
\begin{align*}
&\frac{1}{r_n\bar{p}_n(t_0)}\left\|\ell_n''(\theta_1,t_0)-\ell_n''(\theta_2,t_0)\right\| \\
\leq&\frac{1}{r_n\bar{p}_n(t_0)}\sum_{u\in L_n}\int_0^TK_{h,t_0}(s)C_{n,u}(s)\left\|X_{n,u}(s)X_{n,u}(s)^T\right\|\cdot\left|e^{\theta_1^TX_{n,u}(s)}-e^{\theta_2^TX_{n,u}(s)}\right|ds \\
\leq&\frac{1}{r_n\bar{p}_n(t_0)}\sum_{u\in L_n}\int_0^TK_{h,t_0}(s)C_{n,u}(s)ds\cdot\hat{K}^3e^{\tau \hat{K}}\cdot\|\theta_1-\theta_2\|.
\end{align*}
Consequently, we can choose $K_n:=\sup_{t_0\in\IT}\hat{K}^3e^{\tau \hat{K}}\cdot\frac{1}{r_n\bar{p}_n(t_0)}\sum_{u\in L_n}\int_0^TK_{h,t_0}(s)C_{n,u}(s)ds$ which is $O_P(1)$ if
\begin{equation}
\label{eq:cub}
\sup_{t_0\in\IT}\frac{1}{r_n\bar{p}_n(t_0)}\sum_{u\in L_n}\int_0^TK_{h,t_0}(s)C_{n,u}(s)ds=O_P(1).
\end{equation}
So let us prove this. Denote therefore by $T_n$ a grid with mesh $n^{-k}$ (where $k$ is chosen later) which covers $\IT$. For a given time $t_0\in\IT$ we denote by $t_0^*\in T_n$ the closest element of $T_n$ to $t_0$, i.e., $|t_0-t_0^*|\leq n^{-k}$. Now we split the $\sup$ over an uncountable set as usual in a $\sup$ over close elements and a $\sup$ over a finite set:
\begin{align}
&\sup_{t_0\in\IT}\frac{1}{r_n\bar{p}_n(t_0)}\sum_{u\in L_n}\int_0^T\frac{1}{h}K\left(\frac{s-t_0}{h}\right)C_{n,u}(s)ds \nonumber \\
\leq&\sup_{t_0\in\IT}\int_0^T\frac{1}{r_n}\left[\frac{1}{\bar{p}_n(t_0)}K_{h,t_0}(s)-\frac{1}{\bar{p}_n(t_0^*)}K_{h,t_0^*}(s)\right]\sum_{u\in L_n}C_{n,u}(s)ds \label{eq:s1} \\
&\quad+\sup_{t_0\in\IT}\int_0^T\frac{1}{r_n\bar{p}_n(t_0^*)}K_{h,t_0^*}(s)\sum_{u\in L_n}C_{n,u}(s)ds \label{eq:s3}
\end{align}
We can apply the simple bound $\sum_{u\in L_n}C_{n,u}(s)\leq r_n$ and use Hoelder continuity of $K$ and $\bar{p}_n(t)^{-1}$, cf. Assumption (A4, \ref{ass:kernel_hoelder}) and Lemma \ref{lem:p_continuous} respectively, to see that \eqref{eq:s1} converges to zero in probability. Next, we show that $\eqref{eq:s3}=O_P(1)$ which concludes the proof of the Lemma. We begin, as usual, by taking the $\sup$ out of the probability (and recall that $\bar{p}_n(t_0)=\int_0^TK_{h,t_0}(s)p_n(s)ds$):
\begin{align}
&\IP(\eqref{eq:s3}>x) \nonumber \\
\leq&\#T_n\cdot\sup_{t_0\in T_n}\IP\Bigg(\frac{1}{r_n\bar{p}_n(t_0)}\sum_{u\in L_n}\int_0^TK_{h,t_0}(s)\left(C_{n,u}(s)-p_n(s)\right)ds>x-1\Bigg) \label{eq:s31}
\end{align}
We will apply Lemma \ref{lem:sufficient_mixing} to
$$Z_{n,u}=\int_0^TK_{h,t_0}(s)C_{n,u}(s)ds.$$
As before we use the $\Delta_n$-partitions as provided in Assumption (D3). Then, all requirements of Lemma \ref{lem:sufficient_mixing} on the partitions are fulfilled and $|E|_{n,t_0}=r_n\bar{p}_n(t_0)$. Moreover $Z_{n,u}$ is bounded by $M:=1$ and the asymptotic uncorrelation holds also by Assumption (D3, \ref{eq:UUM3}). As a consequence we can apply Lemma \ref{lem:sufficient_mixing} to obtain an upper bound on \eqref{eq:s31}. Taking into account the assumptions on $\Gamma_n^t$ and the $\beta$ mixing coefficients in Assumption (D3), the upper bound on \eqref{eq:s31} provided by Lemma \ref{lem:sufficient_mixing} converges faster to zero than any power of $n$ for $x>1$. Therefore $\eqref{eq:s31}\to0$ as $n\to\infty$. We conclude that $\eqref{eq:s3}=O_P(1)$.
\end{proof}

\begin{proof}[Proof of Lemma \ref{lem_ass:exp1}]
By employing Lemma \ref{lem:exponential_inequality} the proof of this result is fairly straight forward. Let $c^{**}$ be the constant such that $\|y\|\leq c^{**}\cdot\|y\|_1$ for all $y\in\IR^q$ where $\|.\|$ and $\|.\|_1$ denote the Euclidean- and the 1-Norm, respectively. We have
\begin{align}
&\IP\left(\left\|\frac{\ell_n'(\theta_0,t_0)}{r_n\sqrt{\bar{p}_n(t_0)}}\right\|\geq C\cdot\sqrt{\frac{\log r_n}{r_nh}}\right) \nonumber \\
\leq&\IP\left(\left\|\frac{\ell_n'(\theta_0,t_0)}{r_n\bar{p}_n(t_0)}\right\|\geq\frac{C}{qc^{**}\sqrt{h}}\cdot qc^{**}\sqrt{\frac{\log r_n\bar{p}_n(t_0)}{r_n\bar{p}_n(t_0)}}\right). \label{eq:lem41form}
\end{align}
Since
$$\ell_n'(\theta_0,t_0)=\sum_{u\in L_n}\int_0^TK_{h,t_0}(t)X_{n,u}(t)dM_{n,u}(t),$$
we can directly apply Lemma \ref{lem:exponential_inequality} and obtain
\begin{align*}
&\eqref{eq:lem41form} \\
\leq&2q\Bigg(\Kappa\exp\left(-\frac{c_2\frac{C^2}{q^2(c^{**})^2}}{2K\hat{K}A\left(\Lambda A+\sqrt{\frac{\Lambda}{2}}c_3\cdot\frac{C}{qc^{**}}\right)}\cdot\log np_n(t_0-h)\right) \\
&\quad\quad+\beta_{t_0}(\Delta_n)\cdot\Kappa r_n+\IP(\Gamma_n^{t_0}=0)\Bigg).
\end{align*}
We see that for a $C>0$ chosen sufficiently large the first term decreases faster as $hn^{k_0}$. Moreover, by Assumption (D3), $\beta(\Delta_n)$ decreases fast enough too. Finally, $\IP(\Gamma_n^{t_0}=0)$ decreases fast enough as well by the same assumption.
\end{proof}

\begin{proof}[Proof of Lemma \ref{lem_ass:exp2}]
Let $\delta_n:=\sqrt{\frac{\log r_np_n}{r_np_n\cdot h}}$. We begin with the standard union bound argument:
\begin{align}
&\IP\Bigg(\sup_{(t_0,\theta)\in T_{n,k_0}}\Bigg|\frac{1}{r_n\bar{p}_n(t_0)}\sum_{u\in L_n}\int_0^TK_{h,t_0}(s)\tilde{H}_{n,u}(s,\theta)ds\Bigg|>C\delta_n\Bigg) \nonumber \\
\leq&\#T_{n,k_0}\cdot\sup_{(t_0,\theta)\in T_{n,k_0}}\IP\Bigg(\Bigg|\frac{1}{r_n\bar{p}_n(t_0)}\sum_{u\in L_n}\int_0^TK_{h,t_0}(s)\tilde{H}_{n,u}(s,\theta)ds\Bigg|>C\delta_n\Bigg) \nonumber \\
=&\#T_{n,k_0}\cdot\sup_{(t_0,\theta)\in T_{n,k_0}}\IP\Bigg(\Bigg|\sum_{u\in L_n}\int_0^T\frac{K_{h,t_0}(s)}{r_n\bar{p}_n(t_0)}\left(H_{n,u}(s,\theta)-\IE(H_{n,u}(\theta,s))\right)ds\Bigg|>C\delta_n\Bigg) \nonumber
\end{align}
The above probability can be shown by Lemma \ref{lem:sufficient_mixing} to decrease faster than any power when $C>0$ is chosen large enough. As a matter of fact we argued already that Lemma \ref{lem:sufficient_mixing} can be applied in this situation when we discussed \eqref{eq:smcs2}. We just have to replace $\theta_0$ by an arbitrary $\theta\in\Theta$. But this does not change the argument and all necessary assumptions were made in Assumption (D3).
\end{proof}

\subsection{Primitive Assumptions for (D1)}
\label{subsec:fassumpD1}
In this section we provide additional assumptions under which the conditions \eqref{eq:cond1}-\eqref{eq:cond5} of Theorem \ref{thm:easy_non-pred} can be proven. These additional assumptions are in the spirit of the assumptions (H2) and (AD). Therefore, we present them here as extensions of them.

\textbf{(H2) Hub size restriction} \\
\emph{There is a deterministic sequence $H_n$ such that $H_n\geq K_m^{L_n}$ with
\begin{equation}
\sup_{n\in \IN}H_n^4\frac{\log r_np_n}{r_np_n}<+\infty. \label{eq:HubS6}
\end{equation}}

The extension of (H2) requires a deterministic bound to the maximal hub-size. This seems restrictive but for large networks it is believable that even highly connected actors are only connected to a fraction of the whole network. Keep also in mind, that connected here really means active influence. So if we were to allow that there is one single pair who influences the entire network, then statistical inference is (at least intuitively) impossible.

Denote
\begin{align*}
\Omega_n(t):=&\frac{1}{r_np_n(t)}\sum_{(i,j)\in L_n}\left(\int_{t-2h}^{t-}dN_{n,ij}(s)+\int_{t-2h}^{t-}\lambda_{n,ij}(s)ds\right)\mathcal{C}_{n,ij}, \\
\omega_n(t):=&\frac{1}{r_np_n(t)}\sum_{(i,j)\in L_n}C_{n,ij}(s)\mathcal{C}_{n,ij}.
\end{align*}

\textbf{(AD) Additional Dependence} \\
\emph{Let $k\in\IN$ be arbitrary and consider the following choices for the pair $(\epsilon_n,c_n)$
\begin{align*}
\epsilon_n=n^{-k}, \quad\quad&c_n=\sqrt{h}H_n^{-1} \textrm{ and}\\
\epsilon_n=2h+n^{-k}, \quad\quad&c_n=p_n\sqrt{h}n^{k\cdot\alpha_c}H_n^{-1}H_{n,c}^{-1}
\end{align*}
where $\alpha_c$ and $H_{n,c}$ are the Hoelder exponent and constant of $p_n(t)$, respectively. For any given $k_0\in\IN$ we can choose $k\in\IN$such that for both choices above it holds that
\begin{equation}
\IP\left(\underset{|t-s|<\epsilon_n}{\sup_{t,s\in[0,T]}}\frac{1}{r_np_n}\sum_{(i,j)\in L_n}N_{n,ij}(t)-N_{n,ij}(s)>c_n\right)=o\left(n^{-k_0}\right). \label{eq:AD1}
\end{equation}
Moreover, there is $\kappa>0$ such that for all $\xi_1,\xi_2>1$, $(i,j)\in L_n$ it holds that
\begin{align}
&\sup_{t\in[0,T]}\IE\left(\mathcal{C}_{n,ij}^2\Omega_n(t)^2C_{n,ij}(t)\Big|\Omega_n(t)>\sqrt{h}\xi_1\right)=O(n^{\kappa}), \label{eq:AD3} \\
&\sup_{t\in[0,T]}\IE\Bigg(\int_0^t\frac{\mathcal{C}_{n,ij}\omega_{n}(t)\Omega_n(t)\Omega_n(s)}{p_n(s)}d|M_{n,ij}|(s)\Bigg|\Omega_n(t)>\sqrt{h}\xi_1\textrm{ or }\omega_n(t)>\xi_2\Bigg)=O(n^{\kappa}), \label{eq:AD2} \\
&\sup_{s\in[0,T]}\IE\left(\int_0^s\frac{\mathcal{C}_{n,ij}}{p_n(t)}\Omega_n(t)d|M_{n,ij}|(t)\Big|\sup_{t\in[0,T]}\Omega_n(t)>\xi_1\right)=O(n^{\kappa}), \label{eq:AD2a} \\
&\IE\Bigg(\underset{t\in[0,T]}{\underset{(i,j)\neq (k,l)}{\sup_{(i,j),(k,l)\in L_n}}}\left(\frac{A_n(t)}{r_np_n}\int_t^{t+2h}\int_{\xi-2h}^{\xi-}\frac{A_n(\xi)}{r_np_n}d|M_{n,kl}|(\rho)d|M_{n,ij}|(\xi)\right)^2 \nonumber \\
&\quad\quad\quad\quad\Bigg|\sup_{t\in[0,T]}\frac{A_n(t)}{r_np_n}>\xi_1\Bigg)=O\left(n^{\kappa}\right), \label{eq:AD6} \\
&\sup_{t\in[0,T]}\frac{1}{r_n^2}\underset{(i,j)\neq (k,l)}{\sum_{(i,j),(k,l)\in L_n}}\int_0^T\int_{t-2h}^t\IE\left[\left(\frac{A_n(t)}{r_np_n}\right)^2\cdot\frac{C_{n,ij}(t)C_{n,kl}(r)}{p_n^2}\Bigg|\frac{A_n(t)}{r_np_n}>\xi_1\right]drdt=O(n^{\kappa}), \label{eq:AD7} \\
&\IE\left(\sup_{t\in[0,T]}\underset{(i,j)\neq (k,l)}{\sup_{(i,j),(k,l)\in L_n}}\left(\int_t^{t+2h}\int_{s-2h}^{s-}d|M_{n,kl}|(\rho)d|M_{n,ij}|(s)\right)^2\right)=o(1). \label{eq:AD4}
\end{align}}

Assumption (AD, \ref{eq:AD1}) is essentially stating that not too many $N_{n,ij}$ jump at the same time.  Note that (AD, \ref{eq:AD1}) could be proven (similar as for (AD, \ref{eq:AD5})) by using other technical assumptions. In order to prove these assumptions we need an exponential inequality for averages of counting processes. Such an inequality can be shown by employing $\beta$-mixing as in the proof of Lemma \ref{lem:network_exp}. However, instead of using the Bernstein inequality  (see e.g. Proposition \ref{prop:bernstein}) we need a tail bound valid for independent sums of counting processes with bounded intensity functions. For our purposes it is sufficient to use a tail bound induced by using Chebyshev's Inequality in its exponential form. For the remaining assumptions we emphasize that they are requiring a polynomial growth of certain moments. Since the power of the polynomial can be chosen arbitrarily, we regard these assumptions as easy to believe.

\begin{lemma}
\label{lem:pred_cond_hold}
Suppose that the the assumptions (H2) and (AD) hold in their extended form as above. Then, the conditions \eqref{eq:cond1}-\eqref{eq:cond5} of Theorem \ref{thm:easy_non-pred} hold for $\tilde{\phi}_{n,u_1u_2}^I$ as in \eqref{eq:phi_tilde} with $\delta_n=h$.
\end{lemma}
\begin{proof}
We begin with \eqref{eq:cond1}. Note that by Assumption (H1), $H_{UB}^{u}$ is measurable with respect to $\mathcal{F}_0^n$ for all $u\in L_n$ and that by (A2, \ref{ass:boundedness}) $\lambda_{n,u}(t)$ is bounded by $C_{n,u}(t)\Lambda$. Denote $d|M_{n,u}|(t)=dN_{n,u}(t)+\lambda_{n,u}(t)dt$. We get by applying the estimate \eqref{eq:h2} from Lemma \ref{lem:help2} for any $\epsilon>0$ and any $F>0$

\begin{align}
&\IP\left(\left|\frac{1}{r_n}\underset{u_1\neq u_2}{\sum_{u_1,u_2\in L_n}}\int_0^T\int_{t-2h}^{t-}\phi_{n,u_1u_2}(t,r)-\tilde{\phi}_{n,u_1u_2}^{u_1u_2}(t,r)dM_{n,u_2}(r)dM_{n,u_1}(t)\right|>\epsilon\right) \nonumber \\
\leq&\IP\left(\underset{u_1\neq u_2}{\sum_{u_1,u_2\in L_n}}\int_0^T\int_{t-2h}^{t-}\frac{4C^*}{r_n^2p_n(t)^2}\left(F+K_m^{L_n}H_{UB}^{u_1u_2}\right)d|M_{n,u_2}|(r)d|M_{n,u_1}|(t)>\epsilon\right). \label{eq:byProduct}
\end{align}
We keep this in mind and make a similar estimation for \eqref{eq:cond2a}: We use \eqref{eq:h2} in Lemma \ref{lem:help2} in order to obtain (the step from \eqref{eq:b0} to \eqref{eq:b1} is unmotivated at this point, but will be useful later)
\begin{align}
&\Bigg|\frac{1}{r_n^2}\underset{u_1\neq u_1', u_2\neq u_2'}{\sum_{u_1,u_2,u_1'u_2'\in L_n}}\int_0^T\int_{t-2h}^{t-}\tilde{\phi}_{n,u_1u_2}^{u_1u_2}-\tilde{\phi}_{n,u_1u_2}^{u_1u_2u_1'u_2'}dM_{n,u_2}(r)dM_{n,u_1}(t) \nonumber \\
&\quad\quad\quad\times\int_0^T\int_{t-2h}^{t-}\tilde{\phi}_{n,u'_1u'_2}^{u'_1u'_2}(t,r)-\tilde{\phi}_{n,u'_1u'_2}^{u_1u_2u'_2u'_2}(t,r)dM_{n,u'_2}(r)dM_{n,u_1'}(t)\Bigg| \nonumber \\
\leq&\frac{1}{r_n^2}\underset{u_1\neq u_1', u_2\neq u_2'}{\sum_{u_1,u_2,u_1'u_2'\in L_n}}\int_0^T\frac{4C^*}{p_n(t)^2}\left(F+K_m^{L_n}H_{UB}^{u_1'u_2'}\right)\int_{t-2h}^{t-}d|M_{n,u_2}|(r)d|M_{n,u_1}|(t) \nonumber \\
&\quad\quad\quad\times\frac{1}{r_n^2}\int_0^T\frac{4C^*}{p_n(t)^2}\left(F+K_m^{L_n}H_{UB}^{u_1u_2}\right)\int_{t-2h}^{t-}d|M_{n,u_2'}|(r)d|M_{n,u_1'}|(t) \nonumber \\
\leq&\left(\frac{1}{r_n^2}\sum_{u_1,u_2\in L_n}\int_0^T\frac{4C^*}{p_n(t)^2}\int_{t-2h}^{t-}\left(F+H_{UB}^{u_1u_2}K_m^{L_n}\right)d|M_{n,u_2}|(r)d|M_{n,u_1}|(t)\right)^2 \label{eq:b0} \\
\leq&\left(\frac{1}{r_n^2}\sum_{u_1,u_2\in L_n}\int_0^T\frac{4C^*}{p_n(t)^2}\int_{t-2h}^{t-}\left(F+H_{UB}^{u_1}K_m^{L_n}\right)\left(F+H_{UB}^{u_1u_2}K_m^{L_n}\right)d|M_{n,u_2}|(r)d|M_{n,u_1}|(t)\right)^2. \label{eq:b1}
\end{align}
A simple application of Markov's Inequality is hence showing that $\IE(\eqref{eq:b1})\to0$ implies that $\eqref{eq:byProduct}\to0$. To show the former we make the following definitions. Denote $\mathcal{C}_{n,u_1}:=F+3H_{UB}^{u_1}\left(K_m^{L_n}\right)^2$. Using that $H_{UB}^{u_1u_2}\leq H_{UB}^{u_1}+H_{UB}^{u_2}$, we can calculate that
$$\left(F+H_{UB}^{u_1}K_m^{L_n}\right)\left(F+H_{UB}^{u_1u_2}K_m^{L_n}\right)\leq \mathcal{C}_{n,u_1}\mathcal{C}_{n,u_2}.$$
Define $\Omega_n(t):=\frac{1}{r_np_n(t)}\sum_{u_2\in L_n}\int_{t-2h}^{t-}d|M_{n,u_2}|(r)\mathcal{C}_{n,u_2}$ and continue with the estimation:
\begin{align*}
\eqref{eq:b1}\leq&\left(\frac{1}{r_n}\sum_{u_1\in L_n}\int_0^T\frac{4C^*}{p_n(t)}\left(\frac{1}{r_np_n(t)}\sum_{u_2\in L_n}\int_{t-2h}^{t-}d|M_{n,u_2}|(r)\mathcal{C}_{n,u_2}\right)d|M_{n,u_1}|(t)\mathcal{C}_{n,u_1}\right)^2 \\
=&\left(\sum_{u_1\in L_n}\int_0^T\frac{4C^*}{r_np_n(t)}\Omega_n(t)d|M_{n,u_1}|(t)\mathcal{C}_{n,u_1}\right)^2.
\end{align*}
Define $\omega_n(s):=1/(r_np_n(s))\sum_{v\in L_n}\mathcal{C}_{n,v}C_{n,v}(s)$. By using Itô's Lemma (cf. Theorem \ref{thm:Ito}) for the semi-martingale $|M_{n,u}|(t):=N_{n,u}(t)+\int_0^t\lambda_{n,u}(s)ds$ in the first step we obtain for fixed constants $\xi_1,\xi_2>0$
\begin{align}
&\IE(\eqref{eq:b1}) \nonumber \\
\leq&4\sum_{u,v\in L_n}\int_0^T\IE\left(\int_0^s\frac{16(C^*)^2\mathcal{C}_{n,u}\mathcal{C}_{n,v}\Omega_n(t)\Omega_n(s)}{r_n^2p_n(t)p_n(s)}d|M_{n,u}|(t)\lambda_{n,v}(s)\right)ds \nonumber \\
&\quad+\sum_{u\in L_n}\int_0^T\IE\left(\frac{16(C^*)^2\mathcal{C}_{n,u}^2\Omega_n(s)^2}{r_n^2p_n(s)^2}\lambda_{n,u}(s)\right)ds \nonumber \\
\leq&64(C^*)^2\Lambda\int_0^T\IE\left(\int_0^s\frac{\mathcal{C}_{n,u}}{p_n(t)}\omega_{n}(s)\Omega_n(t)\Omega_n(s)d|M_{n,u}|(t)\right)ds \nonumber \\
&\quad+\frac{16(C^*)^2\Lambda}{r_np_n}\int_0^T\IE\left(\mathcal{C}_{n,u}^2\Omega_n(s)^2\Big| C_{n,u}(s)=1\right)ds \nonumber \\
\leq&64(C^*)^2\Lambda\xi_1\xi_2\sqrt{h}\int_0^T\IE\left(\int_0^s\frac{\mathcal{C}_{n,u}}{p_n(t)}\Omega_n(t)d|M_{n,u}|(t)\right)ds \label{b11} \\
&+64(C^*)^2\Lambda\int_0^T\IE\left(\int_0^s\frac{\mathcal{C}_{n,u}}{p_n(t)}\omega_n(s)\Omega_n(t)\Omega_n(s)d|M_{n,u}|(t)\Bigg|\Omega_n(s)>\sqrt{h}\xi_1\textrm{ or }\omega_n(s)>\xi_2\right) \nonumber \\
&\quad\quad\quad\times\left(\IP\left(\Omega_n(s)>\sqrt{h}\xi_1\right)+\IP\left(\omega_n(s)>\xi_2\right)\right)ds \label{b12} \\
&+\frac{16(C^*)^2\Lambda h\xi_1^2}{r_np_n}\int_0^T\IE\left(\mathcal{C}_{n,u}^2\Big| C_{n,u}(s)=1\right)ds \label{b13} \\
&+\frac{16(C^*)^2\Lambda}{r_np_n^2}\int_0^T\IE\left(\mathcal{C}_{n,u}^2\Omega_n(s)^2C_{n,u}(s)\Big|\Omega_n(s)>\sqrt{h}\xi_1\right)\IP\left(\Omega_n(s)>\sqrt{h}\xi_1\right)ds. \label{b14}
\end{align}
Line \eqref{b13} converges to zero by \eqref{eq:HubS5} and $h\to0$. Line \eqref{b14} converges to zero by Lemma \ref{lem:OmegasGood} and Assumption (AD, \ref{eq:AD3}). For line \eqref{b12} we use Lemma \ref{lem:OmegasGood} together with Assumption (AD, \ref{eq:AD2}) to show that it converges to zero for the correct choices of $\xi_1$ and $\xi_2$. For line \eqref{b11} we note that for $\xi>0$
\begin{align*}
&\IE\left(\int_0^s\frac{\mathcal{C}_{n,u}}{p_n(t)}\Omega_n(t)d|M_{n,u}|(t)\right) \\
\leq&\xi\IE\left(\int_0^s\frac{\mathcal{C}_{n,u}}{p_n(t)}d|M_{n,u}|(t)\right) \\
&\quad+\IE\left(\int_0^s\frac{\mathcal{C}_{n,u}}{p_n(t)}\Omega_n(t)d|M_{n,u}|(t)\Big|\sup_{t\in[0,T]}\Omega_n(t)>\xi\right)\IP\left(\sup_{t\in[0,T]}\Omega_n(t)>\xi\right)
\end{align*}
is bounded by Assumptions (H2, \ref{eq:HubS7}), (AD, \ref{eq:AD2a}) and Lemma \ref{lem:OmegasGood}. Thus we have shown that \eqref{eq:cond1} and \eqref{eq:cond2a} hold.

For showing condition \eqref{eq:cond3} we define the random number of active edges as
$$A_n(t):=\sum_{u\in L_n}\sup_{r\in[-2,2]}C_{n,u}(t+rh)$$
and use this to estimate for all $u_1,u_2\in L_n$, all $I\subseteq L_n$ and all $r,t\in[0,T]$
\begin{equation}
\label{eq:phi_bound}
\left|\tilde{\phi}_{n,u_1u_2}^{I}(t,r)\right|\leq\frac{2C^*}{r_np_n^2}A_n(t).
\end{equation}
In addition to the above estimate, we also observe that
\begin{align*}
&\sum_{u_1',u_2'\in L_n}\sup_{\rho\in[t-2h,t+2h]}C_{n,u_2'}(\rho)\sup_{\xi\in[t,t+2h]}C_{n,u_1'}(\xi)\Ind(\neg u_1',u_2'\in F_{u_1}(t-2\delta_n)) \\
=&\sum_{u_1'\in L_n}\sup_{\xi\in[t,t+2h]}C_{n,u_1'}(\xi)\Ind(u_1'\in F_{u_1}(t-2\delta_n))\times\sum_{u_2'\in L_n}\sup_{\rho\in[t-2h,t+2h]}C_{n,u_2'}(\rho)\Ind(u_2'\notin F_{u_1}(t-2\delta_n)) \\
&+\sum_{u_1'\in L_n}\sup_{\xi\in[t,t+2h]}C_{n,u_1'}(\xi)\Ind(u_1'\notin F_{u_1}(t-2\delta_n))\times\sum_{u_2'\in L_n}\sup_{\rho\in[t-2h,t+2h]}C_{n,u_2'}(\rho)\Ind(u_2'\in F_{u_1}(t-2\delta_n)) \\
&+\sum_{u_1'\in L_n}\sup_{\xi\in[t,t+2h]}C_{n,u_1'}(\xi)\Ind(u_1'\notin F_{u_1}(t-2\delta_n))\times\sum_{u_2'\in L_n}\sup_{\rho\in[t-2h,t+2h]}C_{n,u_2'}(\rho)\Ind(u_2'\notin F_{u_1}(t-2\delta_n)) \\
\leq&3A_n(t)K_m^{u_1}(t).
\end{align*}
Using these two estimates together with the estimate in \eqref{eq:h2}, we obtain
\begin{align}
&|\eqref{eq:cond3}| \nonumber \\
\leq&\frac{2}{r_n^2}\underset{u_1\neq u_2,u_1'\neq u_2'}{\sum_{u_1,u_2,u_1',u_2'\in L_n}}\IE\Bigg(\int_0^T\int_{t-2h}^t\frac{4C^*}{r_np_n(t)^2}\left(F+K_m^{L_n}H_{UB}^{u_1u_2}\right)d|M_{n,u_2}|(r) \nonumber \\
&\times\int_t^{t+2h}\int_{\xi-2h}^{\xi-}\frac{4C^*}{r_np_n^2}A_n(\xi)d|M_{n,u_2'}|(\rho)d|M_{n,u_1'}|(\xi)\Ind(\neg u_1',u_2'\in F_{u_1}(t-2\delta_n))d|M_{n,u_1}|(t)\Bigg) \nonumber \\
\leq&2\IE\Bigg(\frac{1}{r_n^2}\underset{u_1\neq u_2}{\sum_{u_1,u_2\in L_n}}\int_0^T\int_{t-2h}^t\frac{4C^*}{p_n(t)^2}\left(F+K_m^{L_n}H_{UB}^{u_1u_2}\right)d|M_{n,u_2}|(r) \nonumber \\
&\times\underset{u'_1\neq u'_2}{\sum_{u'_1,u'_2\in L_n}}\int_t^{t+2h}\int_{\xi-2h}^{\xi-}\frac{4C^*}{r_n^2p_n^2}A_n(\xi)d|M_{n,u_2'}|(\rho)d|M_{n,u_1'}|(\xi)\Ind(\neg u_1',u_2'\in F_{u_1}(t-2\delta_n))d|M_{n,u_1}|(t)\Bigg) \nonumber \\
\leq&2\IE\Bigg(\frac{1}{r_n^2}\underset{u_1\neq u_2}{\sum_{u_1,u_2\in L_n}}\int_0^T\int_{t-2h}^t\frac{4C^*}{p_n(t)^2}\left(F+K_m^{L_n}H_{UB}^{u_1u_2}\right)d|M_{n,u_2}|(r)d|M_{n,u_1}| \nonumber \\
&\times\underset{u'_1\neq u'_2}{\sup_{u'_1,u'_2\in L_n}}\sup_{t\in[0,T]}\int_t^{t+2h}\int_{\xi-2h}^{\xi-}\frac{4C^*}{r_n^2p_n^2}A_n(\xi)d|M_{n,u_2'}|(\rho)d|M_{n,u_1'}|(\xi)3A_n(t)K_m^{u_1}\Bigg) \nonumber \\
\leq&2\sqrt{\IE\left(\left(\frac{1}{r_n^2}\underset{u_1\neq u_2}{\sum_{u_1,u_2\in L_n}}\int_0^T\int_{t-2h}^t\frac{4C^*}{p_n(t)^2}K_m^{u_1}\left(F+K_m^{L_n}H_{UB}^{u_1u_2}\right)d|M_{n,u_2}|(r)d|M_{n,u_1}|\right)^2\right)} \nonumber \\
&\times\sqrt{\IE\left(\underset{t\in[0,T]}{\underset{u'_1\neq u'_2}{\sup_{u'_1,u'_2\in L_n}}}\left(\int_t^{t+2h}\int_{\xi-2h}^{\xi-}\frac{4C^*}{r_n^2p_n^2}A_n(\xi)d|M_{n,u_2'}|(\rho)d|M_{n,u_1'}|(\xi)3A_n(t)\right)^2\right)} \nonumber
\end{align}
Note that $K_m^{u_1}\leq F+K_m^{L_n}H_{UB}^{u_1}$. Hence, we just saw in the proof of \eqref{eq:cond2a} (cf. \eqref{eq:b1}) that the first part above converges to zero. For the second part we employ similar techniques, i.e., we condition on the event that $A_n(t)/r_np_n$ is bounded. So for any $\alpha>0$ the second part above (without the square root) can be bounded by
\begin{align*}
&\IE\left(\underset{t\in[0,T]}{\underset{u'_1\neq u'_2}{\sup_{u'_1,u'_2\in L_n}}}\left(12C^*\alpha^2\int_t^{t+2h}\int_{\xi-2h}^{\xi-}d|M_{n,u_2'}|(\rho)d|M_{n,u_1'}|(\xi)\right)^2\right) \\
+&\IE\Bigg(\underset{t\in[0,T]}{\underset{u'_1\neq u'_2}{\sup_{u'_1,u'_2\in L_n}}}\left(12C^*\frac{A_n(t)}{r_np_n}\int_t^{t+2h}\int_{\xi-2h}^{\xi-}\frac{A_n(\xi)}{r_np_n}d|M_{n,u_2'}|(\rho)d|M_{n,u_1'}|(\xi)\right)^2 \\
&\quad\quad\quad\quad\Bigg|\sup_{t\in[0,T]}\frac{A_n(t)}{r_np_n}>\alpha\Bigg)\IP\left(\sup_{t\in[0,T]}\frac{A_n(t)}{r_np_n}>\alpha\right).
\end{align*}
By the Assumptions (AD, \ref{eq:AD5}, \ref{eq:AD4}, \ref{eq:AD6}) the expression above converges to zero and we have shown \eqref{eq:cond3}.

The indicator function in \eqref{eq:cond4} is not significantly shortening the sum and hence we just ignore it. Moreover, we use the bound from \eqref{eq:phi_bound} to obtain for any $\alpha>0$
\begin{align*}
\eqref{eq:cond4}\leq&\frac{\Lambda^2}{r_n^2}\underset{u_1\neq u_2}{\sum_{u_1,u_2\in L_n}}\int_0^T\int_{t-2h}^t\IE\left[\left(\frac{2C^*A_n(t)}{r_np_n}\right)^2\cdot\frac{C_{n,u_1}(t)C_{n,u_2}(r)}{p_n^2}\right]drdt \\
\leq&\frac{4\Lambda^2(C^*)^2\alpha^2}{r_n^2}\underset{u_1\neq u_2}{\sum_{u_1,u_2\in L_n}}\int_0^T\int_{t-2h}^t\frac{\IP(C_{n,u_1}(t)C_{n,u_2}(r)=1)}{p_n^2}drdt \\
&+\frac{\Lambda^2}{r_n^2}\underset{u_1\neq u_2}{\sum_{u_1,u_2\in L_n}}\int_0^T\int_{t-2h}^t\sup_{t\in[0,T]}\IE\left[\left(\frac{2C^*A_n(t)}{r_np_n}\right)^2\cdot\frac{C_{n,u_1}(t)C_{n,u_2}(r)}{p_n^2}\Bigg|\frac{A_n(t)}{r_np_n}>\alpha\right]drdt \\
&\quad\quad\times\sup_{t\in[0,T]}\IP\left(\frac{A_n(t)}{r_np_n}>\alpha\right).
\end{align*}
The expression above converges to zero by Assumptions (D2, \ref{CA:AU1}), (AD, \ref{eq:AD5}, \ref{eq:AD7}). Hence, we have also that \eqref{eq:cond4} holds.

For \eqref{eq:cond5} we finally use that for every fixed choice of $u_2,u_2'\in L_n$ we get
\begin{align*}
\sum_{u_1\in L_n}C_{n,u_1}(t)\Ind(\neg u_2,u_2'\in F_{u_1}(t-2h))\leq& K_m^{u_2}(t+2h)+K_m^{u_2'}(t+2h) \\
\leq&\left(F+K_m^{L_n}H_{UB}^{u_2}\right)\left(F+K_m^{L_n}H_{UB}^{u_2'}\right)
\end{align*}
Thus, we obtain together with \eqref{eq:phi_bound}
\begin{align*}
|\eqref{eq:cond5}|\leq&\frac{\Lambda}{r_n^2}\underset{u_2\neq u_2'}{\sum_{u_2,u_2'\in L_n}}\int_0^T\IE\Bigg[\int_{t-2h}^{t-}\frac{2C^*}{r_np_n^2}A_n(t)d|M_{n,u_2}|(r) \\
&\times\int_{t-2h}^{t-}\frac{2C^*}{r_np_n^2}A_n(t)d|M_{n,u_2'}|(r')\sum_{u_1\in L_n}C_{n,u_1}(t)\Ind(\neg u_2,u_2'\in F_{u_1}(t-2h)\Bigg]dt \\
\leq&\Lambda\int_0^T\IE\left[\left(\sum_{u\in L_n}\int_{t-2h}^{t-}\frac{\left(F+K_m^{L_n}H_{UB}^{u}\right)}{r_np_n}d|M_{n,u}|(r)\right)^2\left(\frac{2C^*A_n(t)}{r_np_n}\right)^2\right]dt.
\end{align*}
The above converges to zero by conditioning on $A_n(t)/r_np_n>\alpha$ and the Assumptions (H2, \ref{eq:HubS7}) and (AD, \ref{eq:AD5}, \ref{eq:AD8}). Hence, we have shown that \eqref{eq:cond5} holds as well. 
\end{proof}

\subsection{Further Proofs}
\label{subsec:fproofs}
\begin{proof}[Proof of Lemma \ref{lem:martingale}]
Note that $\mathcal{F}_t^{n,J,m}\supseteq\mathcal{F}_t^n$. Hence, $(N_{n,ij}(t))_{(i,j)\in J}$ is adapted with respect to $\mathcal{F}_t^{n,J,m}$ and $(\lambda_{n,ij}(t))_{(i,j)\in J}$ is predictable with respect to $\mathcal{F}_t^{n,J,m}$. So $N_{n,ij}(t)$ is a counting process for all $(i,j)\in J$. In order to check that $\lambda_{n,ij}(t)$ is the intensity function, we need to check the martingale property: Let $t>0$ and $t'\in[t,t+6\delta_n]$, then by definition and assumption (recall $M_{n,ij}(t)=N_{n,ij}(t)-\int_0^t\lambda_{n,ij}(s)ds$)
$$\IE(M_{n,ij}(t')|\mathcal{F}_t^{n,J,m})=\IE(M_{n,ij}(t')|\mathcal{F}_t^n)=M_{n,ij}(t).$$
\end{proof}

\begin{proof}[Proof of Proposition \ref{prop:single_non-pred}]
The proof is almost exactly along the lines of \citet{MN07} but it is not identical and we give it here for completeness. We see at first that
\begin{align}
&\IE\left(\left(\sum_{(i,j)\in L_n}\int_0^T\phi_{n,ij}(t)dM_{n,ij}(t)\right)^2\right) \nonumber \\
=&\sum_{(i,j),(k,l)\in L_n}\IE\left(\int_0^T\int_0^T\tilde{\phi}_{n,ij}^{ij,kl}(t)\tilde{\phi}_{n,kl}^{ij,kl}(r)dM_{n,ij}(t)dM_{n,kl}(r)\right) \label{eq:MN1} \\
&+2\sum_{(i,j),(k,l)\in L_n}\IE\left(\int_0^T\int_0^T\tilde{\phi}^{ij,kl}_{n,ij}(t)\left(\phi_{n,kl}(r)-\tilde{\phi}_{n,kl}^{ij,kl}(r)\right)dM_{n,ij}(t)dM_{n,kl}(r)\right) \nonumber \\
&+\sum_{(i,j),(k,l)\in L_n}\IE\left(\int_0^T\int_0^T\left(\phi_{n,ij}(t)-\tilde{\phi}_{n,ij}^{ij,kl}(t)\right)\cdot\left(\phi_{n,kl}(r)-\tilde{\phi}_{n,kl}^{ij,kl}(r)\right)dM_{n,ij}(t)dM_{n,kl}(r)\right). \nonumber
\end{align}
We use now that $\tilde{\phi}_{n,ij}^{ij,kl}$ and $\tilde{\phi}_{n,kl}^{ij,kl}$ are both predictable with respect to $\mathcal{F}_t^{n,\{(i,j),(k,l)\},m}$ and that $M_{n,ij}$ and $M_{n,kl}$ are uncorrelated martingales with respect to the same filtration (cf. Lemma \ref{lem:martingale}). Hence, we obtain
$$\eqref{eq:MN1}=\sum_{(i,j)\in L_n}\int_0^T\IE\left(\tilde{\phi}_{n,ij}^{ij}(t)^2C_{n,ij}(t)\lambda_{n,ij}(\theta_0,t)\right)dt$$
and the statement follows.
\end{proof}

The following result provides an exponential inequality for martingales. Note that the result is different from \citet{vdg95} because the asymptotics in the motivation are different and in the following result we can avoid appearance of the higher order variation process in the probability.
\begin{lemma}
\label{lem:exponential_inequality}
Suppose that (A3, \ref{ass:boundedness}) and (A4, \ref{ass:kernel_hoelder}) hold. Recall the following definitions from (A3) and (A4): $\Lambda$ is the bound on the intensity function, $K$ the bound on the kernel and $\hat{K}$ the bound on the covariates. Let $A>0$ be so large such that
\begin{equation}
\label{eq:chioceA}
A\geq \max\left\{\sqrt{\hat{K}},\hat{K},\frac{1}{K},\sqrt{2^{\frac{3}{2}}\sqrt{\Lambda}\frac{\hat{K}}{K}}\right\},\,\frac{1}{A}\cdot\exp\left(\frac{\sqrt{2}}{A\sqrt{\Lambda}}\right)\leq1.
\end{equation}
Suppose we have a $\Delta_n$-partition measurable with respect to $\mathcal{F}_{t_0-h}^n$ (we write $I_{n,u}^{k,m}:=I_{n,u}^{k,k,t_0-h}$ for ease of notation)
$$\sum_k\sum_mI_{n,u}^{k,m}=\sup_{r\in[t_0-h,t_0+h]}C_{n,u}(r)$$
for some $t_0\in[0,T]$ and all $u\in L_n$. Define furthermore for arbitrary $c_3>0$ and  $u\in L_n$,
\begin{align*}
E_{k,m}^{n,t_0}:=&r_n\int_0^TK_{h,t_0}\left(s\right)\IE\left[I_{n,u}^{k,m}C_{n,u}(s)\right]ds,&E_k^{n,t_0}:=\sqrt{\frac{r_n\bar{p}_n(t_0)}{\log r_n\bar{p}_n(t_0)}}\cdot c_3, \\
S_{k,m}:=&\sum_{u\in L_n}\int_0^TK_{h,t_0}\left(s\right)\IE\left(I_{n,u}^{k,m}C_{n,u}(s)\Big|\mathcal{F}_{t_0-h}^n\right)ds,&S_k:=\max_{m=1,...,r_n}\sum_{u\in L_n}I_{n,u}^{k,m}, \\
\Gamma_n^{t_0}:=&\Ind\left(\forall k:\frac{S_k^2\cdot\log r_n\bar{p}_n(t_0)}{r_n\bar{p}_n(t_0)}\leq c_3^2, S_k\sqrt{h}\geq 1\right), \\
\sigma^2:=&\frac{1}{h}\Lambda A^2K\hat{K},\quad c_1:=\sqrt{\frac{\Lambda}{2h}}K\hat{K}A.
\end{align*}
Assume that there is $c_2>0$ such that for all $k\in\{1,...,\Kappa\}$
$$\frac{1}{r_n\bar{p}_n(t_0)}\sum_{m=1}^{r_n}E_{k,m}^{n,t_0}\geq c_2.$$

Then it holds that
\begin{align*}
&\IP\left(\frac{1}{r_n\bar{p}_n(t_0)}\left\|\sum_{u\in L_n}\int_0^TK_{h,t_0}(s)X_{n,u}(s)dM_{n,u}(s)\right\|\geq xqc^{**}\sqrt{\frac{\log r_n\bar{p}_n(t_0)}{r_n\bar{p}_n(t_0)}}\right) \\
\leq& 2q\Kappa\left[r_n\bar{p}_n(t_0)\right]^{-\frac{c_2x^2}{2(\sigma^2+c_1c_3x)}}+2q\beta_{t_0}(\Delta_n)\cdot \Kappa r_n+2q\IP(\Gamma_n^{t_0}=0),
\end{align*}
where $q$ is the dimension of the covariate and $c^{**}$ is the constant for which $\|y\|\leq c^{**}\|y\|_1$ for all $y\in\IR^q$ and $\|.\|$ and $\|.\|_1$ are the Euclidean and $1$-Norm respectively. The process $X_{n,u}$ can be replaced by any other predictable process which is bounded by $\hat{K}$.
\end{lemma}
\begin{proof}
We remark firstly that it is sufficient to consider univariate covariates, because (denote by $X_{n,u}^r$ the $r$-th entry of $X_{n,u}$ for $r=1,...,q$)
\begin{align*}
&\IP\left(\frac{1}{r_n\bar{p}_n(t_0)}\left\|\sum_{u\in L_n}\int_0^TK_{h,t_0}(s)X_{n,u}(s)dM_{n,u}(s)\right\|\geq xqc^{**}\sqrt{\frac{\log r_n\bar{p}_n(t_0)}{r_n\bar{p}_n(t_0)}}\right) \\
\leq&\sum_{r=1}^q\IP\left(\frac{1}{r_n\bar{p}_n(t_0)}\left|\sum_{u\in L_n}\int_0^TK_{h,t_0}(s)X_{n,u}^r(s)dM_{n,u}(s)\right|\geq x\sqrt{\frac{\log r_n\bar{p}_n(t_0)}{r_n\bar{p}_n(t_0)}}\right) \\
\leq&\sum_{r=1}^q\IP\left(\frac{1}{r_n\bar{p}_n(t_0)}\sum_{u\in L_n}\int_0^TK_{h,t_0}(s)X_{n,u}^r(s)dM_{n,u}(s)\geq x\sqrt{\frac{\log r_n\bar{p}_n(t_0)}{r_n\bar{p}_n(t_0)}}\right) \\
&+\IP\left(\frac{1}{r_n\bar{p}_n(t_0)}\sum_{u\in L_n}\int_0^TK_{h,t_0}(s)(-X_{n,u}^r(s))dM_{n,u}(s)\geq x\sqrt{\frac{\log r_n\bar{p}_n(t_0)}{r_n\bar{p}_n(t_0)}}\right).
\end{align*}
Since $-X_{n,u}^r$ is a covariate with the exact same properties as $X_{n,u}$ (in particular predictability with respect to $\mathcal{F}_t^n$ and boundedness by $\hat{K}$, cf. Assumption (A3, \ref{ass:boundedness}), it is sufficient to assume (for simplicity of notation) that $X_{n,u}$ is univariate and to prove that
\begin{align}
&\IP\left(\frac{1}{r_n\bar{p}_n(t_0)}\sum_{u\in L_n}\int_0^TK_{h,t_0}(s)X_{n,u}(s)dM_{n,u}(s)\geq x\sqrt{\frac{\log r_n\bar{p}_n(t_0)}{r_n\bar{p}_n(t_0)}}\right) \nonumber \\
\leq& \Kappa\left[r_n\bar{p}_n(t_0)\right]^{-\frac{c_2x^2}{2(\sigma^2+c_1c_3x)}}+\beta_{t_0}(\Delta_n)\cdot \Kappa r_n+\IP(\Gamma_n^{t_0}=0). \label{eq:aim_of_proof}
\end{align}

The main idea of the proof is to apply Lemma \ref{lem:network_exp} to the correct structured interaction network (in the sense of Definition \ref{defin:struc_interaction_net}). Define to this end
\begin{align*}
\tilde{F}_{n,u}(s):=&K_{h,t_0}(s)X_{n,u}(s)\cdot\Gamma_n^{t_0},\quad F_{n,u}^{k,m}(s):=\tilde{F}_{n,u}(s)\cdot I_{n,u}^{k,m}\textrm{ and} \\
Z_{n,u}(t):=&\int_0^t\tilde{F}_{n,u}(s)dM_{n,u}(s).
\end{align*}
Note that both, $\tilde{F}_{n,u}(s)$ and $F_{n,u}^{k,m}(s)$, are predictable processes because they are deterministically equal to zero for $s\leq t_0-h$ and the sets $t\mapsto G^{t}(k,m,\Delta_n)$ are predictable with respect to $\mathcal{F}^n_{t}$. Hence, $Z_{n,u}(t)$ is a martingale. We are going to prove that $(Z_{n,u}(T))_{u\in L_n}$ fulfils the conditions of Lemma \ref{lem:network_exp}. Condition \eqref{eq:GeneralCoverCondition} is easy to check.  Note that for $s\in[t_0-h,t_0+h]$ we have $\sup_{r\in[t_0-h,t_0+h]}C_{n,u}(r)C_{n,u}(s)=C_{n,u}(s)$ and hence
\begin{align*}
|E|_{n,t_0}=\sum_{k=1}^{\Kappa}\sum_{m=1}^{r_n}E_{k,m}^{n,t_0}=r_n\int_0^TK_{h,t_0}(s)\IE(C_{n,u}(s))ds=r_n\bar{p}_n(t_0).
\end{align*}
Thus, condition 2 of Lemma \ref{lem:network_exp} holds simply by assumption and definition of $E_k^{n,t_0}$. The main part of this proof is now to prove condition 1. Note therefore firstly that
$$\IE\left(\int_0^T\tilde{F}_{n,u}(s)dM_{n,u}(s)I_{n,u}^{k,m}\right)=\IE\left(\underbrace{\IE\left(\int_0^T\tilde{F}_{n,u}(s)dM_{n,u}(s)\Bigg|\mathcal{F}_{t_0-h}^n\right)}_{=0}I_{n,u}^{k,m}\right)=0.$$
Hence, we need to show for $t=T$
\begin{align}
&\IE\left(\left|U_{k,m}^{n,t_0}(\Delta_n)\right|^{\rho}\right) \nonumber \\
=&\IE\left(\left|\sum_{u\in L_n}\int_0^tF_{n,u}^{k,m}(s)dM_{n,u}(s)\right|^{\rho}\right) \nonumber \\
\leq&\frac{\rho!}{2}\cdot E_{k,m}^{n,t_0}\sigma^2\cdot\left(E_{k}^{n,t_0}c_1\right)^{\rho-2}. \label{eq:mombound}
\end{align}
We will show \eqref{eq:mombound} for all $t\in[0,T]$ and then it holds particularly in the case $t=T$ which is of primary interest to us. The idea of the proof is to prove a recursion inequality for the moments of stochastic integrals by applying Itô's Formula and then using induction. 
Note that $F_{n,u}^{k,m}(s)=0$ for $s\notin[t_0-h,t_0+h]$. Therefore, \eqref{eq:mombound} holds trivially for $t\leq t_0-h$ and it holds for $t\geq t_0+h$ when it holds for $t=t_0+h$. Hence, we can restrict to the case $t\in[t_0-h,t_0+h]$. 

For $\rho\geq2$ we have that the function $f_{\rho}(x):=|x|^\rho$ is twice continuously differentiable and hence also $\tilde{f}_{\rho}(x_1,...,x_m):=f_{\rho}(x_1+...+x_m)$ is twice continuously differentiable. So by the multivariate Itô Formula for semi-martingales with jumps given in Theorem \ref{thm:Ito} and the fact that with probability one no two counting processes jump at the same time, we obtain for $\rho\geq2$: Enumerate for the following computations the pairs in $L_n$, i.e., such that $L_n=\{1,...,r_n\}$. Then,
$$\left|\sum_{u\in L_n}\int_0^tF_{n,u}^{k,m}(\tau)dM_{n,u}(\tau)\right|^{\rho} \\
=\tilde{f}_{\rho}\left(\int_0^tF_{n,1}^{k,m}(\tau)dM_{n,1}(\tau),...,\int_0^tF_{n,r_n}^{k,m}(\tau)dM_{n,r_n}(\tau)\right)$$
and we can compute
\begin{eqnarray*}
&&\left|\sum_{u\in L_n}\int_0^tF_{n,u}^{k,m}(\tau)dM_{n,u}(\tau)\right|^{\rho} \\
&=&\tilde{f}_{\rho}\left(\int_0^tF_{n,1}^{k,m}(\tau)dM_{n,1}(\tau),...,\int_0^tF_{n,r_n}^{k,m}(\tau)dM_{n,r_n}(\tau)\right) \\
&=&\sum_{u\in L_n}\int_0^t\partial_u\tilde{f}_{\rho}\left(\int_0^{s-}F_{n,1}^{k,m}(\tau)dM_{n,1}(\tau),...,\int_0^{s-}F_{n,r_n}^{k,m}(\tau)dM_{n,r_n}(\tau)\right)F_{n,u}^{k,m}(s)dM_{n,u}(s) \\
&&+\frac{1}{2}\sum_{u,v\in L_n}\int_0^t\partial_{uv}\tilde{f}_{\rho}\left(\int_0^{s-}F_{n,1}^{k,m}(\tau)dM_{n,1}(\tau),...,\int_0^{s-}F_{n,r_n}^{k,m}(\tau)dM_{n,r_n}(\tau)\right) \\
&&\quad\quad\quad\quad\quad F_{n,u}^{k,m}(s)F_{n,v}^{k,m}(s)d[M_{n,u},M_{n,v}(s)](s) \\
&&+\int_0^t\tilde{f}_{\rho}\left(\int_0^{s}F_{n,1}^{k,m}(\tau)dM_{n,1}(\tau),...,\int_0^{s}F_{n,r_n}^{k,m}(\tau)dM_{n,r_n}(\tau)\right) \\
&&\quad\quad\quad-\tilde{f}_{\rho}\left(\int_0^{s-}F_{n,1}^{k,m}(\tau)dM_{n,1}(\tau),...,\int_0^{s-}F_{n,r_n}^{k,m}(\tau)dM_{n,r_n}(\tau)\right) \\
&&-\sum_{u\in L_n}\partial_u\tilde{f}_{\rho}\left(\int_0^{s-}F_{n,1}^{k,m}(\tau)dM_{n,1}(\tau),...,\int_0^{s-}F_{n,r_n}^{k,m}(\tau)dM_{n,r_n}(\tau)\right)F_{n,u}^{k,m}(s)\Delta N_{n,u}(s) \\
&&-\frac{1}{2}\sum_{u,v}^n\partial_{uv}\tilde{f}_{\rho}\left(\int_0^{s-}F_{n,1}^{k,m}(\tau)dM_{n,1}(\tau),...,\int_0^{s-}F_{n,r_n}^{k,m}(\tau)dM_{n,r_n}(\tau)\right) \\
&&\quad\quad\quad\quad \times F_{n,u}^{k,m}(s)F_{n,v}^{k,m}(s)\Delta N_{n,u}(s)\Delta N_{n,v}(s) d\left(\sum_{r\in L_n} N_{n,r}\right)(s) \\
&=&\sum_{u\in L_n}\int_0^tf'_{\rho}\left(\sum_{r\in L_n}\int_0^{s-}F_{n,r}^{k,m}(\tau)dM_{n,r}(\tau)\right)F_{n,u}^{k,m}(s)dM_{n,u}(s) \\
&&+\frac{1}{2}\sum_{u\in L_n}\int_0^tf''_{\rho}\left(\sum_{r\in L_n}\int_0^{s-}F_{n,r}^{k,m}(\tau)dM_{n,r}(\tau)\right)F_{n,u}^{k,m}(s)^2dN_{n,u}(s) \\
&&+\int_0^tf_{\rho}\left(\sum_{r\in L_n}\int_0^{s}F_{n,r}^{k,m}(\tau)dM_{n,r}(\tau)\right)-f_{\rho}\left(\sum_{r\in L_n}\int_0^{s-}F_{n,r}^{k,m}(\tau)dM_{n,r}(\tau)\right) \\
&&-\sum_{u\in L_n}f'_{\rho}\left(\sum_{r\in L_n}\int_0^{s-}F_{n,r}^{k,m}(\tau)dM_{n,r}(\tau)\right)F_{n,u}^{k,m}(s)\Delta N_{n,u}(s) \\
&&-\frac{1}{2}\sum_{u\in L_n}f''_{\rho}\left(\sum_{r\in L_n}\int_0^{s-}F_{n,r}^{k,m}(\tau)dM_{n,r}(\tau)\right)F_{n,u}^{k,m}(s)^2\Delta N_{n,u}(s)\,d\left(\sum_{r\in L_n} N_{n,r}\right)(s) \\
&=&\sum_{u\in L_n}\int_0^tf'_{\rho}\left(\sum_{r\in L_n}\int_0^{s-}F_{n,r}^{k,m}(\tau)dM_{n,r}(\tau)\right)F_{n,u}^{k,m}(s)dM_{n,u}(s) \\
&&+\int_0^tf_{\rho}\left(\sum_{r\in L_n}\int_0^{s}F_{n,r}^{k,m}(\tau)dM_{n,r}(\tau)\right)-f_{\rho}\left(\sum_{r\in L_n}\int_0^{s-}F_{n,r}^{k,m}(\tau)dM_{n,r}(\tau)\right) \\
&&-f'_{\rho}\left(\sum_{r\in L_n}\int_0^{s-}F_{n,r}^{k,m}(\tau)dM_{n,r}(\tau)\right)\sum_{u\in L_n}F_{n,u}^{k,m}(s)\Delta N_{n,u}(s)d\left(\sum_{r\in L_n} N_{n,r}\right)(s) \\
&=:& (*)
\end{eqnarray*}
Note now that
$$\sum_{r\in L_n}\int_0^sF_{n,r}^{k,m}(\tau)dM_{n,r}(\tau)-\sum_{r\in L_n}\int_0^{s-}F_{n,r}^{k,m}(\tau)dM_{n,r}(\tau)=\sum_{r\in L_n}F_{n,r}^{k,m}(s)\Delta N_{n,r}(s).$$
Hence, (*) contains a Taylor series expansion of $f_{\rho}$ around the point
$$\sum_{r\in L_n}\int_0^{s-}F_{n,r}^{k,m}(\tau)dM_{n,r}(\tau)$$
and we continue:
\begin{eqnarray*}
&&(*) \\
&=&\sum_{u\in L_n}\int_0^tf'_{\rho}\left(\sum_{r\in L_n}\int_0^{s-}F_{n,r}^{k,m}(\tau)dM_{n,r}(\tau)\right)F_{n,u}^{k,m}(s)dM_{n,u}(s) \\
&&+\int_0^t\frac{1}{2}f''_{\rho}\left(\sum_{r\in L_n}\int_0^{s-}F_{n,r}^{k,m}(\tau)dM_{n,r}(\tau)+\Delta(s)\right) \\
&&\quad\quad\quad\times\left(\sum_{r\in L_n}F_{n,r}^{k,m}(s)\Delta N_{n,r}(s)\right)^2d\left(\sum_{r\in L_n}N_{n,r}\right)(s) \\
&=:&(**),
\end{eqnarray*}
where $\Delta(s)\in\Big[0,\sum_{r\in L_n}F_{n,r}^{k,m}(s)\Delta N_{n,r}(s)\Big]$. Since only one of the counting processes jumps at a time, we obtain $|\Delta(s)|\leq K_h$ with $K_h:=\frac{1}{h}K\hat{K}$ and continue by using again that no two processes jump at the same time:
\begin{eqnarray*}
&&(**) \\
&=&\sum_{u\in L_n}\int_0^tf'_{\rho}\left(\sum_{r\in L_n}\int_0^{s-}F_{n,r}^{k,m}(\tau)dM_{n,r}(\tau)\right)F_{n,u}^{k,m}(s)dM_{n,u}(s) \\
&&+\sum_{u\in L_n}\int_0^t\frac{1}{2}f''_p\left(\sum_{r\in L_n}\int_0^{s-}F_{n,r}^{k,m}(\tau)dM_{n,r}(\tau)+\Delta(s)\right)F_{n,u}^{k,m}(s)^2dN_{n,u}(s) \\
&\leq&\sum_{u\in L_n}\int_0^tf'_{\rho}\left(\sum_{r\in L_n}\int_0^{s-}F_{n,r}^{k,m}(\tau)dM_{n,r}(\tau)\right)F_{n,u}^{k,m}(s)dM_{n,u}(s) \\
&&+\sum_{u\in L_n}\int_{t_0-h}^t\frac{1}{2}f''_{\rho}\left(\left|\sum_{r\in L_n}\int_0^{s-}F_{n,r}^{k,m}(\tau)dM_{n,r}(\tau)\right|+K_h\right) \\
&&\quad\times K_{h,t_0}(s)K_h\hat{K}I_{n,u}^{k,m}\Gamma_n^{t_0}dN_{n,u}(s),
\end{eqnarray*}
where we used in the last line that $F_{n,u}(s)=0$ when $t\leq t_0-h$. Now, the integrand is predictable and we can apply the expectation on both sides, to obtain a recursion formula: Use that for $x\geq0$ we have $f_{\rho}''(x)=\rho(\rho-1)f_{\rho-2}(x)$ to get
\begin{align*}
&\IE\left(\left|\sum_{u\in L_n}\int_0^tF_{n,u}^{k,m}(\tau)dM_{n,u}(\tau)\right|^{\rho}\Bigg|\mathcal{F}_{t_0-h}^n\right) \\
\leq&\sum_{u\in L_n}\int_{t_0-h}^t\frac{1}{2}\rho(\rho-1)\IE\Bigg(\left(\left|\sum_{r\in L_n}\int_0^{s-}F_{n,r}^{k,m}(\tau)dM_{n,r}(\tau)\right|+K_h\right)^{\rho-2} \\
&\quad\times K_{h,t_0}(s)K_h\hat{K}I_{n,u}^{k,m}\Gamma_n^{t_0}\lambda_{n,u}(s)\Bigg|\mathcal{F}_{t_0-h}^n\Bigg)ds \\
\leq&\int_{t_0-h}^t\frac{1}{2}\rho(\rho-1)K_h\hat{K}\Lambda\IE\Bigg(\left(\left|\sum_{r\in L_n}\int_0^{s-}F_{n,r}^{k,m}(s)dM_{n,r}(s)\right|+K_h\right)^{\rho-2} \\
&\quad\times\sum_{u\in L_n}K_{h,t_0}(s)C_{n,u}(s) I_{n,u}^{k,m}\Gamma_n^{t_0}\Bigg|\mathcal{F}_{t_0-h}^n\Bigg)ds.
\end{align*}
Define $Z^{k,m}(t)=\sum_{u\in L_n}\int_0^tF_{n,u}^{k,m}(\tau)dM_{n,u}(\tau)$ to summarize the previous inequality chain in the following recursion formula: For $\rho\geq2$ it holds almost surely
\begin{align}
&\IE\left(\left|Z^{k,m}(t)\right|^{\rho}\Big|\mathcal{F}_{t_0-h}^n\right) \nonumber \\
\leq& \frac{1}{2}\int_{t_0-h}^t\rho(\rho-1)K_h\hat{K}\Lambda K_{h,t_0}(s) \nonumber \\
&\quad\times\IE\left(\left(|Z^{k,m}(s-)|+K_h\right)^{\rho-2}\sum_{u\in L_n}I_{n,u}^{k,m}C_{n,u}(s)\Big|\mathcal{F}_{t_0-h}^n\right)\Gamma_n^{t_0}ds. \label{eq:recursion}
\end{align}
By uniting the (countably many) exception sets of measure zero, these inequalities hold for all $\rho\geq2$ and all $t\in[0,T]\cap\IQ$ on the same set of measure one. Since both sides are continuous from the right (cf. Corollary 5.1.9 in \citet{CE15}), we also have it for all $t\in[0,T]$ on the same set of measure one. Taking now limits from the left and repeating the same argument with continuity from the left, we obtain the same result for $Z^{k,m}(t-)$ on the left hand side also on the same set of measure one.

We are going to prove now via induction that almost surely (on the same set of measure one)
\begin{equation}
\label{eq:rec_cond}
\IE\left(|Z^{k,m}(t)|^{\rho}\Big|\mathcal{F}_{t_0-h}^n\right)\leq\frac{\rho!}{2}S_{k,m}\Lambda K_hA^{\rho}\left(S_{k}\sqrt{\frac{h\Lambda}{2}}K_h\right)^{\rho-2}\Gamma_n^{t_0}.
\end{equation}
We begin with the induction start: For $\rho=2$, \eqref{eq:recursion} gives for all $t\in[t_0-h,t_0+h]$
$$\IE\left(|Z^{k,m}(t)|^2\Big|\mathcal{F}_{t_0-h}^n\right)\leq K_h\hat{K}\Lambda S_{k,m}\cdot\Gamma_n^{t_0}\leq S_{k,m}\Lambda K_h A^2\cdot\Gamma_n^{t_0},$$
where the last inequality holds by choice of $A$ in \eqref{eq:chioceA} and because $t\in[t_0-h,t_0+h]$. Hence, the induction start is complete and we continue with the induction step. Assume that \eqref{eq:rec_cond} holds for all powers $2\leq p\leq\rho$ and all $t\in[t_0-h,t_0+h]$ and show that it holds for $\rho+1$ and all $t\in[t_0-h,t_0+h]$ as well. We use first \eqref{eq:recursion}, then the binomial theorem and finally the induction hypothesis \eqref{eq:rec_cond} for powers greater than one:
\begin{align}
&\IE\left(|Z^{k,m}(t)|^{\rho+1}\Big|\mathcal{F}_{t_0-h}^n\right) \nonumber \\
\leq&\frac{1}{2}\int_{t_0-h}^t(\rho+1)\rho K_h\hat{K}\Lambda K_{h,t_0}(s) \nonumber \\
&\quad\times\IE\left((|Z^{k,m}(s-)|+K_h)^{\rho-1}\sum_{u\in L_n}I_{n,u}^{k,m}C_{n,u}(s)\Big|\mathcal{F}^n_{t_0-h}\right)ds\cdot\Gamma_n^{t_0} \nonumber \\
\leq&\frac{1}{2}\int_{t_0-h}^t(\rho+1)\rho K_h\hat{K}\Lambda\sum_{p=0}^{\rho-1}\begin{pmatrix}\rho-1 \\ p \end{pmatrix}K_{h,t_0}(s) \nonumber \\
&\quad\times\IE\left(|Z^{k,m}(s-)|^{\rho-1-p}\sum_{u\in L_n}I_{n,u}^{k,m}C_{n,u}(s)\Big|\mathcal{F}^n_{t_0-h}\right)K_h^pds\cdot\Gamma_n^{t_0} \nonumber \\
\leq&\frac{K_h\hat{K}\Lambda}{2}(\rho+1)\rho \Bigg[\sum_{p=0}^{\rho-3}\frac{(\rho-1)!}{2p!}S_{k,m}\Lambda K_h\left(S_k\sqrt{\frac{h\Lambda}{2}}K_h\right)^{\rho-3-p}A^{\rho-1-p}K_h^pS_k \nonumber \\
&\quad\quad\quad+(\rho-1)\int_{t_0-h}^{t_0+h}K_{h,t_0}(s)\IE\left(|Z^{k,m}(s-)|\sum_{u\in L_n}I_{n,u}^{k,m}C_{n,u}(s)\Big|\mathcal{F}^n_{t_0-h}\right)K_h^{\rho-2}ds \nonumber \\
&\quad\quad\quad+K_h^{\rho-1}S_{k,m}\Bigg]\cdot\Gamma_n^{t_0} \label{eq:cfh}
\end{align}
Recall that $S_k=\max_{m=1,...,r_n}\sum_{u\in L_n}I_{n,u}^{k,m}\geq \sum_{u=1}I_{n,u}^{k,m}C_{n,u}(s)$ for all $k$ and $m$ as well as for all $s$, moreover $S_k$ is measurable with respect to $\mathcal{F}_{t_0-h}^n$. Hence, we may estimate
\begin{align*}
&\int_{t_0-h}^{t_0+h}K_{h,t_0}(s)\IE\left(\left|Z^{k,m}(s-)\right|\sum_{u\in L_n}I_{n,u}^{k,m}C_{n,u}(s)\Big|\mathcal{F}_{t_0-h}^n\right)ds \\
\leq&\int_{t_0-h}^{t_0+h}K_{h,t_0}(s)\sum_{u\in L_n}\IE\left(\int_{t_0-h}^{t_0+h}\frac{1}{h}K\left(\frac{\tau-t_0}{h}\right)I_{n,u}^{k,m}\hat{K}d|M_{n,u}|(\tau)\Big|\mathcal{F}_{t_0-h}^n\right)S_kds \\
=&\int_{t_0-h}^{t_0+h}K_{h,t_0}(s)\sum_{u\in L_n}\int_{t_0-h}^{t_0+h}\frac{1}{h}K\left(\frac{\tau-t_0}{h}\right)2\Lambda\hat{K}\IE\left(I_{n,u}^{k,m}C_{n,u}(\tau)\Big|\mathcal{F}_{t_0-h}^n\right)S_kds \\
=&2\Lambda\hat{K}S_{k,m}S_k.
\end{align*}
Using this estimation we continue with the main inequality chain
\begin{align*}
&\eqref{eq:cfh} \\
\leq&\frac{K_h\hat{K}\Lambda}{2}(\rho+1)\rho \Bigg[\sum_{p=0}^{\rho-3}\frac{(\rho-1)!}{2p!}S_{k,m}\Lambda K_h\left(S_k\sqrt{\frac{h\Lambda}{2}}K_h\right)^{\rho-3-p}A^{\rho-1-p}K_h^pS_k  \\
&\quad\quad\quad+(\rho-1)K_h^{\rho-2}2\Lambda\hat{K}S_{k,m}S_k \\
&\quad\quad\quad+K_h^{\rho-1}S_{k,m}\Bigg]\cdot\Gamma_n^{t_0} \\
=&\frac{(\rho+1)!}{2}S_{k,m}\Lambda K_hA^{\rho+1}\left(S_k\sqrt{\frac{h\Lambda}{2}}K_h\right)^{\rho-1}\Gamma_n^{t_0} \\
&\quad\times\frac{1}{A}\cdot\Bigg[\sum_{p=0}^{\rho-3}\frac{1}{2p!}\cdot\Lambda K_h\hat{K}\left(S_k\sqrt{\frac{h\Lambda}{2}}K_h\right)^{-2-p}A^{-1-p}K_h^pS_k \\
&\quad\quad\quad+\frac{1}{(\rho-2)!}K_h^{\rho-2}2\Lambda\hat{K}^2S_kA^{-\rho}\left(S_k\sqrt{\frac{h\Lambda}{2}}K_h\right)^{-\rho+1} \\
&\quad\quad\quad+\frac{1}{(\rho-1)!}K_h^{\rho-1}\hat{K}A^{-\rho}\left(S_k\sqrt{\frac{h\Lambda}{2}}K_h\right)^{-\rho+1}\Bigg]
\end{align*}
At this point, we see that we're obviously done with the induction step if $\Gamma_n^{t_0}=0$. Hence, we only need to show that the above is lesser than or equal to \eqref{eq:rec_cond} on the event $\Gamma_n^{t_0}=1$. This, in turn, we may conclude if the second part above is smaller than or equal to one (on the event $\Gamma_n^{t_0}=1$). This is the case because we have chosen $A$ appropriately and because $h\leq1$ and $S_k\sqrt{h}\geq1$ (and thus also $S_k\geq1$) on $\Gamma_n^{t_0}$:
\begin{align*}
&\frac{1}{A}\cdot\Bigg[\sum_{p=0}^{\rho-3}\frac{1}{2p!}\cdot\Lambda K_h\hat{K}\left(S_k\sqrt{\frac{h\Lambda}{2}}K_h\right)^{-2-p}A^{-1-p}K_h^pS_k \\
&\quad\quad\quad+\frac{1}{(\rho-2)!}K_h^{\rho-2}2\Lambda\hat{K}^2S_kA^{-\rho}\left(S_k\sqrt{\frac{h\Lambda}{2}}K_h\right)^{-\rho+1} \\
&\quad\quad\quad+\frac{1}{(\rho-1)!}K_h^{\rho-1}\hat{K}A^{-\rho}\left(S_k\sqrt{\frac{h\Lambda}{2}}K_h\right)^{-\rho+1}\Bigg] \\
=&\frac{1}{A}\cdot\Bigg[\sum_{p=0}^{\rho-3}\frac{1}{p!}\left(S_k\sqrt{\frac{h\Lambda}{2}}A\right)^{-p}\cdot \frac{1}{S_kK A} \\
&\quad\quad\quad+\frac{1}{(\rho-2)!}\left(S_k\sqrt{\frac{h\Lambda}{2}}A\right)^{-\rho+2} \frac{2^{\frac{3}{2}}}{KA^2}\sqrt{h\Lambda}\hat{K} \\
&\quad\quad\quad+\frac{1}{(\rho-1)!}\left(S_k\sqrt{\frac{h\Lambda}{2}}A\right)^{-\rho+1}\hat{K}A^{-1}\Bigg] \\
\leq&\frac{1}{A}\sum_{p=0}^{\infty}\frac{1}{p!}\left(S_k\sqrt{\frac{h\Lambda}{2}}A\right)^{-p} \\
=&\frac{1}{A}\exp\left(\frac{\sqrt{2}}{AS_k\sqrt{h\Lambda}}\right)\leq\frac{1}{A}\exp\left(\frac{\sqrt{2}}{A\sqrt{\Lambda}}\right)\leq1
\end{align*}
and the induction is complete. To finalize the proof, we compute the expectation of $S_{k,m}S_k^{\rho-2}$. Note that on $\Gamma_n^{t_0}=1$, $S_k\leq c_3\cdot\sqrt{\frac{r_n\bar{p}_n(t_0)}{\log r_n\bar{p}_n(t_0)}}=E_k^{n,t_0}$
$$\IE\left(S_{k,m}S_k^{\rho-2}\Gamma_n^{t_0}\right)\leq\IE(S_{k,m})\cdot\left(E_k^{n,t_0}\right)^{\rho-2}\leq E_{k,m}^{n,t_0}\cdot\left(E_k^{n,t_0}\right)^{\rho-2}.$$
Taking expectations on both sides of \eqref{eq:rec_cond} and together with the previous  line, we obtain
$$\IE\left(|Z^{k,m}(T)|^{\rho}\right)\leq\frac{\rho!}{2}E_{k,m}^{n,t_0}\tilde{K}K_hA^{\rho}\left(E_k^{n,t_0}\sqrt{\frac{h\Lambda}{2}}K_h\right)^{\rho-2}.$$
Hence, condition 1 of Lemma \ref{lem:network_exp} is fulfilled and we can apply it to get
\begin{eqnarray*}
&&\IP\left(\frac{1}{r_n\bar{p}_n(t_0)}\sum_{u\in L_n}\int_0^TK\left(\frac{s-t_0}{h}\right)X_{n,u}(s)dM_{n,u}(s)\geq x\cdot\sqrt{\frac{\log r_n\bar{p}_n(t_0)}{r_n\bar{p}_n(t_0)}}\right) \\
&\leq&\IP\left(\frac{1}{r_n\bar{p}_n(t_0)}\sum_{u\in L_n}\int_0^TK\left(\frac{s-t_0}{h}\right)X_{n,u}(s)dM_{n,u}(s)\geq x\cdot\sqrt{\frac{\log r_n\bar{p}_n(t_0)}{r_n\bar{p}_n(t_0)}},\Gamma_n^{t_0}=1\right) \\
&&\quad\quad+\IP(\Gamma_n^{t_0}=0) \\
&\leq&\IP\left(\frac{1}{r_n\bar{p}_n(t_0)}\sum_{u\in L_n}Z_{n,u}(T)\geq x\cdot\sqrt{\frac{\log r_n\bar{p}_n(t_0)}{r_n\bar{p}_n(t_0)}}\right)+\IP(\Gamma_n^{t_0}=0) \\
&\leq&\Kappa(r_n\bar{p}_n(t_0))^{-\frac{c_2\cdot x^2}{2(\sigma^2+c_1c_3x)}}+\beta_t(\Delta_n)\cdot\Kappa r_n+\IP(\Gamma_n^{t_0}=0).
\end{eqnarray*}
\end{proof}

\begin{lemma}
\label{lem:OmegasGood}
Let $\Omega_n(s)$ and $\omega_n(s)$ be defined as in the proof of Lemma \ref{lem_ass:var}. For any $\alpha>0$ there is $\xi>1$ and $C>0$ such that
\begin{align*}
\IP\left(\sup_{s\in[0,T]}\Omega_n(s)>\sqrt{h}\xi\right)<&C(r_np_n)^{-\alpha}\textrm{ and }\IP(\omega_n(s)>\xi)<C\left(r_np_n\right)^{-\alpha}.
\end{align*}
\end{lemma}
\begin{proof}
The proof follows standard arguments. Let $T_{n,k}$ denote a discrete grid of $[0,T]$ with $O(n^k)$ many elements such that for any $t,s\in T_{n,k}$ with $t\neq s$ it holds that $|t-s|<n^{-k}$. Then,
\begin{align}
&\IP(\sup_{s\in[0,T]}\Omega_n(s)>\sqrt{h}\xi) \nonumber \\
\leq&\IP\left(\underset{|s-t|\leq n^{-k}}{\sup_{s,t\in[0,T],}}\Omega_n(s)-\Omega_n(t)>\sqrt{h}\right)+\sup_{t\in T_{n,k}}\#T_{n,k}\cdot\IP\left(\Omega_n(t)>\sqrt{h}(\xi-1)\right). \label{eq:ab1}
\end{align}
For the first probability we note that for $|t-s|<n^{-k}$ the intervals $[t-2h,t)$ and $[s-2h,s)$ are overlapping on a length of at most $2h$ and the area covered by only one interval is of length at most $2n^{-k}$. We get for $t<s$ when using $d|M_{n,u}|(r)=dN_{n,u}(r)+\lambda_{n,u}(r)dr$ (recall that $\mathcal{C}_{n,u}$ is bounded by $H_n$)
\begin{align}
&\Omega_n(t)-\Omega_n(s) \nonumber \\
=&\sum_{u\in L_n}\int_0^T\mathcal{C}_{n,u}\left(\frac{\Ind(r\in[t-2h,t))}{r_np_n(t)}-\frac{\Ind(r\in[s-2h,s))}{r_np_n(s)}\right)d|M_{n,u}|(r) \nonumber \\
\leq&\sum_{u\in L_n}H_n\left(\frac{N_{n,u}([t-2h,s-2h])}{r_np_n}+\frac{N_{n,u}([t,s])}{r_np_n}+N_{n,u}([s-2h,t])\frac{|p_n(t)-p_n(s)|}{r_np_n^2}\right) \nonumber \\
&+2n^{-k}\frac{\Lambda H_n}{p_n}+2h H_n\frac{|p_n(t)-p_n(s)|}{p_n^2}. \nonumber
\end{align}
The last line is deterministic and converges faster to zero than $\sqrt{h}$ by the Hoelder continuity of $p_n(t)$ (cf. Assumption (A6)) and since $H_n$ grows moderately (cf. Assumption (H2, \eqref{eq:HubS6})). For the expressions in the first line, we note that in the end it comes down to evaluating expressions of the type $\sup_{|t-s|<\epsilon_n}\sum_{u\in L_n}N_{n,u}([s,t])$ where $\epsilon_n$ equals either $n^{-k}$ or $2h+n^{-k}$. In Assumption (AD, \ref{eq:AD1}) we assume that in both cases the average behaves in such a way that the first probability in \eqref{eq:ab1} converges to zero as fast as required if $k$ is chosen large enough. We keep this choice of $k$ fixed for the remainder of the proof.

For the second part of \eqref{eq:ab1}, we rewrite
\begin{equation}
\label{eq:Omega} \Omega_n(s)=\frac{1}{r_np_n(t)}\sum_{u\in L_n}\int_{t-2h}^{t-}\mathcal{C}_{n,u}dM_{n,u}(r)+\frac{2}{r_np_n(t)}\sum_{u\in L_n}\int_{t-2h}^{t-}\mathcal{C}_{n,u}\lambda_{n,u}(r)dr.
\end{equation}
For both parts we have exponential inequalities available in Lemmas \ref{lem:exponential_inequality} and \ref{lem:sufficient_mixing}, respectively. So we just have to check that their conditions hold. Since $\mathcal{C}_{n,u}$ is bounded by $H_n^2$, we can divide by the bound and apply Lemma \ref{lem:exponential_inequality} with $X_{n,u}(s)=1$ and $K_{h,t_0}(s)=\frac{1}{2h}\Ind(s\in[t-2h,t))$. Note firstly that by Assumption (A6) we can replace the $p_n(t)$ by $\bar{p}_n(t)$ when adding a multiplicative constant which we can compensate for by choosing $\xi$ appropriately. Moreover, by Assumption (D3) there are $\Delta_n$-partitions as required and the $\beta$-mixing coefficients decay exponentially fast. Since $\Delta_n=a\log n$, the mixing coefficients decay as fast as required. Moreover, by the same assumption, $\sup_{t\in[0,T]}\IP(\Gamma_n^{t}=0)$ vanishes exponentially fast. Finally, by Assumption (H4, \ref{eq:HubS6}), the bound $H_n$ on $K_m^{L_n}$ behaves exactly such that also the leading term decays as fast we want if $\xi$ is chosen large enough. Therefore, the probability that the first part of \eqref{eq:Omega} is larger than $\frac{\xi-1}{2}\sqrt{h}$ decreases to zero faster than any given power of $r_np_n$ for large enough $\xi$.

The second term in \eqref{eq:Omega} can be bounded by analogous arguments and Lemma \ref{lem:sufficient_mixing}. Denote $Y_{n,u}=\frac{1}{h}\int_{t-2h}^{t-}\mathcal{C}_{n,u}\lambda_{n,u}(r)dr$. Then $\IE(Y_{n,u})/p_n(t)\leq c^*$ by Assumption (H2, \ref{eq:HubS7}) (note that $c^*$ is independent of $t$ and $n$). Keeping this in mind we obtain for the second term in \eqref{eq:Omega} for small enough $h$ for all $n$ and $t$
\begin{align*}
&\IP\left(\frac{2h}{r_np_n(t)}\sum_{u\in L_n}Y_{n,u}>\sqrt{h}\xi\right) \\
\leq&\IP\left(\frac{1}{r_np_n(t)}\sum_{u\in L_n}(Y_{n,u}-\IE(Y_{n,u}))>\frac{\xi}{2\sqrt{h}}-\frac{\IE(Y_{n,u})}{p_n(t)}\right) \\
\leq&\IP\left(\frac{1}{r_np_n(t)}\sum_{u\in L_n}(Y_{n,u}-\IE(Y_{n,u}))>\frac{\xi-1}{2\sqrt{h}}\right).
\end{align*}
Choose $E_{k,m}^{n,t}$ in the same way as in Lemma \ref{lem:exponential_inequality} with $K_{h,t}(s)=\frac{1}{2h}\Ind(s\in[t-2h,t))$. Then $|E|_{n,t}=r_n\bar{p}_n(t)$ and also $E_{k}^{n,t}$ is as defined in Lemma \ref{lem:exponential_inequality}. Therefore all restrictions on the $\Delta_n$-partitioning are fulfilled by Assumption (D3) and the mixing coefficients vanish exponentially fast. Moreover $\IP(\Gamma_n^t=0)$ vanishes exponentially fast. Lastly, by Assumption (D3, \ref{eq:UUM4}), the asymptotic uncorrelation conditions hold. Hence, we may apply Lemma \ref{lem:sufficient_mixing} and obtain the desired results by the same arguments as for the first part of \eqref{eq:Omega} by using again Assumption (H2, \ref{eq:HubS6}).

The proof of the concentration inequality for $\omega_n(s)$ follows from similar arguments.
\end{proof}

\subsection{Details for Example \ref{subsubsec:mixing}}
\label{subsec:addexp}
Let $(D_1,D_2)$ and $(\tilde{D}_1,\tilde{D}_2)$ be two pairs of random variables with
$$(D_1,D_2)\sim\mathcal{N}\left(0,\begin{pmatrix}
\Sigma_1 & \sigma \\ \sigma' & \Sigma_2
\end{pmatrix}\right),\quad (\tilde{D}_1,\tilde{D}_2)\sim\mathcal{N}\left(0,\begin{pmatrix}
\Sigma_1 & 0 \\ 0 & \Sigma_2
\end{pmatrix}\right).$$
We suppose that $D_1,D_2,\tilde{D}_1,\tilde{D}_2\in\IR^p$ all have the same dimension $p$. The matrix $\sigma\in\IR^{p\times p}$ contains the covariances of $D_1$ and $D_2$. Let $|.|$ denote the determinant of a matrix and $I_p$ is the $p\times p$ identity matrix. We can compute that (use formulas for the Kullback-Leibler divergence of two multivariate normals and for the determinant of block-matrices)
\begin{align*}
&KL((D_1,D_2),(\tilde{D}_1,\tilde{D}_2)) \\
=&\frac{1}{2}\left(tr\left(\begin{pmatrix}
\Sigma_1^{-1} & 0 \\ 0 & \Sigma_2^{-1}
\end{pmatrix}\begin{pmatrix}
\Sigma_1 & \sigma \\ \sigma' & \Sigma_2
\end{pmatrix}\right)-2p+\log\frac{|\Sigma_1|\cdot|\Sigma_2|}{|\Sigma_1|\cdot|\Sigma_2-\sigma'\Sigma^{-1}_1\sigma|}\right) \\
=&\frac{1}{2}\log\left|\left(I_p-\Sigma_2^{-1}\sigma'\Sigma_1^{-1}\sigma\right)^{-1}\right|.
\end{align*}
Suppose that the entries of $\sigma$ are small and that $\Sigma_1$ and $\Sigma_2$ are positive definite. In that case $\Sigma_2^{-1}\sigma'\Sigma_1^{-1}\sigma$ has small eigenvalues and $I_p-\Sigma_2^{-1}\sigma'\Sigma_1^{-1}\sigma$ is positive definite. With this we may continue the estimation by applying the bound $\log|A|\leq\textrm{tr}(A-I)$ and using the Neumann Series representation:
\begin{align*}
\leq&\frac{1}{2}\textrm{tr}\left(\left(I_p-\Sigma_2^{-1}\sigma'\Sigma_1^{-1}\sigma\right)^{-1}-I_p\right) \\
=&\frac{1}{2}\textrm{tr}\left(\sum_{k=1}^{\infty}\left(\Sigma_2^{-1}\sigma'\Sigma_1^{-1}\sigma\right)^k\right).
\end{align*}
Let every entry of $\sigma$ be bounded in absolute value by $\epsilon$. Then, there is $c>0$ such that each entry of $\Sigma_2^{-1}\sigma'\Sigma_1^{-1}\sigma$ is in absolute value bounded by $cp^3\epsilon^2\leq cp^4\epsilon^2$. And by induction each entry of $\left(\Sigma_2^{-1}\sigma'\Sigma_1^{-1}\sigma\right)^k$ is bounded by $c^k\epsilon^{2k}p^{4k}$. Then, we continue
\begin{align*}
\leq&\frac{1}{2}p\sum_{k=1}^{\infty}\left(c\epsilon^2p^4\right)^k=\frac{p}{2}\frac{c\epsilon^2p^4}{1-c\epsilon^2p^4}.
\end{align*}
We have shown in \eqref{eq:expdecay} that $\epsilon=c^*\sqrt{6\alpha_0}^\Delta$. Moreover, $p=|I_1|=|I_2|=2(\Delta-1)^2$. Recalling in addition that $6\alpha_0<1$, we conclude overall that $c\epsilon^2p^4\to0$ exponentially fast for $\Delta\to\infty$.

\subsection{Useful Results}
For the convenience of the reader, we collect some result which are needed in the proofs.

We will consider the Grouping Lemma in the following form (\citet{R17}, Lemma 5.1 therein).
\begin{lemma}
\label{lem:grouping_lemma}
Let $\mathcal{A}$ be a $\sigma$-field in $(\Omega,\mathcal{F},\IP)$ and let $X$ be a random variable with values in a Polish space $\mathcal{X}$. Let $\delta$ be a random variable with uniform distribution over $[0,1]$ which is independent of the $\sigma$-field generated by $\mathcal{A}$ and $X$. Then, there exists a random variable $X^*$ which has the same law as $X$ and which is independent of $\mathcal{A}$, such that $\IP(X\neq X^*)=\beta(\mathcal{A},\sigma(X))$. Furthermore, $X^*$ is measurable with respect to the $\sigma$-field generated by $\mathcal{A}$ and $(X,\delta)$.
\end{lemma}

The Bernstein Inequality will be used in this form (cf. \citet{GN16}).
\begin{proposition}
\label{prop:bernstein}
Let $X_i$, $i=1,...,n$ be a sequence of independent, centred random variables such that there are numbers $c$ and $\sigma_i$ such that for all $k$ $\IE(|X_i|^k|)\leq\frac{k!}{2}\sigma_i^2c^{k-2}$. Set $\sigma^2:=\sum_{i=1}^n\sigma_i^2$, $S_n:=\sum_{i=1}^nX_i$. Then, for all $t\geq0$ $\IP(S_n\geq t)\leq\exp\left(-\frac{t^2}{2(\sigma^2+ct)}\right)$.
\end{proposition}

Lenglart's Inequality shows how a martingale may be controlled by using the quadratic variation. We state in the following a slight adaptation of the original version as it is provided in  \citet{L77}.

\begin{lemma}
\label{lem:Lenglart}
Let $X$ be a non-negative, right-continuous local sub-martingale and denote by $A$ its compensator. Then it holds for all finite stopping times $S>0$ and all $c,d>0$ that
$$\IP\left(\sup_{t\in[0,S]}X_t\geq c\right)\leq\frac{1}{c}\IE\left(A_S \wedge d\right)+\IP\left(A_S\geq d\right)\leq\frac{d}{c}+\IP(A_S\geq d).$$
\end{lemma}

In this paper, we will apply Lenglart's Inequality mostly in the following form which is close to \citet{ABGK93}. The following is an easy corollary to the previous lemma.

\begin{corollary}
\label{cor:Lenglart}
Let $M$ be a locally square integrable, right-continuous martingale and denote by $\langle M\rangle$ it's compensator.
\begin{enumerate}
\item For all $T,c,d>0$ we have
$$\IP\left(\sup_{t\in[0,T]}|M_t|\geq c\right)\leq\frac{d}{c^2}+\IP\left(\langle M\rangle_T\geq d\right).$$
\item For all $T>0$ it is true that
$$\langle M\rangle_T\overset{\IP}{\rightarrow}0\,\implies\,\sup_{t\in[0,T]}|M_t|\overset{\IP}{\rightarrow}0.$$
\end{enumerate}
\end{corollary}

The main tool for finding the asymptotic distributions in this paper is Rebolledo's Martingale Central Limit Theorem. It is known that a Brownian Motion is the only continuous Gaussian process with a certain covariance structure. This is used to formulate a martingale central limit theorem in the following. We state here the version of the theorem as Theorem II.5.1 in \citet{ABGK93}, the original work is \citet{R80}.
 
Let $M^n=(M_1^n,...,M_k^n)$ be a vector of sequences of locally square integrable martingales on an interval $\mathcal{T}$. For $\epsilon>0$ we denote by $M_{\epsilon}^n$ a vector of locally square integrable martingales that contain all jumps of components of $M^n$ which are larger in absolute value than $\epsilon$, i.e., $M_i^n-M_{\epsilon,i}^n$ is a local square integrable martingale for all $i=1,...,k$ and $|\Delta M_i^n-\Delta M_{\epsilon,i}^n|\leq\epsilon$. Furthermore, we denote by $\langle M^n\rangle:=\left(\langle M_i^n,M_j^n\rangle\right)_{i,j=1,...,k}$ the $k\times k$ matrix of quadratic covariations. Moreover, we denote by $M$ a multivariate, continuous Gaussian martingale with $\langle M\rangle_t=V_t$, where $V:\mathcal{T}\to\IR^{k\times k}$ is a continuous deterministic $k\times k$ positive semi-definite matrix valued function on $\mathcal{T}$ such that its increments $V_t-V_s$ are also positive semi-definite for $s\leq t$, then $M_t-M_s\sim \mathcal{N}(0,V_t-V_s)$ is independent of $(M_r:\,r\leq s)$. Given such a function $V$, such a Gaussian process $M$ always exists. We can now formulate the central limit theorem for martingales.

\begin{theorem}
\label{thm:Rebolledo}
Let $\mathcal{T}_0\subseteq\mathcal{T}$. Assume that for all $t\in\mathcal{T}_0$ as $n\to\infty$
\begin{align*}
&\langle M^n\rangle_t\overset{\IP}{\rightarrow} V_t \\
&\langle M_{\epsilon}^n\rangle\overset{\IP}{\rightarrow}0.
\end{align*}
Then
$$M^n_t\overset{d}{\rightarrow} M_t$$
as $n\to\infty$ for all $t\in\mathcal{T}_0$.
\end{theorem}
We remark that the predictable quadratic variation may be replaced by the optional quadratic variation. But we do not use that in this paper.

Finally, the theorem by Kantorovich gives a relation between the solution of an equation system and its derivative at the solution (see e.g. \citet{D85}):
\begin{theorem}
\label{thm:kantorovich} (Newton-Kantorovich Theorem)
Let $R(x)=0$ be a system of equations where $R:D_0\subseteq\IR^p\to\IR$ is a function defined on $D_0$. Let $R$ be differentiable and denote by $R'$ its first derivative. Assume that there is  an $x_0$ such that all expressions in the following statements exist and such that the following statements are true
\begin{enumerate}
\item $||R'(x_0)^{-1}||\leq B$,
\item $||R'(x_0)^{-1}R(x_0)||\leq\eta$,
\item $||R'(x)-R'(y)||\leq K||x-y||$ for all $x,y\in D_0$,
\item $r:=BK\eta\leq\frac{1}{2}$ and $\Omega_*:=\{x:||x-x_0||<2\eta\}\subseteq D_0$.
\end{enumerate}
Then there is $x^*\in\Omega_*$ with $R(x^*)=0$ and
$$||x^*-x_0||\leq2\eta\textrm{ and }||x^*-(x_0-R'(x_0)^{-1}R(x_0))||\leq2r\eta.$$
\end{theorem}

We also use Itô's Formula for semi-martingales with jumps (Theorem 14.2.4 in \citet{CE15}). Here $X$ is to be understood as the cadlag modification of $X$ (cf. Corollary 5.1.9 in \citet{CE15}).
\begin{theorem}
\label{thm:Ito}
Let $X$ be a $n$-dimensional vector of semi-martingales $X=(X^1,...,X^n)$ and let $f:\IR^n\to\IR$ be a twice continuously differentiable function. Then,
\begin{align*}
f(X_t)&=f(X_0)+\sum_{i=1}^n\int_{(0,t]}\partial_{i}f(X_{s-})dX^i_s+\frac{1}{2}\sum_{i,j=1}^n\int_{(0,t]}\partial_{ij}f(X_{s-})d[X^i,X^j]_s \\
&+\sum_{0<s\leq t}\left(f(X_s)-f(X_{s-})-\sum_{i=1}^n\partial_if(X_{s-})\Delta X^i_s-\frac{1}{2}\sum_{i,j=1}^n\partial_{ij}f(X_{s-})\Delta X^i_s\Delta X^j_s\right).
\end{align*}
The above equality means that the processes to the left and to the right are indistinguishable and $[X^i,X^j]$ denotes the optional covariation of $X^i$ and $X^j$ (see below).
\end{theorem}
The optional quadratic variation of a cadlag square integrable local martingale $M$ is given by $[M]_t:=M_t^2-\int_0^tM_{s-}dM_s$. The optional covariation for two such martingales $M$ and $N$ is given by $[M,N]:=\frac{1}{2}([M+n]-[M]-[N])$.

\newpage

\textbf{Acknowledgement} This work is part of my PhD Thesis which I have written under the supervision of Prof. Dr. Enno Mammen (Heidelberg University) and Prof. Dr. Wolfgang Polonik (UC Davis). I am thankful for many discussions and helpful remarks. My work was supported by Deutsche Forschungsgemeinschaft through the Research Training Group RTG 1953.

I am also very thankful for the comments of the associate editor and two referees, their suggestions significantly contributed to a great improvement the paper.

\bibliography{mybib}{}
\bibliographystyle{plainnat}

\end{document}